\documentclass[10pt,leqno]{article}
\usepackage{amsmath,amssymb,amsfonts,amsthm}
\usepackage{authblk,hyperref,enumerate,xcolor}
\newtheorem{theorem}{Theorem}[section]
\newtheorem{proposition}[theorem]{Proposition}
\newtheorem{corollary}[theorem]{Corollary}
\newtheorem{lemma}[theorem]{Lemma}
\newtheorem{definition}[theorem]{Definition}
\newtheorem{remark}[theorem]{Remark}
\newtheorem{notation}[theorem]{Notation}
\input bussproofs.sty
\EnableBpAbbreviations
\let\seq\boldsymbol		% Typeface for finite sequences
\let\lang\mathcal		% Typeface for languages
\let\classic\mathfrak	% Typeface for classical models
\let\kripke\mathbf		% Typeface for Kripke models
\let\thy\mathsf			% Typeface for theories
\let\frm\mathbf			% Typeface for special formulas
\let\ax\mathsf			% Typeface for axioms and schemas
\let\clsfm\mathsf		% Typeface for classes of formulas
\let\ind\mathrm			% Typeface for indicators
\let\ra\rightarrow
\let\lr\leftrightarrow
\let\Ra\Rightarrow
\let\Lr\Leftrightarrow
\let\all\forall
\let\ex\exists
\def\vdv{\mathrel{\dashv\vdash}}
\def\x{\seq x}
\def\y{\seq y}
\def\z{\seq z}
\def\a{\seq a}
\def\b{\seq b}
\def\cc{\seq c}
\def\t{\seq t}
\def\LL{\lang L}
\def\Lc{\LL_{\ind c}}
\def\km{\kripke K}
\def\kzp{\km_{0+}}
\def\K{\km^*}
\def\kb{\km_\bullet}
\def\kc{\km_\circ}
\let\bk\bullet
\let\ck\circ
\def\M{\classic M_k}
\def\Mp{\classic M_{k'}}
\def\Mc{\classic M_\ck}
\def\nn{\mathbb N}

\def\Rzp{\mathbb R_{0+}}
\def\N{\nn^*}
\let\8\infty
\def\9{\boldsymbol\8}
\def\Open{\clsfm{Open}}
\def\Ex{\ex_1^+}
\def\Ep{\ex^+}
\def\Eo{\ex_1}
\def\Dz{\clsfm\Delta_0}
\def\Sigz{\clsfm\Sigma_0}
\def\Sig{\clsfm\Sigma_1}

\def\Sign{\clsfm\Sigma_n}
\def\Signp{\clsfm\Sigma_{n+1}}
\def\Piz{\clsfm\Pi_0}
\def\Pio{\clsfm\Pi_1}
\def\Pit{\clsfm\Pi_2}
\def\Pin{\clsfm\Pi_n}
\def\Pinp{\clsfm\Pi_{n+1}}
\def\U{\ax U}
\def\EXP{\ax{EXP}}
\def\MRDP{\ax{MRDP}}
\def\MRDPw{\ax{MRDP}^{\ind w}}
\DeclareFontFamily{U}{MnSymbolC}{}
\DeclareFontShape{U}{MnSymbolC}{m}{n}
{<-6>	MnSymbolC5
<6-7>	MnSymbolC6
<7-8>	MnSymbolC7
<8-9>	MnSymbolC8
<9-10>	MnSymbolC9
<10-12>	MnSymbolC10
<12->	MnSymbolC12}{}
\DeclareSymbolFont{MnSyC}{U}{MnSymbolC}{m}{n}
\DeclareMathSymbol{\dotminus}{\mathbin}{MnSyC}{24}
\newcommand{\sqt}[1]{\boldsymbol{\mathcal #1}}	% Special sequent
\newcommand{\I}[1]{{\thy I}#1}					% Fragment of arithmetic
\newcommand{\cmt}[1]{#1^{\ind c}}				% Complement
\newcommand{\gp}[1]{#1^\ex}						% Geometric part
\newcommand{\sgp}[1]{#1^\all}					% Semi-geometric part
\newcommand{\gps}[1]{#1^+}						% Geometric equivalent
\newcommand{\gng}[1]{#1^-}						% Geometric negation
\newcommand{\trm}[2]{#1^{=#2}}					% Defined fnc. symb. elim.
\newcommand{\dfe}[1]{#1^*}						% Defined fnc. symb. elim.
\newcommand{\gt}[1]{#1^{\ind g}}				% Goedel's negative traslation
\newcommand{\dzn}[1]{#1^\neg}					% Bounded negation
\newcommand{\trns}[1]{#1^{\ind t}}				% Translation EBA to ISigma_n
\def\GQC{\thy{GQC}}
\def\BQC{\thy{BQC}}
\def\EBQC{\thy{EBQC}}
\def\IQC{\thy{IQC}}
\def\CQC{\thy{CQC}}
\def\LK{\thy{LK}}
\def\Q{\thy Q}
\def\GA{\thy{GA}}
\def\GAw{\thy{GA^{\ind w}}}
\def\BA{\thy{BA}}
\def\BAc{\thy{BA_{\ind c}}}
\def\BAw{\thy{BA^{\ind w}}}
\def\EBA{\thy{EBA}}
\def\EBD{\thy{EB}\Dz}
\def\HA{\thy{HA}}
\def\PA{\thy{PA}}
\def\PAm{\thy{PA}^-}
\def\Iop{\I\Open}
\def\ID{\I\Dz}
\def\IEx{\I\Ex}
\def\IEp{\I\Ep}
\def\IE{\I\Eo}
\def\ISig{\I\Sig}

\def\ISign{\I\Sign}
\def\ISignp{\I\Signp}
\def\IC{\I\C}
\def\T{\thy T}
\DeclareMathOperator{\pd}{pd}			% Predecessor function
\DeclareMathOperator{\range}{range}		% Range of a function
\DeclareMathOperator{\PTF}{PTF}			% Provably total functions
\DeclareMathOperator{\PTRF}{PTRF}		% Provably total recursive functions
\DeclareMathOperator{\Th}{Th}			% Theory of a structure/system
\DeclareMathOperator{\Even}{Even}		% Evenness function
\DeclareMathOperator{\es}{\sqt E}		% Existence sequent
\DeclareMathOperator{\us}{\sqt U}		% Uniqueness sequent
\DeclareMathOperator{\Prov}{\frm{Prov}}	% Proof predicate
\DeclareMathOperator{\Con}{\frm{Con}}	% Consistency sentence
\def\C{\mathcal C}					% Class of formulas
\def\A{\mathcal A}					% Set of natural numbers
\def\D{\mathfrak D}					% Formal derivation
\def\PR{\boldsymbol{\mathcal{PR}}}	% The set of all PR functions
\def\PRec{\frm{PR}}					% Formula indicating codes of PR functions
\def\r{\mathbf q^{\PR}}				% PR realizability
\def\rd{\mathbf q^D}				% D-bounded recursive realizability
\def\B{\frm B_D}					% Formula indicating boundedness by D
\def\e{{\frm A}_{\exp}}				% Formula defining exponentiation
\def\p{{\frm A}_{\pd}}				% Formula defining exponentiation
\def\co{{\frm A}_{\ind c}}			% Formula defining cut-off subtraction

\begin{document}
\title{The provably total recursive functions and the MRDP theorem
in Basic Arithmetic and its extensions}
\author[1,*]{Mohammad Ardeshir}
\author[2,3,$\dagger$]{Erfan Khaniki}
\author[1,$\ddagger$]{Mohsen Shahriari}
\affil[1]{Department of Mathematical Sciences\\
Sharif University of Technology}
\affil[2]{Faculty of Mathematics and Physics\\
Charles University}
\affil[3]{Institute of Mathematics\\
Czech Academy of Sciences}
{\makeatletter
\renewcommand\AB@affilsepx{: \protect\Affilfont}
\makeatother
\affil[ ]{Email IDs}
\makeatletter
\renewcommand\AB@affilsepx{, \protect\Affilfont}
\makeatother
\affil[*]{\texttt{mardeshir@sharif.edu}}
\affil[$\dagger$]{\texttt{khaniki@math.cas.cz}}
\affil[$\ddagger$]{\texttt{m.shahriari@sharif.edu}}}
\maketitle

\begin{abstract}
We study Basic Arithmetic, $\BA$ introduced by W. Ruitenburg. $\BA$ is an
arithmetical theory based on basic logic which is weaker than intuitionistic
logic. We show that the class of the provably total recursive functions of
$\BA$ is a \emph{proper} sub-class of the primitive recursive functions. Three
extensions of $\BA$, called $\BA+\U$, $\BAc$ and $\EBA$ are investigated
with relation to their provably total recursive functions. It is shown that
the provably total recursive functions of these three extensions of $\BA$ are
\emph{exactly} the primitive recursive functions. Moreover, among other things,
it is shown that the well-known MRDP theorem does not hold in $\BA$,
$\BA+\U$, $\BAc$, but holds in $\EBA$.
\end{abstract}

\section{Introduction}\label{intro}
Basic Arithmetic, $\BA$ is an arithmetical theory introduced by W. Ruitenburg
in \cite{R98}, based on Basic Predicate Calculus, $\BQC$, as Heyting
Arithmetic, $\HA$ is based on Intuitionistic Predicate Calculus, $\IQC$ and
Peano Arithmetic, $\PA$ based on Classical Predicate Calculus, $\CQC$. 
$\BQC$ is a weaker logic than $\IQC$, in which the rule of \emph{Modus Ponens}
is weakened. It was motivated by a revision of the Brouwer-Heyting-Kolmogorov
interpretation (see \cite{R91} and \cite{R12}). Formally, basic logic is
an extension of the well-known \emph{geometric logic}\footnote {In the
literature, the term ``geometric logic" is usually used for a logic with
infinitary disjunction, and the formulas with only finitary disjunction are
called \emph{coherent} geometric formulas. In this paper, we only work with
finitary languages, and only nullary and binary disjunction ($\bot$ and $\lor$,
respectively) are included in the language, which are sufficient for making up
all the finitary disjunctions. Also, derivability in geometric logic is often
considered in a \emph{context} of a finite sequence of variables, which for
example allows for an empty domain of discourse when considering semantics.
We will use the term ``geometric" only for reference to the finitary part of
the language and logic, and consider derivability \emph{without} variable
contexts.} (see \cite{V}) by adding implication and universal quantification
to the language, in a way that they reflect the meaning of ``derivability"
at the object language level. We will use this close relation between
basic logic and geometric logic to obtain some of our results about $\BA$
and its extensions. We will also briefly discuss some properties of theories
formalized over basic logic that are absent in the context of geometric logic,
including those of the subtheory $\BAw$ of $\BA$.

Although the arithmetical axioms of $\BA$ are essentially the same as Peano
axioms, $\BA$ is weaker than $\HA$. For instance, $\BA$ does not prove the
cancellation law, i.e., $x+y=x+z\Ra y=z$ (See Lemma \ref{3d17}). More
interestingly, every provably total function of $\BA$ is primitive recursive.
This result is not new; it has already been proved using primitive recursive
realizability introduced in \cite{S03}. We use a different analysis based
on \emph{functionality} of the relations given by the defining formulas,
instead of their \emph{totality}. Moreover, we show that every provably total
function of $\BA$ is definable in $\BA$ by a geometric formula.
However, there are primitive recursive functions that are
\emph{not} provably total in $\BA$. One such primitive recursive function is
the cut-off subtraction. That means that the set of all the provably total
recursive functions of $\BA$, indicated by $\PTRF(\BA)$ is a \emph{proper}
subset of all the primitive functions, $\PR$. 

We consider three extensions of $\BA$. One is a \emph{logical} extension and
the other two are \emph{arithmetical} extensions. The logical extension we
study is $\EBA$, an extension of $\BA$ with the \emph{logical} axiom
$\top\ra\bot\Ra\bot$. This extension is introduced in \cite{AH} and it is
shown there that its behavior is very close to $\HA$, but still weaker than
that. We show that $\PTRF(\EBA)=\PR$. The two arithmetical extensions we will
consider are the following. The first one is an extension of $\BA$ by the
\emph{arithmetical} axiom $\U:x+y=x+z\Ra y=z$ (the cancellation law). It turns
out that $\PTRF(\BA+\U)=\PR$.
The other one is an extension of $\BA$ by adding a symbol for the cut-off
subtraction in the language, and adding its properties as extra axioms to
$\BA$. Again, it turns out that the provably total recursive functions of this
extension, denoted by $\BAc$, are exactly the primitive recursive functions. 

It is shown that the well-known Matiyasevich-Robinson-Davis-Putnam, MRDP
theorem does \emph{not} hold in $\BA$ and the two arithmetical extensions of
$\BA$ mentioned above. However, $\EBA$ proves the MRDP theorem. More properties
of $\EBA$ are considered in the last section of this paper. For instance, it is
shown that $\EBA$ proves every $\Pit$ theorem of the classical arithmetical
theory $\ISig$. Moreover, it is shown that $\PA$ is $\Pio$-conservative over
$\EBA$, and in particular, $\EBA\vdash\Con_{\ISign}$, for any $n$.

In the course of our investigation of the provably total functions and
the MRDP theorem for $\BA$ and its extensions, we use geometric logic
and arithmetic as a bridge, to connect our theories of interest to some
classical theories. Most notably, we consider two classical arithmetical
theories $\IEx$ and $\IEx+\U$. It turns out that the class of provably total
recursive functions of both of these theories are exactly the class of
primitive recursive functions, and while the MRDP theorem holds for the
latter theory, it does not hold for the former. We did not find any
indication of these results in the literature, and they are interesting
on their own.

Finally, based on our analysis in this paper, we suggest that $\BA+\U$
is a better candidate for the title ``basic arithmetic" than $\BA$.
The same holds for the theory which we have denoted in this paper by $\GA$
and named ``geometric arithmetic": $\GA+\U$ would be more suitable for
the title.

\section{Preliminaries to Basic Arithmetic}\label{prelim}
In this section, we introduce the logical and arithmetical theories
we will work with, including $\BQC$ and $\BA$. The semantics of
Kripke model theory for $\BQC$, $\BA$ and their extensions is introduced
and its basic properties that are used in this paper are stated. We
assume that the reader is familiar with the semantics of first-order
classical and intuitionistic theories
(See Chapters 3 and 6 of \cite{D}, for example).
For motivations and basic properties of $\BQC$ and $\BA$,
see \cite{R98} and \cite{AH}.

\label{bqc}
\subsection{Axioms and Rules of Basic Predicate Calculus}
The logical vocabulary of \emph{Basic Predicate Calculus}, $\BQC$ is the
same as that of \emph{Intuitionistic Predicate Calculus}, $\IQC$. It was
originally axiomatized in sequent notation, i.e., using pairs of formulas
called \emph{sequents}, denoted by $A\Ra B$ where $A$ and $B$ are formulas
in the logical language $\{\lor, \land, \ra, \bot, \top, \ex, \all\}$
\cite{R98}. Since \emph{Modus Ponens} is \emph{not} a rule in
$\BQC$, a universally quantified formula like $\all x\all
y\,A$ is different from $\all xy\,A$.
The first one is read $ \all x (\top \ra \all y
(\top \ra A)) $ and the second one is read $\all x y (\top \ra A)$. 
In $\BQC$, when we write
$\all\x\,(A\ra B)$, we mean $\x$ to be a finite
sequence of variables \emph{once} quantified. For existential
quantification no such problems occur over $\BQC$. So, as usual over $\IQC$,
we may occasionally write $\ex\x\,A$ as short for $\ex x_1\dots\ex x_n\,A$.
Beside a set of predicate and function symbols of possibly different finite
arity, we also include the binary predicate ``$=$" for \emph{equality}.
\emph{Terms} and a\emph{tomic formulas} are defined as usual. A \emph{prime
formula} is either an atomic formula, $ \top $ or $ \bot $. \emph{Formulas}
are defined as usual except for implication and universal quantification: if
$A$ and $B$ are formulas, and $\x$ is a (possibly empty) finite sequence of
variables, then $\all\x\,(A\ra B)$ is a formula. In case $A$ is not of the
form $B\ra C$, $\forall\x\,A$ will be used as an abbreviation for
$\forall\x\,(\top\ra A)$. An \emph{implication} $A\ra B$ is a
\emph{universal quantification} $\all\x\,(A\ra B)$ where $\x$ is
the empty sequence. $\neg A$ means $A\ra\bot$, and $\all\x\,(A\lr B)$
means $\all\x\,(A\ra B)\land\all\x\,(B\ra A)$. The concepts of \emph{free} and
\emph{bound} variables are defined as usual. A \emph{fresh} variable
is a variable that does not occur in any of the terms and formulas
in the context being discussed. A \emph{sentence} is a formula
with no free variables. Given a sequence of variables $\x$ without repetitions
and a sequence of terms $\t$ with the same length as $\x$, substitutability
of $\t$ for $\x$ in a formula is defined as usual. $s[\x/\t]$ and
$A[\x/\t]$ stand for, respectively, the term and formula that results from
substituting the terms $\t$ for all free occurrences of the variables of $\x$
in the term $s$ and the formula $A$. For details, see \cite{R98} and
\cite{AH}. We often simply write $A$ as an abbreviation for the sequent
$\top\Ra A$, and $A\Lr B$ for $A\Ra B$ and $B\Ra A$ together.

A \emph{rule} $R$ assigns a sequent $\alpha$ to finitely many sequents
$\alpha_1$, $\dots$ and $\alpha_n$, and is denoted by
\AXC{$\alpha_1\quad\dots\quad\alpha_n$}\UIC{$\alpha$}\DP. $\alpha$ is
called the \emph{lower sequent} of $R$, and $\alpha_1$, $\dots$ and
$\alpha_n$ are called the \emph{upper sequents} of $R$. To any
given first-order language, one can add several new propositional
symbols, and define substitution of formulas of the original language
in place of the added propositional symbols appearing in a given
formula in the extended language, as usual. This definition of
substitution can be extended to the sequents and rules, in the
obvious way. A \emph{sequent schema} is a collection of substitution
instances of a certain sequent. A \emph{rule schema} is defined similarly.
A \emph{theory} $\T$ is given by a set of sequents called the \emph{axioms},
and a set of rules. An \emph{axiom schema} is a sequent schema, all of
the instances of which are axioms. We may simply refer to axiom schemas
and rule schemas as axioms and rules, respectively.
In the following subsections, we introduce several theories that
are of interest in this paper, by listing their axioms and rules.
Unless specified otherwise, the formulas, variables,
sequences of variables and sequences of terms appearing in the listed
axioms and rules are arbitrary. A double horizontal line denotes
\emph{reversible} rules; i.e. it denotes several rules together,
where each \emph{converse} rule has one of the upper sequents of
the original rule as the lower sequent, and the lower sequent of
the original rule as the only upper sequent.

A \emph{derivation} $\D$ is a finite rooted tree with a sequent assigned
to each node. We often identify the nodes of a derivation with the sequent
assigned to them. Considering the partial order of the tree as an
above-below relation (so that the root is the lower-most node), the
upper-most nodes are called the \emph{initial nodes} of the derivation, and
the sequents assigned to them are called the \emph{initial sequents}.
For a theory $\T$, a set of sequents $\Gamma$ and a sequent $\alpha$,
A \emph{derivation of $\alpha$ from $\Gamma$ in $\T$} is a derivation with
$\alpha$ as the root, initial sequents either belonging to $\Gamma$ or
the set of axioms of $\T$, and each non-initial sequent the lower sequent
of a rule of $\T$ such that the sequents immediately above it are
the upper sequents of that rule. We say that $\alpha$ is \emph{derivable}
or \emph{provable} in $\T$ from $\Gamma$, or alternatively, $\T$
\emph{proves} $\alpha$ from $\Gamma$, whenever such a derivation exists.
We denote this relation by $\T+\Gamma\vdash\alpha$. In case $\Gamma$
is empty, we drop the ``from $\Gamma$", and simply write $\T\vdash\alpha$.
$\T$ \emph{proves} a rule
\AXC{$\alpha_1\quad\dots\quad\alpha_n$}\LeftLabel{$R=$}\UIC{$\alpha$}\DP,
or $R$ is \emph{derivable} in $\T$, whenever
$\T+\{\alpha_1,\dots,\alpha_n\}\vdash\alpha$. Derivability of a rule
in a theory from a set of sequents is defined similarly.
For theories $\T$ and $\T'$, $\T+\T'$ denotes the theory whose sets of axioms
and rules are the union of axioms of $\T$ and $\T'$ and the union of the
rules of $\T$ and $\T'$, respectively (note that this notation is consistent
with the previous notation for derivability from a set of sequents).
$\T$ is an \emph{extension} of $\T'$, denoted by $\T\vdash\T'$, whenever
the language of $\T$ extends that of $\T'$, and we have
$\T\vdash\alpha$ and $\T'\vdash R$ for all axioms $\alpha$ and all rules
$R$ of $\T'$. $\T\vdv\T'$ means $\T\vdash\T'$ and $\T'\vdash\T$. A
\emph{theorem} of $\T$ is a sentence $A$ such that $\T\vdash A$. $\Th(\T)$
denotes the set of all theorems of $\T$.\footnote{
This notation ``$\Th$" is different from that used in \cite{R98},
which is for sequents and not just sentences.} For a class $\C$ of formulas,
$\Th_\C(\T)$ denotes the set of all \emph{$\C$-theorems} of $\T$,
$\Th(\T)\cap\C$. $\T$ is a \emph{conservative} extension of $\T'$ whenever
$\T\vdash\T'$ and $\Th_\C(\T)=\Th(\T')$, with $\C$ being the class of
all formulas in the language of $\T'$. $\T$ is said to be
\emph{$\C$-conservative over $\T'$} whenever $\T\vdash\T'$ and
$\Th_\C(\T)=\Th_\C(\T')$.

\subsubsection*{Axioms and rules of logical theories}
We first introduce \emph{Geometric Predicate Calculus},
$\GQC$, which is the weakest logical theory we need. $\GQC$
is formalized in a language without implication and universal quatification.
Formulas in such a language are called \emph{geometric} formulas.
A sequent $A\Ra B$ is called \emph{geometric} whenever both $A$ and $B$
are geometric formulas. A rule is called \emph{geometric} whenever
its lower sequent and all its upper sequents are geometric. $\GQC$
is a theory with geometric axioms and rules. In the following list
of axioms and rules of $\GQC$, all the formulas are supposed to be geometric.
\begin{enumerate}
\item $A\Ra A$;
\item $A\Ra\top$;
\item $\bot\Ra A$;
\item $A\land(B\lor C)\Ra (A\land B)\lor (A\land C)$;
\item $A\land\ex x\,B\Ra\ex\,x(A\land B)$;
where $x$ is not free in $A$;
\item $x=x$;
\item $x=y\land A\Ra A[x/y]$,
where $A$ is atomic;
\item \AXC{$A\Ra B$}\AXC{$B\Ra C$}\BIC{$A\Ra C$}\DP;
\item \AXC{$A\Ra B$}\AXC{$A\Ra C$}\doubleLine\BIC{$A\Ra B\land C$}\DP;
\item \AXC{$B\Ra A$}\AXC{$C\Ra A$}\doubleLine\BIC{$B\lor C\Ra A$}\DP;
\item \AXC{$A\Ra B$}\UIC{$A[\x/\t]\Ra B[\x/\t]$}\DP;
where $\t$ is substitutable for $\x$ in both $A$ and $B$;
\item \AXC{$B\Ra A$}\doubleLine\UIC{$\ex x\, B\Ra A$}\DP,
where $x$ is not free in $A$.
\end{enumerate}

\emph{Basic Predicate Calculus}, $\BQC$ extends $\GQC$ first by
including all the instances of the above schemas in the language
containing implication and universal quantification, and second by
adding the following axioms and rules:
\begin{enumerate}
\setcounter{enumi}{12}
\item $\all\x\,(A\ra B)\land\all\x\,(B\ra C)\Ra\all\x\,(A\ra C)$;
\item $\all\x\,(A\ra B)\land\all\x\,(A\ra C)\Ra\all\x\,(A\ra B\land C)$;
\item $\all\x\,(B\ra A)\land\all\x\,(C\ra A)\Ra\all\x\,(B\lor C\ra A)$;
\item $\all\x\,(A\ra B)\Ra\all\x\,(A[\x/\t]\ra B[\x/\t])$,
where $\t$ is substitutable for $\x$ in both $A$ and $B$;
\item $\all\x\,(A\ra B)\Ra\all\y\,(A\ra B)$,
where no variable in $\y$ is free on the left hand side;
\item $\all\y x\,(B\ra A)\Ra\all\y\,(\ex x\, B\ra A)$,
where $x$ is not free in $A$;
\item \AXC{$A\land B\Ra C$}\UIC{$A\Ra\all\x\,(B\ra C)$}\DP,
where no variable in $\x$ is free in $A$.
\end{enumerate}

There are three other logical theories that we are interested in.
\emph{Extended Basic Predicate Calculus}, $\EBQC$ is axiomatized by
adding the axiom $\top\ra\bot\Ra\bot$ to $\BQC$, \emph{Intuitionistic
Predicate Calculus}, $\IQC$ is axiomatized by adding the schema
$\top\ra A\Ra A$ to $\BQC$, and \emph{Classical Predicate Calculus},
$\CQC$ is axiomatized by adding the schema $A\lor\neg A$ to $\IQC$.

1-19 are collectively called the \emph{logical axioms and rules};
1-4 are the \emph{propositional axioms}; 5 and 13-18 are the
\emph{predicate axioms}; 6 and 7 are the \emph{equality axioms};
8-10 are the \emph{propositional rules}; 11, 12 and 19 are the
\emph{predicate rules}.

\subsection{Axioms and Rules of Basic Arithmetic}
\label{AaRoBA}

The arithmetical theories we consider are all formalized in a
first-order language containing non-logical symbols from the set
$\{0,S,+,\cdot\}$, where $0$ is a constant symbol for zero, $S$ is a unary
function symbol for successor, and $+$ and $\cdot$ are binary function
symbols for addition and multiplication, respectively.

The ordering relation can be formalized in arithmetical theories by defining
$s<t$ for any terms $s$ and $t$ to mean $\ex x(s+Sx=t)$, with $x$ being a
fresh variable. $s\le t$ can be defined as $s<t\lor s=t$, and $s>t$ and
$s\ge t$ can be defined in the obvious ways. Alternatively, one can
include ``$<$" as a binary predicate symbol in the language, and
$s<t\Lr\ex x(s+Sx=t)$ as axioms in the theory. For the properties of
the arithmetical theories that we are interested in, these alternatives
are essentially equivalent. For convenience,
we sometimes the latter alternative, and the arithmetical theories
will be considered in the language $\LL=\{0,S,+,\cdot,<\}$
(in the case of $\BAc$, the language $\LL\cup\{\dotminus\}$)
and with the mentioned axioms, without further mentioning of it. While
theories like $\Q$, $\GA$, $\BA$, $\EBA$, $\HA$, $\PA$ and $\IC$
(see the following subsection for definitions) were originally formalized
in the language without the primitive symbol $<$, we denote the theories
in the language $\LL$ by the same names, since each of them is a conservative
extension of the corresponding theory. This can be seen by the fact that if
one substitutes each appearance of a formula of the form $s<t$ in the
sequents of a derivation in the extended theory by a formula of the form
$\ex x(s+Sx=t)$, the result will be a derivation in the original theory.

\subsubsection*{Axioms and rules of arithmetical theories}

The weakest set of arithmetical axioms that we consider
consists of the ones in the following list.
\begin{enumerate}
\setcounter{enumi}{19}
\item $Sx=0\Ra\bot$;
\item $Sx=Sy\Ra x=y$;
\item $x+0=x$;
\item $x+Sy=S(x+y)$;
\item $x\cdot 0=0$;
\item $x\cdot Sy=x\cdot y+x$.
\end{enumerate}
These axioms, together with $x=0\lor\ex y(x=Sy)$, are collectively known
as \emph{Robinson's axioms}. The theory axiomatized by Robinson's axioms
over $\CQC$ is known as \emph{Robinson Arithmetic}, $\Q$. The exclusion
of the last of the Robinson's axioms from the above list is due to the fact
that all the arithmetical theories considered in this paper prove it.
To see why, note that it is derivable by geometric logic and the instance
of the rule of induction with the induction formula
$x=0\lor\ex y(y<x\land x=Sy)$ (see below for the definitions).
All the arithmetical theories we consider extend $\GQC$ and
have the mentioned instance of the rule of induction.

\emph{Basic Arithmetic}, $\BA$ is formalized over $\BQC$ by the axioms
in the above list together with the following axiom and rule, respectively
called the \emph{induction axiom} and the \emph{rule of induction}.
\begin{enumerate}
\setcounter{enumi}{25}
\item $\all\y x\,(A\ra A[x/Sx])\Ra\all\y x\,(A[x/0]\ra A)$,
\item \AXC{$A\Ra A[x/Sx]$}\UIC{$A[x/0]\Ra A$}\DP.
\end{enumerate}

\emph{Weakened Basic Arithmetic}, $\BAw$ is the theory obtained by
dropping the rule of induction from $\BA$. Any extension $\T$ of $\BAw$
is in fact closed under the rule
\AXC{$A\Ra A[x/Sx]$}\UIC{$A[x/0]\Ra\all x(\top\ra A)$}\DP.
To see why, assume that $\T\vdash A\Ra A[x/Sx]$. This implies
$\T\vdash\all x(A\ra A[x/Sx])$, which combining with the induction axiom
gives $\T\vdash\all x(A[x/0]\ra A)$. As
$\T\vdash A[x/0]\Ra\all x(\top\ra A[x/0])$, we get
$\T\vdash A[x/0]\Ra\all x(\top\ra A)$.

\emph{Extended Basic Arithmetic}, $\EBA$ is formalized by the same
arithmetical axioms and rules as $\BA$ over $\EBQC$.
\emph{Heyting Arithmetic}, $\HA$ and \emph{Peano Arithmetic}, $\PA$
are respectively formalized over $\IQC$ and $\CQC$
by the same arithmetical axioms, but without the rule of
induction. Note that by the above discussion about $\BAw$, any theory
extending $\HA$ is automatically closed under the rule of induction,
since $\HA\vdash\all x(\top\ra A[x/0])\Ra A[x/0]$.

The formula $A$ appearing in the induction axiom and the rule of induction
is called the \emph{induction formula}, and the variable $x$ in it is
called the \emph{eigenvariable}. The induction formula can be restricted
to belong to a certain class $\C$ of formulas, resulting in the
\emph{$\C$-induction axiom} and the \emph{rule of $\C$-induction}.
We will consider several theories with such restricted inductions.
The first one is \emph{Geometric Arithmetic}, $\GA$ which is axiomatized
(in the language restricted to geometric formulas) over $\GQC$ by the
axioms 20-25 and the rule of \emph{geometric induction};
i.e. the rule of induction restricted to geometric induction formulas.
Before introducing the other theories, we first need to
define several classes of formulas.

\begin{definition}
\label{1d6}
$ $
\begin{itemize}
\item $\Dz$ is the smallest set of formulas such that:
\begin{itemize}
\item If $A(\x)$ is prime, then $A(\x)\in\Dz$,
\item If $A(\x),B(\x)\in\Dz$, then $A(\x)\circ B(\x)\in\Dz$,
where $\circ$ is $\land$, $\lor$ or $\ra$,
\item If $A(x,\y)\in \Dz$ and $s$ is a term in which
$x$ does not occur, then $\ex x(x<s \land A(x,\y))\in\Dz$,
\item If $A(x,\y)\in \Dz$ and $s$ is a term in which
$x$ does not occur, then $\all x(x<s \ra A(x,\y))\in\Dz$.
\end{itemize}
\item For any natural number $n$, $\Sign$ and $\Pin$ are
defined as follows:
\begin{itemize}
\item $\Sigz=\Piz=\Dz$,
\item $\Signp$ is the smallest set of formulas such that
if $A(\x,\y)\in\Pin$ then $\ex\x\,A(\x,\y)\in\Signp$,
\item $\Pinp$ is the smallest set of formulas such that
if $A(\x,\y)\in\Sign$ then $\all\x\,A(\x,\y)\in\Pinp$.
\end{itemize}
\item $\Open$ is the set of formulas whith no existential quantification
and no universal quatification over nonepmty sequence of variables (i.e.
it may contain implications). A formula in $\Open$ may be called
\emph{quantifier-free} or \emph{open}.
\item $\Eo$ is the set of formulas of the form $\ex\x\,A(\x,\y)$,
where $A(\x,\y)$ is a quantifier-free formula. A formula in $\Eo$ is
sometimes called \emph{existential} or \emph{Diophantine}.
\item $\Ex$ is the set of formulas of the form $\ex\x\,A(\x,\y)$,
where $A(\x,\y)$ is a geometric quantifier-free formula; meaning that
$A(\x,\y)$ does not contain implications. A formula in $\Ex$ is
sometimes called \emph{positive existential}, hence the notations
$\Ex$ and $\Ep$ are used for the class of formulas.\footnote
{Every $\Ex$ formula is geometric, and every gemetric formula is
equivalent to a $\Ex$ formula over geometric logic, which is
the weakest logic we are considering. We could simply skip this
definition, and work only with geometric formulas instead.
Our purpose of the definition is only to use the notation that already
has been used in the literature.}
\end{itemize}
\end{definition}

$\EBD$ is defined to be the theory axiomatized over $\EBQC$ by the
axioms 20-25, the $\Dz$-induction axiom and the rule of $\Dz$-induction.

For a class $\C$ of formulas, by $\IC$ we
mean the \emph{classical} arithmetical theory $\Q$ together with
induction restricted to $\C$ formulas. I.e. $\IC$ is axiomatized
over $\CQC$ by the axioms 20-25 and the $\C$-induction axiom.
Note that we can argue similar to what we did for $\BAw$ to
prove that in fact $\IC$ is closed under the rule of $\C$-induction.

Some authors prefer to define $\IC$ based on another theory $\PAm$
instead of $\Q$. $\PAm$ is the \emph{classical} theory of
\emph{nonnegative parts of discretely ordered rings},
which is formulated in the language of ordered rings,
i.e. $\{0,1,+,\cdot,<\}$. By letting $1\equiv S0$, $\PAm$
becomes an arithmetical theory. Conversely, one can
consider arithmetical theories in this language by letting
$St\equiv t+1$. $\PAm$ can be axiomatized over $\CQC$ by the following
list of axioms, taken from Section 2.1 of \cite{K91}.
\begin{enumerate}
\setcounter{enumi}{27}
\item $(x+y)+z=x+(y+z)$;
\item $x+y=y+x$;
\item $(x\cdot y)\cdot z=x\cdot(y\cdot z)$;
\item $x\cdot y=y\cdot x$;
\item $x\cdot(y+z)=x\cdot y+x\cdot z$;
\item $x+0=x$;
\item $x\cdot 0=0$;
\item $x\cdot 1=x$;
\item $x<y\land y<z\Ra x<z$;
\item $x<x\Ra\bot$;
\item $x<y\lor x=y\lor y<x$;
\item $x<y\Ra x+z<y+z$;
\item $0<z\land x<y\Ra x\cdot z=y\cdot z$;
\item $x<y\Ra\ex z\,(x+z=y)$;
\item $0<1$;
\item $x>0\Ra x\ge 1$;
\item $x\ge0$.
\end{enumerate}
There are other axiomatizations of $\PAm$, for example the list
appearing in Theorem 1.10 of \cite{HP} or in the
beginning of \S2 of \cite{GD}. For a classical theory, these
axiomatizations are equivalent. It is rather straightforward to see
that $\PAm\vdash\Q$. Since we have $\Iop\vdash\PAm$ (see Theorem 1.10 in
\cite{HP}), there is no difference in defining $\IC$ over either of
$\Q$ and $\PAm$, when $\C\supseteq\Open$. Our choice of $\Q$ over $\PAm$
is due to relations between $\BA$ and classical arithmetical theories
with induction restricted to geometric formulas (see Corollary \ref{2c6.1}).
Note that by Lemmas 2.5 and 2.8 and Proposition 2.6 in \cite{AH},
$\BA$ proves all the axioms 28-44 except 37.

We consider adding another axiom $\U$ to theories that were discussed in the previous
paragraph, to get around the mentioned problems. $\U$ is defined to be
\begin{enumerate}
\setcounter{enumi}{44}
\item $x+y=x+z\Ra y=z$,
\end{enumerate}
which just expresses the cancellation law for addition.
This is in fact equivalent to axiom 37 and some other candidates
(see Lemma \ref{U-equiv}). The choice of $\U$ over axiom 37
is due to the fact that it is formulated in the language without a symbol
for the ordering relation. We chose $\U$ over the remaining candidates
because of its relation to provable totality of the cut-off subtraction,
which is a recurring subject in this paper. We chose the name ``$\U$"
because of the relation between the cancellation law for addition and the
uniqueness sequent for cut-off subtraction, later denoted by ``$\us(\co)$".

Finally, we consider the theory $\BAc$ in the language
$\Lc=\LL\cup\{\dotminus\}$, where $\dotminus$ is a binary function
symbol for \emph{truncated subtraction} or \emph{cut-off subtraction},
which is the function with the following definition:
$$x\dotminus y=\begin{cases}0&x<y\text;\\x-y&x\ge y\text.\end{cases}$$
$\BAc$ is the theory in the language $\Lc$ axiomatized over $\BQC$ by
the axioms 20-25, the induction axiom and the rule of induction, together
with the following two axioms:
\begin{enumerate}
\setcounter{enumi}{45}
\item $x\le y\Ra x\dotminus y=0$;
\item $y\le x\Ra Sx\dotminus y=S(x\dotminus y)$.
\end{enumerate}

\subsection{Kripke semantics of Basic Logic and Arithmetic}
\label{semantics}

We introduce Kripke semantics for logical and arithmetical theories.
For simplicity of notation, we may not distinguish between a closed
term and the object denoted by it in the model. In particular,
we may identify a natural number $n$ with the \emph{numeral}
$S\dots S0$ consisting of $n$ consecutive appearances of $S$.
Similarly, we may identify a constant symbol denoting an individual in
the domain of discourse with the individual itself, when working in
a language extended with parameters from a given model.

A \emph{Kripke model} for Basic predicate logic is a quadruple
$\km=(K,\prec,D,\Vdash)$, like the one for intuitionistic predicate logic,
except that the relation $\prec$ is \emph{not} needed to be reflexive.
I.e. $\prec$ is a transitive relation on $K$, $D(k)$ is nonempty for
each $k\in K$, $D(k)$ is embedded in $D(k')$ when $k\prec k'$, and if
$k\prec k'$ then $k\Vdash A$ implies $k'\Vdash A$ for an atomic sentence $A$
in the language extended by parameters from $D(k)$. See \cite{AH} and
\cite{R98} for more details. We write $\preceq$ for reflexife closure of
$\prec$, and $\succ$ and $\succeq$ for the converse relations of $\prec$ and
$\preceq$, respectively. The \emph{forcing relation} $\Vdash$ between a
\emph{node} $k\in K$ and a sentence of the form $\all\x\,(A\ra B)$
(in the language extended by parameters from $D(k)$) is defined as
\begin{quote}
$k\Vdash\all\x\,(A\ra B)$ if and only if for all $k'\succ k$
and all $\a\in D(k')$,\\
$k'\Vdash A[\x/\a]$ implies $k'\Vdash B[\x/\a]$.
\end{quote}
This for implication reduces to
\begin{quote}
$k\Vdash A\ra B$ if and only if for all $k'\succ k$, $k'\Vdash A$
implies $k'\Vdash B$.
\end{quote}
For a formula $A$ with free variables $\x$, we define
\begin{quote}
$k\Vdash A$ if and only if for all $k'\succeq k$ and all $\a\in D(k')$,
$k'\Vdash A[\x/\a]$.
\end{quote}
We extend $\Vdash$ to all sequents and rules. For a sequent $A\Ra B$,
it is defined by
\begin{quote}
$k\Vdash A\Ra B$ if and only if for all $k'\succeq k$ and all $\a\in D(k')$,
$k'\Vdash A[\x/\a]$ implies $k'\Vdash B[\x/\a]$.
\end{quote}
For a rule
\AXC{$\alpha_1\quad\dots\quad\alpha_n$}\LeftLabel{$R=$}\UIC{$\alpha$}\DP,
it is defined by
\begin{quote}
$k\Vdash R$ if and only if for all $k'\succeq k$,
if $k'\Vdash \alpha_i$ for all $i$ with $1\le i\le n$, then $k'\Vdash\alpha$.
\end{quote}

A formula $A$ is \emph{valid} in $\km$, denoted by $\km\Vdash A$,
if $A$ is \emph{forced} in all nodes of $\km$,
i.e. $k\Vdash A$ for all $k\in K$. $\km\Vdash\alpha$ and $\km\Vdash R$
are defined similarly for a sequent $\alpha$ and a rule $R$.
For a set $\Gamma$ of sequents, $\km\Vdash\Gamma$ means that every
sequent in $\Gamma$ is valid in $\km$. For a theory $\T$, $\km\Vdash\T$
means that all the axioms and rules of $\T$ are valid in $\km$. By
$\T+\Gamma\Vdash\alpha$ we mean that for all Kripke models $\km$,
if $\km\Vdash\T$ and $\km\Vdash\Gamma$ then $\km\Vdash\alpha$; $\T+\Gamma$
is said to \emph{force} $\alpha$.

\begin{proposition}[Soundness]
\label{soundness}
Let $\T$ be a theory, $\Gamma$ a set of sequents, $\alpha$ a sequent and
$R$ a rule. If $T+\Gamma\vdash\alpha$ then $T+\Gamma\Vdash\alpha$. If
$T+\Gamma\vdash R$ then $T+\Gamma\Vdash R$.
\end{proposition}
\begin{proof}
Proposition 5.3 in \cite{R98}.
\end{proof}

A theory $\T$ is \emph{functional} whenever for all sequences of formulas
$A_0,B_0,A_1,B_1,\dots,A_n,B_n$ and sentences $A$, if
$$\T+\{A_i\Ra B_i\}_{i=1}^n\vdash A_0\Ra B_0\text,$$
then
$$\T+\{A\land A_i\Ra B_i\}_{i=1}^n\vdash A\land A_0\Ra B_0\text.$$
$\T$ is \emph{well-formed} whenever for all sequences of sentences
$\all\x\,(A_0\ra B_0),\dots,\all\x\,(A_n\ra B_n)$
and all formulas $A$ where no free variable of $A$ occurs in $\x$, if
$$\T+\{A_i\Ra B_i\}_{i=1}^n\vdash A_0\Ra B_0\text,$$
then
$$\T\vdash\bigwedge_{i=1}^n\all\x\,(A\land A_i\ra B_i)
\Ra\all\x\,(A\land A_0\ra B_0)\text.$$

\begin{proposition}[Completeness]
\label{completeness}
Let $\T$ be a functional well-formed theory extending $\BQC$,
$\Gamma$ a set of sequents, $\alpha$ a sequent and $R$ a rule.
If $T+\Gamma\Vdash\alpha$ then $T+\Gamma\vdash\alpha$.
If $T+\Gamma\Vdash R$ then $T+\Gamma\vdash R$.
\end{proposition}
\begin{proof}
Theorem 5.8 in \cite{R98}.
\end{proof}

\begin{proposition}
\label{functional-well-formed}
$ $
\begin{enumerate}
\item
$\BQC$, $\EBQC$, $\IQC$ and $\CQC$ are functional and well-formed.
\item
$\BAw$, $\BA$, $\BA+\U$, $\EBA$, $\HA$ and $\PA$
are functional and well-formed. For any class $\C$
of formulas, $\IC$ and $\IC+\U$ are functional and well-formed.
\item
$\BAc$ and $\EBD$ are functional and well-formed.
\end{enumerate}
\end{proposition}
\begin{proof}
$ $
\begin{enumerate}
\item
Corollaries 4.10 and 4.14 in \cite{R98}.
\item
The cases of $\BAw$, $\HA$, $\PA$, $\IC$ and $\IC+\U$ are consequences of
Corollaries 4.10 and 4.14 in \cite{R98}. The other cases follow from Proposition 6.1 in \cite{R98}.
\item
The proof is similar to that of the previous item. By Propositions 4.9
and 4.13 in \cite{R98}, it suffices to prove that each theory has a
functional and well-formed axiomatization. The only arithmetical rules
of each theory are instances of the induction rule. In case of
$\BAc$, \AXC{$A\land B(x)\Ra B(Sx)$}\UIC{$A\land B(0)\Ra B(x)$}\DP and
$\all\y(A\land B(x)\ra B(Sx))\Ra\all\y(A\land B(0)\ra B(x))$ are derivable by
applying the induction rule and axiom, respectively, with $A\land B(x)$
as the induction formula. The case of $\EBD$ will be
covered by Lemma \ref{EBD-func-wll}.\qedhere
\end{enumerate}
\end{proof}

A Kripke model $\km=(K,\prec,D,\Vdash)$ is called \emph{rooted}
whenever there is a node $k\in K$, called the \emph{root}, such that
$k\preceq k'$, for all $k'\in K$. $\km$ is called \emph{reflexive}
(respectively, \emph{irreflexive}) whenever $\prec$ is reflexive
(respectively, irreflexive). A node $k\in K$ is called \emph{reflexive}
(respectively, \emph{irreflexive}) whenever $k\prec k$ (respectively,
$k\nprec k$). $k$ is called \emph{terminal} whenever for all $k'\in K$,
if $k\prec k'$ then $k=k'$.

A theory $\T$ is \emph{sound} with respect to a class of Kripke models
whenever all of the axioms and rules derivable in it are valid in all models
in the class. $\T$ is \emph{complete} with respect to a class of Kripke models
whenever any sequent or rule that is valid in all of the models in the class,
is provable in $\T$.

\begin{proposition}
\label{scl}
$ $
\begin{enumerate}
\item
$\BQC$ is sound and complete with respect to
the class of all Kripke models.
\item
$\EBQC$ is sound and complete with respect to
the class of all Kripke models in which
every terminal node is reflexive.
\item
$\IQC$ is sound and complete with respect to
the class of all reflexive Kripke models.
\item
$\CQC$ is sound and complete with respect to
the class of all single-node reflexive Kripke models.
\end{enumerate}
\end{proposition}
\begin{proof}
$ $
\begin{enumerate}
\item
Consider Proposition \ref{completeness} for $\T=\BQC$.
\item
Theorem 3.7 in \cite{AH}.
\item
This is the soundness and completeness theorem
of the usual Kripke semantics for intuitionistic logic.
\item
The forcing relation in a single-node reflexive Kripke model
is equivalent to the satisfaction relation in the corresponding
classical model. Thus, this is equivalent to soundness and
completeness theorem of the usual classical semantics.\qedhere
\end{enumerate}
\end{proof}

\begin{proposition}
\label{sca}
$ $
\begin{enumerate}
\item
$\BAw$, $\BA$, $\BA+\U$ and $\BAc$ are sound and complete with respect to
the class of all Kripke models which force their corresponding
arithmetical axioms and rules.
\item
$\EBA$ and $\EBD$ are sound and complete with respect to the
class of all Kripke models with reflexive terminal nodes which force
their corresponding arithmetical axioms and rules.
\item
$\HA$ is sound and complete with respect to the class of all reflexive
Kripke models which force its corresponding arithmetical axioms.
\item
$\PA$, $\IC$ and $\IC+\U$ (for any class $\C$ of formulas) are
sound and complete with respect to the class of all single-node reflexive
Kripke models which force their corresponding arithmetical axioms.
\end{enumerate}
\end{proposition}
\begin{proof}
Combine Propositions \ref{soundness}, \ref{completeness},
\ref{functional-well-formed} and \ref{scl}.
\end{proof}

Consider a theory $\T$ in any first-order language including at least
one constant symbol, so that the set of closed terms is nonempty.
Let $D_\T$ be the quotient set of the set of closed terms by the
equivalence relation of \emph{provable equality in $\T$}; i.e.
the relation defined with $\T\vdash s=t$ for any closed terms
$s$ and $t$. For each set of models $\{\km_i\}_{i\in I}$ of $\T$
we can construct two new models, denoted by $\kb$ and $\kc$, as follows.
Both models are formed by taking the disjoint union of the models
$\{\km_i\}_{i\in I}$ and then adding a new root. In $\kb$ the new root
$\bk$ is reflexive with $D(\bk)=D_\T$, and in $\kc$ the new root $\ck$
is irreflexive with $D(\ck)=D_\T$. In case we are working in the languages
$\LL$ and $\Lc$ of arithmetic and $\T$ is any of the theories
defined above, any closed term is provably equal in $\T$ to a numeral,
and instead of $D_\T$ we can simply take $D(\bk)$ and $D(\ck)$
to be equal to $\nn$, the set of standard natural numbers.
We call $\T$ a \emph{reflexively rooted} theory if for each set
$\{\km_i\}_{i\in I}$ of rooted Kripke models of $\T$, the model
$\kb$ is also a model of $\T$. $\T$ is called
\emph{irreflexively rooted} if $\kc$ is a model of $\T$, and
\emph{fully rooted} if both $\kb$ and $\kc$ are models of $\T$.

\begin{proposition}
\label{rooted}
$ $
\begin{enumerate}
\item
Over a language contaning at least one constant symbol,
$\BQC$ and $\EBQC$ are fully rooted, and $\IQC$ is reflexively rooted.
\item
$\BA$, $\BA+\U$, $\BAc$, $\EBA$ and $\EBD$ are fully rooted.
$\HA$ is reflexively rooted.
\end{enumerate}
\end{proposition}
\begin{proof}
$ $
\begin{enumerate}
\item
The reflexive rootedness of $\IQC$ is well known. The full rootedness of
$\BQC$ follows from item 1 of \ref{scl}. The full rootedness of $\EBQC$
follows from the fact that $\top\Ra\bot\Ra\bot$ is valid in a Kripke model
iff all terminal nodes in the model are reflexive. Consider
$\kb$ and $\kc$ built from the set $\{\km_i\}_{i\in I}$ of models of
$\EBQC$. Since the new root in $\kb$ and $\kc$ is not terminal,
all terminal nodes in $\kb$ and $\kc$ are in some $\km_i$, and
therefore reflexive. By full rootedness of $\BQC$, it follows that
both $\kb$ and $\kc$ are models of $\EBQC$.
\item
For $\BA$ and $\HA$, see Proposition 6.2 in \cite{R98}. This will
immediately imply the statements for $\BA+\U$ and $\EBA$, using
the previous item. The proofs for $\BAc$ and $\EBD$ are identical
to those of $\BA$ and $\EBA$, only focusing on their corresponding
induction formulas in their corresponding languages.\qedhere
\end{enumerate}
\end{proof}

A theory $\T$ is \emph{faithful} whenever for all sequences of sentences
$\all\x\,(A_0\ra B_0),\dots,\all\x\,(A_n\ra B_n)$, if
$$\T\vdash\bigwedge_{i=1}^n\all\x\,(A_i\ra B_i)
\Ra\all\x\,(A_0\ra B_0)\text,$$
then
$$\T+\{A_i\Ra B_i\}_{i=1}^n\vdash A_0\Ra B_0\text.$$
$\T$ has the \emph{disjunction property}, if $\T\vdash A
\lor B$ implies $\T\vdash A$ or $\T\vdash B$, for all sentences
$A$ and $B$. $\T$ has the \emph{existence property},
if $\T\vdash\ex xA$ implies $\T\vdash A[x/t]$ for some
closed term $t$ in $\LL$, for all sentences $\ex xA$.

\begin{proposition}
\label{fwr-dep}
A functional well-formed theory which is either
reflexively rooted or irreflexively rooted is faithful,
and has the disjunction and existence properties.
\end{proposition}
\begin{proof}
Proposition 5.14 in \cite{R98}.
\end{proof}

\begin{corollary}
\label{bafaith}
$\BQC$, $\EBQC$, $\IQC$, $\BA$, $\BA+\U$, $\BAc$, $\EBA$ and $\EBD$
are faithful, and have the disjunction and existence properties.
\end{corollary}
\begin{proof}
Combine Propositions \ref{functional-well-formed},
\ref{rooted} and \ref{fwr-dep}.
For the case of logical theories over languages
without constant symbols, extra care is needed.
See Propositions 5.14 and 4.12 in \cite{R98}.
\end{proof}

A useful way of looking at a Kripke model $\km=(K,\prec,D,\Vdash)$ is by
considering it as a frame such that at any node $k\in K$, a classical
model $\M$ is attached. The domain of discourse of $\M$ is $D(k)$,
and the positive diagram of $\M$ consists of the atomic formulas that
are forced at $k$. If $k\prec k'$, then $\M$ is a substructure of $\Mp$.

For a theory $\T$, a Kripke model $\km$ is said to be \emph{$\T$-normal}
if for every node $k$ of $\km$, $\M\models\T$.

\begin{corollary}
\label{geomba}
Let $\T$ be given by all the geometric sentences true in
the standard model $\nn$. Then $\BA$ is $\T$-normal.
Consequently, for any geometric sentence $A$, if $\nn\models A$
then $\BA\vdash A$. Hence $\BA$, $\EBA$, $\HA$, and $\PA$
prove the same geometric sentences.
\end{corollary}
\begin{proof}
Let $A$ be a geometric sentence and $\nn\models A$. Let $\km$ be
an arbitrary model of $\BA$. The model $\kc$ formed by $\{\km\}$
is also a model of $\BA$, by Propositions \ref{rooted}.
$\Mc\models A$, as $\Mc=\nn$, and $\M\models A$ for any
node $k$ of $\km$, as $A$ is geometric and $\Mc$ is
a substructure of $\M$. So $k\Vdash A$ for any node $k$
of $\km$, because $A$ is geometric, and for geometric $A$
we have $k\Vdash A$ iff $\M\models A$. Therefore $\km\Vdash A$,
and as $\km$ was an arbitrary model of $\BA$, $\BA\Vdash A$.
Hence $\BA\vdash A$, by Proposition \ref{sca}.
\end{proof}

%%%%%%%%%%%%%%%%%%%%%%%%%%%%%%%%%%%%%%%%%%%%%%%%%%%%%%%%%%%%%%%%%%%%%%%%%%%%%%%
%%%%%%%%%%%%%%%%%%%%%%%%%%%%%%%%%%%%%%%%%%%%%%%%%%%%%%%%%%%%%%%%%%%%%%%%%%%%%%%
\section{The Provably Total Recursive Functions of $\BA$ and its extensions}
In this section, we investigate the classification of the provably total
recursive functions of $\BA$ and its three extensions. In the first subsection,
we show that the set of all the provably total recursive functions of $\BA$,
indicated by $\PTRF(\BA)$ is a \emph{proper} subset of all the primitive
recursive functions, $\PR$. Half of this result, i.e.
$\PTRF(\BA)\subseteq \PR$, has already been proved in \cite{S03},
and we give an alternative proof in this paper. The other half of our result,
i.e. $\PTRF(\BA)\ne\PR$, is a consequence of our 
Theorem \ref{3t20} (Remark \ref{rm1}).

In the second subsection, we consider three extensions of $\BA$. As is
explained in Section 1, one of these extensions is a \emph{logical} extension
and the other two are \emph{arithmetical} extensions. It turns out that the
provably total recursive functions of all these three extensions of $\BA$ are
\emph{exactly} the primitive recursive functions.

\subsection{About the Provably Total Recursive Functions of $\BA$}
\begin{definition}
\label{2d1}
For a formula $A$, the \emph{geometric part} of $A$ is denoted
by $\gp A$ and is defined recursively as follows:
\begin{itemize}
\item $\gp A\equiv A$ if $A$ is prime,
\item $\gp A\equiv\gp B\circ\gp C$ if $A$ is of the form
$B\circ C$ and $\circ$ is $\lor$ or $\land$,
\item $\gp A\equiv\ex u\,\gp B$ if $A$ is of the form $\ex u\,B$,
\item $\gp A\equiv\top$ if $A$ is of the form $\all\x\,(B\ra C)$.
\end{itemize}
$\gp{(A\Ra B)}$ is defined as $\gp A\Ra\gp B$. For a set
of sequents $\Gamma$, $\gp\Gamma$ is defined as the set
$\left\{\gp\alpha\bigm|\alpha\in\Gamma\right\}$.
For a rule $R$, $\gp R$ is defined as the rule obtained by
replacing any of the upper and lower sequents with its geometric part.
For a theory $\T$, $\gp\T$ is the theory formalized by all
geometric sequents and rules derivable in $\T$; i.e. axioms of
$\gp\T$ are the geometric sequents $\alpha$ such that $\T\vdash\alpha$,
and rules of $\gp\T$ are the geometric rules such that $\T\vdash R$.
\end{definition}
\begin{proposition}
\label{2p2}
$ $
\begin{enumerate}
\item
For any formula $A$, $\gp A$ is geometric.
\item
If $A$ is geometric, $\gp A=A$.
\item
For any formula $A$, $\BQC\vdash A\Ra\gp A$.
\end{enumerate}
\end{proposition}
\begin{proof}
Straightforward by induction on the complexity of $A$.
\end{proof}
\begin{proposition}
\label{2p5}
$ $
\begin{enumerate}
\item
If $\BQC+\Gamma\vdash\alpha$ then $\GQC+\gp\Gamma\vdash\gp\alpha$.
Consequently, $\gp\BQC\vdv\GQC$, and $\BQC$ is conservative over $\GQC$.
\item
If $\BA+\Gamma\vdash\alpha$ then
$\GA+\gp\Gamma\vdash\gp\alpha$.
Consequently, $\gp\BA\vdv\GA$, and $\BA$ is conservative over $\GA$.
\end{enumerate}
\end{proposition}
\begin{proof}
$ $
\begin{enumerate}
\item
This is a generalization of Proposition 4.2 in \cite{R98}, and the proof
goes along the similar lines. For any axiom $\alpha$ and any rule $R$ of
$\BQC$ in the items 1-12 of the list of the axioms and rules, $\gp\alpha$
and $\gp R$ are geometric instances of the same axiom and rule, respectively,
and therefore they are axioms and rules of $\GA$. For an axiom $\alpha$ in
the items 13-18 of the list, $\gp\alpha=\top\Ra\top$, and therefore an
axiom of $\GQC$. For rule 19, the lower sequent of its geometric part is
of the form $\gp A\Ra\top$, and therefore an axiom of $\GQC$. Hence,
for any derivation $\D$ of $\alpha$ from $\Gamma$ in $\BQC$, one can
replace any sequent in the derivation with its geometric part, and
the sub-derivations ending in an application of rule 19 by a single
geometric axiom, and the result will be a derivation of $\gp\alpha$
from $\gp\Gamma$ in $\GQC$.
\item
The proof is similar to the previous item. One just needs to additionally
note that the axioms 20-25 are already geometric, the geometric part of
the induction axiom is $\top\Ra\top$, and the geometric part of the
rule of induction is a geometric instance of the same rule, which is
a rule of $\GA$.\qedhere
\end{enumerate}
\end{proof}

As a technical tool for a part of the proof of the next lemma, we temporarily
adopt the standard language of first-order classical theories, in which for
any formulas $A$ and $B$ and any variable $x$, $\neg A$, $A\ra B$ and
$\all x A$ are formulas in their own right (not the abbreviations we have
considered so far). The proof system $\LK$ is a sequent calculus formalizing
first-order classical logic over the standard language. Unlike the sequents
we have considered so far, the sequents in $\LK$ are of the form
$\Delta\Ra\Delta'$, where $\Delta$ and $\Delta'$ are finite lists of formulas
(possibly empty). Classical arithmetical theories can be formalized over
$\LK$ by adding the axioms of equality, arithmetical axioms and the induction
rule (with restricted induction formulas, if necessary). In particular,
for any class $\C$ of formulas, $\IC$ can be formalized by considering
the $\C$-induction rule. See the Appendix for definitions.

\begin{lemma}
\label{2l6.02}
Let $\Gamma$ be a set of geometric seuqents, and $\alpha$ a geometric
sequent.
\begin{enumerate}
\item
$\CQC+\Gamma\vdash\alpha$ iff $\GQC+\Gamma\vdash\alpha$. Consequently,
if $\T$ is any of $\CQC$, $\IQC$ and $\EBQC$, then $\gp\T\vdv\GQC$
and $\T$ is conservative over $\GQC$.
\item
$\IEx+\Gamma\vdash\alpha$ iff $\GA+\Gamma\vdash\alpha$. Consequently,
$\gp{\left(\IEx\right)}\vdv\GA$ and $\IEx$ is conservative over $\GA$.
\end{enumerate}
\end{lemma}
\begin{proof}
$ $
\begin{enumerate}
\item
As $\GQC$ is a sub-system of $\CQC$, one direction is trivial.
For the other direction, consider a derivation for $\alpha$
from $\Gamma$ in the system $\LK$. By \emph{free-cut elimination} of
$\LK+\Gamma$ (Theorem 2.4.5 in \cite{B98}),
there is a derivation of $\alpha$ from $\Gamma$ in $\LK$, every formula
appearing in which is a geometric formula (since every formula in
$\Gamma\cup\{\alpha\}$ is geometric, and the set of geometric formulas is
closed under subformulas and term substitution). As a consequence, none of
the rules ($\neg\Ra$), ($\Ra\neg$), ($\ra\Ra$), ($\Ra\ra$), ($\all\Ra$) and
($\Ra\all$) is used in $\D$. So, $\GQC+\Gamma$ proves $\alpha$, since
the other rules and axioms of $\LK$ are valid in $\GQC$, in the sense
that one can substitute $\bigwedge\Delta\Ra\bigvee\Delta'$
for $\Delta\Ra\Delta'$, and the rules and axioms remain valid.
Here, $\bigwedge\varnothing$ and $\bigvee\varnothing$ are
defined, respectively as $\top$ and $\bot$.
\item
The proof is similar to the previous item. For one direction, note
that every geometric formula is provably equivalent over $\GQC$ to a formula
in $\Ex$, which is nothing but its prenex normal form. Therefore, restricting
the rule of induction in $\GA$ to $\Ex$-formulas results in an
equivalent theory. The rule of $\Ex$-induction is derivable in $\IEx$,
and every other axiom/rule of that theory is an axiom/rule of $\IEx$.
For the other direction, we can use Corollary 2.4.7 in \cite{B98},
which implies that since the set of geometric formulas is closed under
subformulas and substitution, and all the formulas in $\Gamma\cup\{\alpha\}$
and all the induction formulas in a derivation in $\IEx$ are geometric,
if $\IEx+\Gamma\vdash\alpha$ then there is a derivation for it in which
all the formulas are geometric. Again, this means that the derivation
is valid in $\GA$, in the same sense as the previous item.\qedhere
\end{enumerate}
\end{proof}

\begin{corollary}
\label{2c6.1}
For a geometric $\alpha$ and a set $\Gamma$ of geometric sequents,
$\BA+\Gamma\vdash\alpha$ iff $\IEx+\Gamma\vdash\alpha$.
\end{corollary}
\begin{proof}
Combine Proposition \ref{2p5} and Lemma \ref{2l6.02}.
\end{proof}

\begin{definition}
\label{fundef}
$ $
\begin{itemize}
\item
For a function $f:\nn^n\to\nn$ and a
formula $A(\x,y)$ in the language of $\BA$, we say $A$
\emph{defines} $f$ if for every $\a\in\nn^n$ and
every $b\in\nn$, $\nn\models A(\a,b)$ iff $f(\a)=b$.
\item
For a formula $A(\x,y)$ and a variable $y$, the \emph{existence sequent}
related to $A(\x,y)$ and $y$, or simply the \emph{existence sequent} of $A$,
is denoted by $\es(A(\x,y),y)$, or simply by $\es(A)$, and is defined to
be $\top\Ra\ex y A(\x,y)$. If $A(\x,y)$ defines a function $f$,
we also call $\es(A)$ the \emph{existence sequent} of $f$.
The existence sequent is the statement that the relation defined by
$A(\x,y)$ over a given model is \emph{total}.
\item
For a formula $A(\x,y)$ and a variable $y$ and fresh variables $u$ and $v$,
the \emph{uniqueness sequent} related to $A(\x,y)$, $y$, $u$ and $v$,
or simply the \emph{uniqueness sequent} of $A$,
is denoted by $\us(A(\x,y),y,u,v)$, or simply by $\us(A)$, and is defined
to be $A(\x,u)\land A(\x,v)\Ra u=v$.
If $A(\x,y)$ defines a function $f$, we also call
$\us(A)$ the \emph{uniqueness sequent} of $f$.
The uniqueness sequent is the statement that the relation defined by
$A(\x,y)$ over a given model is \emph{functional}; i.e. it coincides
with the graph of a partial function.
\item
For a formula $A(\x,y)$ and a function $f:\nn^n\to\nn$,
we say $A(\x,y)$ \emph{defines $f$ in a theory $\T$},
if it defines $f$ and $\T$ proves its existence and uniqueness sequents.
We call a function \emph{provably total} in a theory $\T$ if it is defined
in $\T$ by some formula. A function is \emph{provably total recursive}
in a theory $\T$ if it is defined in $\T$ by a $\Sig$ formula.
$\PTF(\T)$ and $\PTRF(\T)$ respectively denote the set of all
the provably total functions and the set of all
the provably total recursive functions of $\T$.
\end{itemize}
\end{definition}

\begin{theorem}
\label{2t9}
Let $\Gamma$ be a set of geometric sequents such that $\nn\models\Gamma$.
If a function $f:\nn^n\to\nn$ is defined in $\BA+\Gamma$
by a formula $A(\x,y)$, then it is also defined in $\BA+\Gamma$
by the geometric formula $\gp A(\x,y)$. Consequently
$\PTF(\BA+\Gamma)=\PTRF(\BA+\Gamma)=\PTRF(\GA+\Gamma)=\PTF(\GA+\Gamma)$.
\end{theorem}
\begin{proof}
Assume that a function $f:\nn^n\to\nn$ is defined by
a formula $A(\x,y)$ in the language of $\BA$. If $\BA+\Gamma$ proves
existence and uniqueness sequents related to $A(\x,y)$, then by
Proposition \ref{2p5}, it also proves the existence and
uniqueness sequents related to $\gp A(\x,y)$. Since $\nn$ is a model of
$\BA+\Gamma$, $\gp A(\x,y)$ defines a function $g:\nn^n\to\nn$.
But $g$ is just equal to $f$, since by Proposition \ref{2p2},
$\BQC\vdash A(\x,y)\Ra\gp A(\x,y)$ and that means if
$\nn\models A(\a,f(\a))$ for $\a\in\nn^n$, then
$\nn\models\gp A(\a,f(\a))$. So $f$ is defined in
$\BA+\Gamma$ by $\gp A(\x,y)$, which is a $\Ex$ formula.

Every provably total recursive function in $\BA+\Gamma$ is clearly provably
total, and thus $\PTRF(\BA+\Gamma)\subseteq\PTF(\BA+\Gamma)$.
Conversely, as is shown above, every provably total function in
$\BA+\Gamma$ is definable in $\BA+\Gamma$ by an $\Ex$ formula.
Since $\Ex\subseteq\Sig$, every such function is provably total recursive
in $\BA+\Gamma$, and thus $\PTF(\BA+\Gamma)\subseteq\PTRF(\BA+\Gamma)$.
Also, the above application of Proposition \ref{2p5} guarantees that
the derivations of existence and uniqueness sequents can be done in
the language of geometric logic, which completes the proof.
\end{proof}

\begin{theorem}
\label{isig}
The provably total recursive functions of $\ISig$ are exactly the primitive
recursive functions, i.e. $\PTRF(\ISig)=\PR$.
\end{theorem}
\begin{proof}
\cite{HP}, Chapter 4, Corollary 3.7.
\end{proof}

\begin{theorem}
\label{2t10}
The provably total functions of $\BA$ are primitive recursive, i.e.
$\PTF(\BA)\subseteq\PR$. Furthermore these functions are
definable in $\BA$ by $\Ex$ formulas.
\end{theorem}
\begin{proof}
Let $f$ be a provably total function in $\BA$. Then by Theorem
\ref{2t9}, it is definable in $\BA$ by a geometric formula. So its existence
and uniqueness sequents consist of geometric formulas, and by Corollary
\ref{2c6.1}, $\IEx$ proves these sequents as well. Then $f$ is
defined in $\IEx$ by a geometric formula. As
$\Ex\subseteq\Sig$, $f$ is provably total recursive in $\ISig$.
Hence by Theorem \ref{isig}, $f$ is primitive recursive.
\end{proof}

We now show that not all primitive recursive functions are
provably total in $\BA$. For this purpose, we need to take a look at $\N$,
a specific structure, that \emph {is} a model of $\IEx$, but \emph {not} a 
model of $\Iop$.
\begin{definition}
\label{3n17}
$ $
\begin{itemize}
\item
$\N=\nn\cup\{\8\}$ is a classical structure, where $\8$ is a nonstandard
element such that $S\8=\8+\8=\8\cdot\8=\8$,
$0\cdot\8=\8\cdot0=0$, $\8<\8$, $n<\8$, $\8\nless n$, and
$n+\8=\8+n=(n+1)\cdot\8=\8\cdot(n+1)=\8$,
for every $n\in\nn$.
\item
$\K$ is the Kripke model with just one irreflexive
node with the structure $\N$.
\end{itemize}
\end{definition}
$\N$ is not a model of $\Iop$, because using induction on the
open formula $\neg Sx=x$, one can get $\Iop\vdash Sx=x\Ra\bot$,
which is not satisfied in $\N$.
To show that $\N$ is a model of $\IEx$, we need the observation that
terms in the language of arithmetic correspond to polynomials with
non-negative integer coefficients. Univariate polynomials of this kind
are either constant or strictly increasing over the set of natural numbers.
Also, if two such polynomials agree on infinitely many points, they are
identical. When extending the domain and codomain to $\N$, we see that
in the constant case, the polynomial takes the same value at $\8$ as at
other points, and in the strictly increasing case, it takes the value $\8$
at $\8$. We can generalize these observations by considering terms in the
language of $\N$ (the language of arithmetic extended with parameters from
$\N$), in the following way. Such terms correspond to multivariate
polynomials with coefficients from $\N$. In the univariate case,
two such polynomials agreeing on infinitely many points are equal,
possibly with the exception at $0$. Even in the multivariate case, either the
polynomial is constantly equal to a natural number, or it takes the value
$\8$ at $\9$, where $\9=\langle\8,\dots,\8\rangle$ with the length equal to
the number of variables of the polynomial. We use these observations
to prove that $\N$ has the overspill property for geometric formulas.

\begin{lemma}
\label{3l100}
For every $\ex^+_1$ formula $A(x)$ in the language
of $\N$ with $x$ as the only free variable, if the set
$\{a\in\nn|\N\models A(a)\}$ is infinite, then $\N\models A(\8)$.
\end{lemma}
\begin{proof}
Since $A(x)$ is geometric, we can assume that it is of the form
$\bigvee_i\ex\y\bigwedge_js_{ij}(x,\y)=t_{ij}(x,\y)$. Because
$\{a\in \nn|\N\models A(a)\}$ is infinite,
$\left\{a\in\nn\Big|\N\models\ex\y\bigwedge_js_{kj}(a,\y)=t_{kj}(a,\y)\right\}$
is infinite for some index $k$. For the sake of simplicity, we shall omit
the index $k$ and write the formula $\ex\y\bigwedge_js_{kj}(x,\y)=t_{kj}(x,\y)$
as $B(x)\equiv\ex\y\bigwedge_js_j(x,\y)=t_j(x,\y)$. Let
$S=\{a\in\nn|\N\models B(a)\}$. We show $\N\models B(\8)$ by induction on
$|\y|$ (number of variables appearing in $\y$), which implies
$\N\models A(\8)$. For the base case, $|\y|=0$ and hence $B(x)$ is of the form
$\bigwedge_js_{j}(x)=t_{j}(x)$. Because $S$ is infinite, for every $j$, the
polynomials corresponding to $t_j(x)$ and $s_j(x)$ agree on infinitely many
points, and thus give equal values for nonzero inputs. Hence
$s_j(\8)=t_j(\8)$ for every $j$, which means $\N\models B(\8)$.

For the induction step, assume the statement is true when $|\y|=n$, and suppose
$|\y|$ is $n+1$ in $B(x)$. Let $\9=\langle\8,\dots,\8\rangle$, with $|\9|=|\y|$.
If $\N\models \bigwedge_js_{j}(\8,\9)=t_{j}(\8,\9)$, we are done. Otherwise,
there exists an index $u$ such that $\N\models s_u(\8,\9)\ne t_u(\8,\9)$.
Note that in this case, one of $s_u(x,\y)$ and $t_u(x,\y)$ must be constantly
equal to some $c\in\nn$ and the other one must take the value $\8$ at $(\8,\9)$
(they cannot both take the value $\8$ at $(\8,\9)$, and if they are both
constant, then the constants must be different and $S$ will be empty,
contradicting the assumption).
Without loss of generality, assume $t_u(x,\y)=c$. We can represent $s_u(x,\y)$
as $\sum_{i=0}^mx^ip_i(\y)$, where $m\in\nn$ and each $p_i(\y)$ is a term with
variables only from $\y$. We claim that there exists $v$ such that $p_v(\y)$
is not a constant polynomial. Suppose this is not the case, which
means $s_u(x,\y)=\sum_{i=0}^ma_ix^i$, for some $a_0,\dots,a_m\in\N$.
Because $S$ is infinite, there exists $d\in S$ such that $d>c$. We know
$\N\models B(d)$, so in particular, $s_u(d,\y)=\sum_{i=0}^ma_id^i=c$.
This implies that $a_0=c$ and for all $i>0$, $a_i=0$. Thus $s_u(x,\y)=c$,
which leads to a contradiction as $s_u(\8,\9)=\8$. Hence our assumption was
false and there exists $v$ such that $p_v(\y)$ is not a constant polynomial.
Let $f:S\to(\N)^{n+1}$ be a function such that for every $a\in S$,
$\N\models \bigwedge_js_{j}(a,f(a))=t_{j}(a,f(a))$.
The following two cases can happen:
\begin{enumerate}
\item $\range(f)$ is finite:

	In this case there exists $\b\in \range(f)$ such that
	$\{a\in\nn|\N \models \bigwedge_js_{j}(a,\b)=t_{j}(a,\b)\}$ is infinite,
	hence we can use the base step and the proof is complete.
\item $\range(f)$ is infinite:

	For every natural numbers $1\le l\le n+1$ and $ 0\le l'\le c$, define
	$W_{l,l'}=\{\b\in\range(f)|\b_l=l'\}$, and let $W=\bigcup_{l,l'}W_{l,l'}$.
	We claim that at least one of $W_{l,l'}$'s must be infinite. Suppose this
	is not the case. By the assumption, $\range(f)\setminus W$ is infinite.
	This implies that there exists a nonzero $a\in S$ such that
	$f(a)\notin W$. But this leads to a contradiction, because $a>0$ and
	$(f(a))_i>c$ for every $1\le i\le n+1$, hence $a^vp_v(f(a))>c$,
	which implies $s_u(a,f(a))>c$. Therefore for some $h,h'$, $W_{h,h'}$ is
	infinite. Let $B'(x)$ be the formula obtained by removing the existential
	quantifier over $\y_h$ in $B(x)$ and substituting every free occurrence of
	$\y_h$ by $h'$. By the fact that $W_{h,h'}$ is infinite we get
	$\{a\in\nn|\N\models B'(a)\}$ is also infinite. Note that $B'(x)$ has
	less $\y$ variables, so by the induction hypothesis $\N\models B'(\8)$,
	which implies $\N\models B(\8)$.\qedhere
\end{enumerate}
\end{proof}

\begin{lemma}
\label{3d17}
$ $
\begin{enumerate}
\item $\N$ is a (classical) model of $\IEx$.
\item $\K$ is a model of $\BA$.
\end{enumerate}
\end{lemma}
\begin{proof}
$ $
\begin{enumerate}
\item It is easy to check that $\N$ satisfies Robinson's axioms and thus we
only need to show that it satisfies the induction axiom. Suppose $A(x,\y)$
is a $\Ex$ formula and $\N\models \all x\y \, (A(x,\y)\ra A(Sx,\y))$.
If $\N\models A(0,\b)$ for some $\b \in\N$,
then $\{a\in \nn|\N\models A(a,\b)\}=\nn$ and
by Lemma \ref{3l100} $\N\models A(\8,\b)$.
Hence $\N\models \all x\y \, (A(x,\y)\ra A(Sx,\y))
\Ra\all x\y \, (A(0,\y)\ra A(x,\y))$.
\item By Theorem 2.18 in \cite{AH}, a Kripke model with a single
irreflexive node is a model of $\BA$ iff it classically satisfies $\IEx$.
Use this together with the previous item.\qedhere
\end{enumerate}
\end{proof}

\begin{lemma}
\label{3l18}
Let $A$ be a formula in the language of arithmetic.
\begin{enumerate}
\item
If $A$ is geometric and $\nn\models A$ then $\N\models A$.
\item
If $\nn\models \ex y\:(y>x \land A)$, then $\K\Vdash A[y/\8]$.
\end{enumerate}
\end{lemma}
\begin{proof}
$ $
\begin{enumerate}
\item
We use induction on the number of free variables of $A$. If $A$ is a
sentence, then since it is an existential sentence, and $\nn$ is a
classical substructure of $\N$, we get $\N\models A$. If $\x y$ is the
sequence of variables free in $A$, then by induction hypothesis, for
every $b\in\nn$ we have $\N\models A[y/b]$. This shows that for every
$\a\in\N$, $\{b\in\nn|\N\models A[\x/\a][y/b]\}=\nn$. Thus by Lemma
\ref{3l100} we have $\N\models A$.
\item
By Proposition \ref{2p2} we have $\nn\models \ex y\:(y>x \land\gp A)$.
By item 1 we get $\N\models \ex y\:(y>x \land\gp A)$, which shows that
$\N\models\gp A[y/\8]$, and thus $\K\Vdash\gp A[y/\8]$. Since $\K$
consists of just one irreflexive node, we have $\K\Vdash\gp A\Lr A$,
which yields what is desired.\qedhere
\end{enumerate}
\end{proof}

\begin{definition}
\label{3d19} A set $\A\subseteq \nn$ is
\emph{provably decidable} in a theory $\T$ if
$\chi_\A$, i.e., its characteristic function, is provably
total recursive in $\T$.
\end{definition}
\begin{theorem}
\label{3t20} The provably decidable sets of $\BA$ are exactly the
sets $\A\subseteq \nn$ that are either finite or
co-finite (i.e. $\cmt\A=\nn\setminus\A$ is finite).
\end{theorem}
\begin{proof}
Let $\A$ be a provably decidable set in $\BA$, so
$\chi_\A$ is a provably total function in $\BA$. Then
there exists a $\Ex$ formula $ A(x,y)$ such that
\begin{itemize}
\item $\BA \vdash \es(A)$,
\item $\BA\vdash \us(A)$,
\item $\nn\models A(a,\chi_\A(a))$, for every $a\in\nn$.
\end{itemize}
If both $\A$ and $\cmt\A$ are infinite, we have
\begin{itemize}
\item $\nn\models \ex y\:(y>x \land A(x,0))$
\item $\nn\models \ex y\:(y>x \land A(x,1))$.
\end{itemize}
Then by Lemma \ref{3l18}, $\K \Vdash A(\8,0)$ and $\K\Vdash A(\8,1)$,
which show that $\BA\nvdash\us(A)$ and this
leads to a contradiction. Hence either $\A$ or $\cmt\A$ is finite.\\
Now suppose $\A\subseteq \nn$ is finite. Then
$\A$ has a maximum element, say $M$. Define the formula
$ A(x,y)\equiv (x>M \land y=0) \lor\bigvee\limits_{i=0}^M
(x=i \land y=\chi_\A(i))$. It is easy to
see that $ A(x,y)$ is $\Sig$ and defines $\chi_\A$ in $\BA$.
A similar argument works when $\cmt\A$ is finite.
\end{proof}

\begin{remark}\label{rm1}
\emph{Consider the primitive recursive function
$\Even(n)=\begin{cases}1 & n=2k\\0 & n=2k+1\end{cases}$.
If it is provably total recursive in $\BA$,
then $\A=\{a\in \nn | \Even(a)=1 \}$ is a
provably decidable set in $\BA$. Then by Theorem \ref{3t20},
$\A$ is either finite or co-finite, which leads to a
contradiction in either case. Hence $\Even(n)$ is not provably
total in $\BA$, which shows that $\PTRF(\BA)\ne\PR$.}
\end{remark}

One of the most important consequences of Theorem \ref{3t20} is
that given any (definable) pairing function, the corresponding projection
function is not provably total recursive in $\BA$. In the following Corollary,
think of $C(x,y,z)$ and $D(x,y)$ as formulas defining the graphs of a pairing
function and projection on the first entry, respectively.
\begin{corollary}
\label{3c22} There are no formulas $C(x,y,z)$ and
$D(x,y)$ with presented free variables such that:
\begin{enumerate}
\item $\nn \models \ex z \, C(x,y,z)$,
\item $\nn \models C(x,u,z) \land C(x,v,z)\Ra u=v$,
\item $\nn \models C(x,y,z)\Ra D(z,x)$,
\item $\BA \vdash \es(D)$,
\item $\BA \vdash \us(D)$.
\end{enumerate}
\end{corollary}
\begin{proof}
Let $\A_n=\{ a\in \nn| \nn\models D(a,n)\}$.
By (1) and (3) $\A_0$ and $\A_1$
are not empty and because of (2) they are infinite.
Define $B(x,y)\equiv\ex z\:(D(x,z)\land((z=0\land y=1)\lor(z>0\land y=0)))$.
Then by (4) and (5) we have:
\begin{itemize}
\item $\BA\vdash \es(B)$,
\item $\BA \vdash \us(B)$,
\item $\nn\models B(a,\chi_{\A_0}(a))$, for every $a\in\nn$.
\end{itemize}
That means $\A_0$ is a provably decidable set in
$\BA$, and then by Theorem \ref{3t20}, $\A_0$ is either finite
or co-finite. But $\A_0$ and $\A_1\subseteq\cmt\A_0$
are infinite which leads to a contradiction.
\end{proof}

\begin{corollary}\label{cCutOff}
The cut-off subtraction is not provably total in $\BA$.
\end{corollary}
\begin{proof}
Suppose the cut-off subtraction is defined in $\BA$ by a formula
$A(x,y,z)$, which by Theorem \ref{2t9} can be assumed to be geometric.
By definition of the cut-off subtraction we have:
\begin{itemize}
\item ${\nn}\models \ex y(y>x \land A(y,y,0))$,
\item ${\nn}\models \ex y(y>x \land A(Sy,y,1))$.
\end{itemize}
Hence by Lemma \ref{3l18}, $\K\Vdash A(\8,\8,0)$ and also
$\K\Vdash A(S\8,\8,1)$. By the fact that $\K\Vdash S\8=\8$,
we have $\K\nVdash \us(A)$ which shows $\BA\nvdash \us(A)$
and this leads to a contradiction. Hence the cut-off subtraction
is not provably total in $\BA$.
\end{proof}

\begin{corollary}
\label{prime}
Suppose $P(x)$ is a formula that defines prime numbers
(or even an infinite subset of prime numbers). 
Then $\BA\nvdash P(x)\land y|x\Ra y=1\lor y=x$, where
$s|t\equiv\ex z(s\cdot z=t)$, with $z$ being a fresh variable.
\end{corollary}
\begin{proof}
Because $\nn\models \ex y(y>x \land P(y))$, by Lemma \ref{3l18},
we have $\K\Vdash P(\infty)$. Moreover, $\K\Vdash 2|\infty$,
and hence $\K\nVdash P(\infty)\land 2|\infty \Ra 2=1\lor 2=\infty$.
Thus $\BA\nvdash P(x)\land y|x\Ra y=1\lor y=x$.
\end{proof}

We close the discussion about the provably total functions of $\BA$ by
showing that the predecessor function, defined with $\pd(0)=0$ and
$\pd(S(x))=x$, is provably total recursive in $\BA$. \emph{This will give
an explicit instance of primitive recursion applied to a provably
total recursive function of $\BA$ resulting in a function not
provably total in $\BA$. In other words, it shows that
the class of provably total recursive functions of $\BA$ is not closed
under primitive recursion.} Note that the cut-off subtraction can
be defined by iterating the predecessor function: $x\dotminus0=x$
and $x\dotminus S(y)=\pd(x\dotminus y)$. In the next proposition,
note that since $\nn\models\BA$, the formula considered is actually
defining the predecessor function.

\begin{proposition}
\label{pd}
There is a geometric open formula $\p(x,y)$ such that:
\begin{itemize}
\item
$\BA\vdash\p(0,0)$
\item
$\BA\vdash\p(Sx,x)$
\item
$\BA\vdash\es(\p)$
\item
$\BA\vdash\us(\p)$
\end{itemize}
\end{proposition}
\begin{proof}
Let $\p(x,y)\equiv(x=0\land y=0)\lor x=Sy$. The first two items
follow immediately. The third item follows from
$\BA\vdash x=0\lor\ex y\,(x=Sy)$. The last item is a cosequence of
$\BA\vdash Sy=0\Ra\bot$ and $\BA\vdash Sy=Sz\Ra y=z$.
\end{proof}

In the remaining part of this subsection, we analyze $\BQC$ and $\BA$
on a level beyond what can be done inside $\GQC$ and $\GA$.
The main method we have used until now ignores the additional expressive power
of $\BA$ in comparison to $\GA$. To emphasize on this aspect of $\BA$,
let's consider \emph{semi-geometric} formulas, in which every subformula
of the form $\all\x(A\ra B)$ is such that $A$ and $B$ are geometric. In other
words, semi-geometric formulas are those in which there are no nested
implications/universal quantifications. A sequent $A\Ra B$ is called
\emph{semi-geometric} whenever both $A$ and $B$ are semi-geometric formulas.
A rule is called \emph{semi-geometric} whenever its lower sequent and all its
upper sequents are semi-geometric. We show that studying
semi-geometric formulas can give us some results on $\BA$
which are similar to, yet different than, what we have already proven
about $\BA$ using geometric formulas. To give concrete examples, we
consider certain modifications of the notion of provably total functions.
These modifications are equivalent to what we have previously considered,
if one is working over $\HA$ or $\PA$, but not over $\BA$. We first
show that semi-geometric formulas give us ways to recover some properties
of provably total functions of $\BA$ in the modified senses, which are
analogous to our previous results, yet cannot be derived by only looking at
geometric formulas. We then consider the weaker theory $\BAw$, and apply
both methods (involving geometric and semi-geometric formulas) to
study its provably total functions in different senses. Unlike
what was obtained for $\BA$, the results on $\BAw$ will \emph{not} be
all similar to one another, which gives another example of the
contrast between basic arithmetic and geometric arithmetic.

\begin{definition}
\label{2d12}
For a formula $A$, the \emph{semi-geometric part} of $A$ is denoted
by $\sgp A$ and is defined recursively as follows:
\begin{itemize}
\item $\sgp A\equiv A$ if $A$ is prime,
\item $\sgp A\equiv\sgp B\circ\sgp C$ if $A$ is of the form
$B\circ C$ and $\circ$ is $\lor$ or $\land$,
\item $\sgp A\equiv\ex u\,\sgp B$ if $ A$ is of the form $\ex u\,B$,
\item $\sgp A\equiv\all\x(\gp B\ra\gp C)$ if $A$ is of the
form $\all\x(B\ra C)$.
\end{itemize}
$\sgp{(A\Ra B)}$ is defined as $\sgp A\Ra\sgp B$.
For a rule $R$, $\sgp R$ is defined as the rule obtained by replacing
any of the upper and lower sequents with its semi-geometric part.
\end{definition}
\begin{proposition}
\label{2c13}
$ $
\begin{enumerate}
\item
For any formula $A$, $\sgp A$ is semi-geometric.
\item
For any formula $A$, $\gp{\left(\sgp A\right)}=\gp A$.
\item
If $A$ is semi-geometric, $\sgp A=A$.
\item
If $A$ is semi-geometric, $\BQC\vdash A\Ra\gp A$.
Moreover, it can be so that only semi-geometric formulas appear
in the derivation of $A\Ra\gp A$ in $\BQC$.
\end{enumerate}
\end{proposition}
\begin{proof}
Straightforward by induction on the complexity of $A$.
\end{proof}
\begin{proposition}
\label{2c14}
Let $\Gamma$ be a set of geometric sequents.
\begin{enumerate}
\item
If $\BQC+\Gamma\vdash\alpha$ then $\BQC+\Gamma\vdash\sgp\alpha$.
Moreover, it can be so that only semi-geometric formulas appear in the
derivation of $\sgp\alpha$ from $\Gamma$ in $\BQC$.
\item
If $\BA+\Gamma\vdash\alpha$ then $\BA+\Gamma\vdash\sgp\alpha$.
Moreover, it can be so that only semi-geometric formulas appear in the
derivation of $\sgp\alpha$ from $\Gamma$ in $\BA$.
\end{enumerate}
\end{proposition}
\begin{proof}
$ $
\begin{enumerate}
\item
We prove the claim by a structural induction on the derivation of $\alpha$
from $\Gamma$ in $\BQC$. If $\alpha$ is an axiom, then $\sgp\alpha$ is a
semi-geometric instance of the same axiom. If $\alpha\in\Gamma$, then $\alpha$
is geometric and thus equal to its semi-geometric parts, which means
$\sgp\alpha\in\Gamma$. Thus the base case is proved. Now suppose that
$\alpha$ is proved by the rule 19. Then $\alpha$ is of the form
$A\Ra\all\x(B\ra C)$, and we have $\BQC+\Gamma\vdash A\land B\Ra C$
from the immediate sub-derivation. So by Proposition \ref{2p5},
$\GQC+\Gamma\vdash\gp A\land\gp B\Ra\gp C$ and by applying rule 19,
$\BQC+\Gamma\vdash\gp A\Ra\all\x(\gp B\ra\gp C)$. Then by item 4 of
Proposition \ref{2c13}, $\BQC+\Gamma\vdash\sgp\alpha$, with all the formulas
appearing in the derivation being semi-geometric. If $R$ is any other
rule of $\BQC$, $\sgp R$ is a semi-geometric instance of the same rule,
and the proof of the induction step is complete.
\item
The proof is similar to the previous item. One just needs to additionally
note that the axioms 20-25 are all geometric, and the semi-geometric part
of the induction axiom and the rule of induction are semi-geometric instances
of the same axiom and rule, respectively.\qedhere
\end{enumerate}
\end{proof}

The uniqueness sequent related to a formula $A$ can be defined in
different ways which are very close in meaning from the intuitionistic
and/or classical point of view, yet not so in the context of basic logic.
We list some examples of alternative definitions for $\us(A)$
in the following. Taking any of these definitions results in
a different notion of provably total functions of a theory.
\begin{itemize}
\item $\us_0(A)\equiv\all\x uv(A[y/u]\land A[y/v]\ra u=v)$,
where $\x y$ consists of all the free variables of $A$;
\item $\us_1(A)\equiv A[y/u]\land A[y/v]\ra u=v$;
\item $\us_2(A)\equiv\neg u=v\Ra\neg A[y/u]\lor\neg A[y/v]$;
\item $\us_3(A)\equiv A[y/u]\Ra A[y/v]\ra u=v$;
\item $\us_4(A)\equiv A[y/u]\land A[y/v]\Ra\top\ra(\top\ra u=v)$.
\end{itemize}
For $i=0,\dots,4$ and a theory $\T$, one can consider $\PTF_i(\T)$ and
$\PTRF_i(\T)$ by replacing $\us$ with $\us_i$ in the definition of
$\PTF(\T)$ and $\PTRF(\T)$, respectively.
By functionality and faithfulness of $\BA$ (Proposition
\ref{functional-well-formed} and Corollary \ref{bafaith}),
provability of $\us_0(A)$ is equivalent to provability of $\us(A)$ in $\BA$.
Consequently we get the same results regarding $\PTF_0(\BA)$ and
$\PTRF_0(\BA)$. We use the above Proposition to obtain some results about
$\PTF_i(\BA)$ and $\PTRF_i(\BA)$ for $i=1,2,3$.
However, the case of $\us_4$ may need a different treatment, as
$\sgp{(\top\ra(\top\ra u=v))}=\top\ra\top$, and the technique fails
to deliver what is required.

In the course of the proof of the next Theorem, we use the following Lemma
about $\ISig$. The proof of the Lemma is based on well-known methods
applicable in $\ISig$: introducing a new variable as a bound for
existentially quantified variables in a $\Sig$ formula to get a $\Dz$
variant, and using a pairing function to encode two variables into one.

\begin{lemma}
\label{isig-tot}
Let $\T$ be a theory extending $\ISig$ and $A(\x,y)$ be a $\Sig$ formula
such that $\T\vdash\es(A)$. Then there exists a $\Sig$ formula $B(\x,y)$
such that
\begin{itemize}
\item
$\T\vdash\es(B)$;
\item
$\T\vdash\us(B)$;
\item
$\T\vdash B\Ra A$.
\end{itemize}
\end{lemma}
\begin{proof}
Assume that $A(\x,y)$ is of the form $\ex\z D(\x,y,\z)$ where
$D$ is a $\Dz$ formula and $\z=(z_1,\dots,z_k)$. Let $C(x,y,z)$ be a
$\Dz$ formula which defines the Cantor pairing function
$\langle x,y\rangle=\frac12(x+y)(x+y+1)+y$ in $\T$ and moreover
\begin{itemize}
\item $\T\vdash\ex xyC(x,y,z)$;
\item $\T\vdash C(x,y,z)\land C(u,v,z)\Ra x=u\land y=v$.\footnote
{We can take $C(x,y,z)\equiv2z=(x+y)\cdot S(x+y)+2y$, and
then formalize the usual arguments for the desired properties
of the Cantor pairing function (totality and uniqueness properties,
as well as being bijective) inside $\ISig$.}
\end{itemize}
Define
$$E(\x,w)\equiv\ex uv\z(u\le w\land v\le w\land
\bigwedge_{i=1}^kz_i<v\land C(u,v,w)\land D(\x,u,\z))$$
and let $F(\x,w)\equiv E(\x,w)\land\all z(z<w\ra\neg E(\x,z))$.
Note that $F(\x,w)$ is provably equivalent in $\T$ to a $\Dz$ formula.
As $\T\vdash A(\x,y)\Ra\ex wE(\x,w)$ and $\T\vdash\es(A)$, we have
$\T\vdash\es(F)$ by the fact that $\ISig$ proves the least number principle
for $\Dz$ formulas. Also, by the way $F$ has been defined, $\T\vdash\us(F)$.
Finally, define $B(\x,y)\equiv\ex vw(C(y,v,w)\land F(\x,w))$.
It is easy to see that $B$ is provably equivalent in $\T$ to a $\Sig$ formula
and $\T\vdash B(\x,y)\Ra A(\x,y)$. The provability of the existence and
uniqueness sequents of $B$ follow from those of $F$.
\end{proof}

\begin{theorem}
\label{weaker-ptf}
Let $\Gamma$ be a set of geometric sequents such that $\nn\models\Gamma$.
\begin{enumerate}
\item
For $i=1,2$, if $\BA+\Gamma\vdash\us_i(A)$ then
$\BA+\Gamma\vdash\us_i(\gp A)$.
\item
If $\BA+\Gamma\vdash\us_3(A)$ then $\BA+\Gamma\vdash\us_3(\sgp A)$.
If additionally $\BA+\Gamma\vdash\es(A)$ then $\BA+\Gamma\vdash\es(\sgp A)$,
$\BA+\Gamma\vdash\us_1(\gp A)$ and $\BA+\Gamma\vdash\us_3(\gp A)$.
\item
For $i=1,2,3$,
$\PTF_i(\BA+\Gamma)=\PTRF_i(\BA+\Gamma)\subseteq\PTRF(\ISig+\Gamma)$.
\end{enumerate}
\end{theorem}
\begin{proof}
$ $
\begin{enumerate}
\item
Apply Proposition \ref{2c14} and note that
$\sgp{(\us_i(A))}\equiv\us_i(\gp A)$.
\item
If $\BA+\Gamma\vdash\us_3(A)$ then by Proposition \ref{2c14} we have
$\BA+\Gamma\vdash\sgp{(\us_3(A))}\equiv\sgp A[y/u]\Ra\gp A[y/v]\ra u=v$,
which together with item 4 of Proposition \ref{2c13} implies
$\BA+\Gamma\vdash\us_3(\sgp A)$. Assume that we additionally have
$\BA+\Gamma\vdash\es(A)$. First, note that by Proposition \ref{2c14},
we must have $\BA+\Gamma\vdash\es(\sgp A)$. This together with
$\BA+\Gamma\vdash\sgp{(\us_3(A))}$ shows that
$\BA+\Gamma\vdash\ex y(\gp A[y/u]\ra y=u)$. Consequently
$\BA+\Gamma\vdash\ex y(\gp A[y/u]\land\gp A[y/v]\ra y=u\land y=v)$,
which implies $\BA+\Gamma\vdash\us_1(\gp A)$. Also, since
$\BA+\Gamma\vdash\gp A[y/u]\Ra\gp A[y/v]\ra\gp A[y/u]\land\gp A[y/v]$,
we can conclude
$\BA+\Gamma\vdash\us_3(\gp A)\equiv\gp A[y/u]\Ra\gp A[y/v]\ra u=v$.
\item
Suppose $f:\nn^k\to\nn$ is in $\PTF_i(\BA+\Gamma)$; i.e. there exists
a formula $A(\x,y)$ such that:
\begin{itemize}
\item $\BA+\Gamma\vdash\es(A)$;
\item $\BA+\Gamma\vdash\us_i(A)$;
\item $\nn\models A(\a,f(\a))$, for every $\a\in\nn^k$.
\end{itemize}
Then, by the previous items and item 3 of Proposition \ref{2p2}, we have
\begin{itemize}
\item $\BA+\Gamma\vdash\es(\gp A)$;
\item $\BA+\Gamma\vdash\us_i(\gp A)$.
\end{itemize}
As $\nn\models\Gamma$, this means that $\gp A$ also defines a total function,
which by item 3 of Proposition \ref{2p2} must be the same as $f$.
Thus, $f$ is in $\PTRF_i(\BA+\Gamma)$, which shows
$\PTF_i(\BA+\Gamma)=\PTRF_i(\BA+\Gamma)$.

Now, note that $\BA+\Gamma\vdash\es(\gp A)$ implies
$\IEx+\Gamma\vdash\es(\gp A)$, using Lemma \ref{2l6.02}.
Thus $\ISig+\Gamma\vdash\es(\gp A)$, and by Lemma \ref{isig-tot},
we get a $\Sig$ formula $B(\x,y)$ such that
\begin{itemize}
\item $\ISig+\Gamma\vdash\es(B)$;
\item $\ISig+\Gamma\vdash\us(B)$;
\item $\ISig+\Gamma\vdash B\Ra\gp A$.
\end{itemize}
Thus, $B(\x,y)$ defines a total function, that must be the same as $f$.
Hence $f$ is in $\PTRF(\ISig+\Gamma)$, which completes the proof.
\end{enumerate}
\end{proof}

\begin{remark}
\label{sharpening}
The case $\Gamma=\varnothing$ of Theorem \ref{weaker-ptf} and the previous
observations show that $\PTF_i(\BA)\subseteq\PR$ for $i=0,1,2,3$,
by Theorem \ref{isig}. One can get the results $\PTF(\BA)\subseteq\PR$
and $\PTF_i(\BA)\subseteq\PR$ for $i=0,\dots,4$, by an
analysis of provability of the existence sequent, and completely
disregard the uniqueness sequent, as is done in \cite{S03}. See Theorem
\ref{3t14}, where we use that method for the even stronger theory $\EBA$.
\end{remark}

The importance of taking the provably total functions of $\BA$ in the sense
of $\us_1$ comes from the fact that the cut-off subtraction becomes provably
total. For the next proposition, note that since $\nn\models\BA$,
the formula considered is actually defining cut-off subtraction.

\begin{proposition}
\label{cut-off-ptf}
There is a geometric open formula $\co(x,y,z)$ such that:
\begin{itemize}
\item
$\BA\vdash x<y\Ra\co(x,y,0)$
\item
$\BA\vdash x=y+z\Ra\co(x,y,z)$
\item
$\BA\vdash\es(\co)$
\item
$\BA\vdash\us_1(\co)$
\end{itemize}
\end{proposition}
\begin{proof}
Let $\co(x,y,z)\equiv(x<y\land z=0)\lor x=y+z$.
The first two items follow immediately.
The third item is a consequence of $\BA\vdash x<y\lor y\le x$.
To prove the last item, we first use the rule of induction to show that
$\BA\vdash y+u=y+v\ra u=v$, noting that $\BA\vdash 0+u=0+v\ra u=v$ and
$\BA\vdash y+u=y+v\ra u=v\Ra Sy+u=Sy+v\ra u=v$.
Then, we use the rule of induction to show that $\BA\vdash\neg x<x$,
which together with the previous result, implies that
$\BA\vdash\co(x,y,u)\land\co(x,y,v)\ra u=v$.
\end{proof}

Next, we show how the analysis can be done for the weakened basic
arithmetic $\BAw$, in which the rule of induction is dropped from $\BA$.
For this, consider the \emph{Weakened Geometric Arithmetic}, $\GAw$,
in which the rule of geometric induction is dropped from $\GA$.
Note that $\GAw$ is a theory even weaker than Robinson arithmetic $\Q$;
not only its language is restricted to geometric formulas, but also it
lacks the axiom $x=0\lor\ex y\,x=Sy$. An important property of $\GAw$
is that its arithmetical axioms contain only (geometric) open formulas.
As geometric logic is a subtheory of classical logic, this lets us 
use well-known classical methods for analyzing $\GAw$,
most notably the famous Herbrand's theorem. In the next Proposition,
we use these facts to characterize the provably total functions of $\GAw$,
which will become handy in our investigation of the provably total
functions of $\BAw$.

\begin{proposition}
\label{gaw-prov-tot}
Let $\Gamma$ be a set of sequents of (geometric) open formulas.
\begin{enumerate}
\item
If $\GAw+\Gamma\vdash\ex y\ex\z A(\x,y,\z)$ for some (geometric) open formula
$A(\x,y,\z)$, then there exists a (geometric) open formula $B(\x,y)$ and
a finite list $s_1(\x),\dots,s_r(\x)$ of terms such that
$\GAw+\Gamma\vdash B(\x,y)\Ra\ex\z A(\x,y,\z)$ and
$\GAw+\Gamma\vdash\bigvee_{i=1}^rB(\x,s_i(\x))$.
\item
If $\nn\models\Gamma$, then provably total functions of $\GAw+\Gamma$
are definable in $\GAw+\Gamma$ by (geometric) open formulas.
\end{enumerate}
\end{proposition}
\begin{proof}
$ $
\begin{enumerate}
\item
As already discussed in the proof of Lemma \ref{2l6.02}, $\GAw$ can be
fomalized by the fragment of $\LK$ without the rules for implication
and universal quantification, together with the axioms of equality
and arithmetic (see the Appendix), in all of which only quantifier-free
formulas appear. By Herbrand's theorem (Theorem 2.5.1 of \cite{B98}),
this implies that if $\GAw+\Gamma\vdash\ex y\ex\z A(\x,y,\z)$, then
there are finite lists $s_1(\x),\dots,s_r(\x)$ of terms and
$\t_1(\x),\dots,\t_r(\x)$ of sequences of terms with the same length as $\y$,
such that $\GAw+\Gamma\vdash\bigvee_{i=1}^r A(\x,s_i(\x),\t_i(\x))$.
Letting $B(\x,y)\equiv\bigvee_{i=1}^r A(\x,y,\t_i(\x))$, it is
straightforward to verify that $B(\x,y)$ has the desired properties.
\item
Let $f:\nn^n\to\nn$ be defined in $\GAw+\Gamma$ by a (geometric) formula
$A(\x,y)$. Without loss of generality, $A(\x,y)$ can be considered to be
$\Ex$. Choosing $B(\x,y)$ as in the previous item, we can see that
$\GAw+\Gamma$ proves the existence and uniqueness sequents of $B(\x,y)$,
since it does so for $A(\x,y)$. Since $\nn\models\GAw+\Gamma$,
this means that $B(\x,y)$ defines a function $g:\nn^n\to\nn$.
By the provability of the uniqueness sequent of $B(\x,y)$ and the fact that
$\GAw+\Gamma\vdash B(\x,y)\Ra A(\x,y)$, it turns out
that $g$ is the same function as $f$, and thus $f$ is definable in
$\GAw+\Gamma$ by the (geometric) open formula $B(\x,y)$.
\end{enumerate}
\end{proof}

\begin{proposition}
\label{baw-geo}
Let $\Gamma$ be a set of sequents.
\begin{enumerate}
\item
If $\BAw+\Gamma\vdash\alpha$ then $\GAw+\gp\Gamma\vdash\gp\alpha$.
Consequently, $\gp{\left(\BAw\right)}\vdv\GAw$,
and $\BAw$ is conservative over $\GAw$.
\item
Assuming that every sequent in $\Gamma$ is geometric,
if $\BAw+\Gamma\vdash\alpha$ then $\BAw+\Gamma\vdash\sgp\alpha$.
Moreover, it can be so that only semi-geometric formulas appear in the
derivation of $\sgp\alpha$ from $\Gamma$ in $\BAw$.
\end{enumerate}
\end{proposition}
\begin{proof}
$ $
\begin{enumerate}
\item
Same as the proof of item 2 of Proposition \ref{2p5},
without considering the rule of induction.
\item
Same as the proof of item 2 of Proposition \ref{2c14},
without considering the rule of induction.\qedhere
\end{enumerate}
\end{proof}

\begin{proposition}
\label{baw-prov-tot}
Let $\Gamma$ be a set of sequents of geometric open formulas.
\begin{enumerate}
\item
If $\BAw+\Gamma\vdash\ex y\ex\z A(\x,y,\z)$ for some geometric open formula
$A(\x,y,\z)$, then there exists a geometric open formula $B(\x,y)$ and
a finite list $s_1(\x),\dots,s_r(\x)$ of terms such that
$\BAw+\Gamma\vdash B(\x,y)\Ra\ex\z A(\x,y,\z)$ and
$\BAw+\Gamma\vdash\bigvee_{i=1}^rB(\x,s_i(\x))$.
\item
If $\nn\models\Gamma$, then provably total functions of $\BAw+\Gamma$
(in the sense of the original formulation for the uniqueness sequents)
are definable in $\BAw+\Gamma$ by geometric open formulas.
\end{enumerate}
\end{proposition}
\begin{proof}
Combine Proposition \ref{gaw-prov-tot} with item 1 of
Proposition \ref{baw-geo}.
\end{proof}

\begin{theorem}
\label{baw-ptf}
Let $\Gamma$ be a set of geometric sequents such that $\nn\models\Gamma$.
\begin{enumerate}
\item
If $\BAw\vdash\us(A)$ then $\GAw\vdash\us(\gp A)$;
$\PTRF(\BAw+\Gamma)=\PTF(\BAw+\Gamma)=\PTF(\GAw+\Gamma)=\PTRF(\GAw+\Gamma)$.
\item
For $i=0,1,2$, if $\BAw+\Gamma\vdash\us_i(A)$ then
$\BAw+\Gamma\vdash\us_i(\gp A)$;
$\PTF_i(\BAw+\Gamma)=\PTRF_i(\BAw+\Gamma)\subseteq\PTRF(\IEx+\Gamma)$.
\item
If $\BAw\vdash\us_3(A)$ then $\BAw\vdash\us_3(\sgp A)$;
$\PTRF_3(\BAw)\subseteq\PTRF(\IEx+\Gamma)$.
\end{enumerate}
\end{theorem}
\begin{proof}
The first item can be proven by an argument similar to the proof of
Theorem \ref{2t9}, only using item 1 of Proposition
\ref{baw-geo} instead of Proposition \ref{2p5}.
The other items can be proven by an argument similar to the proof of
Theorem \ref{weaker-ptf}, only using item 2 of Proposition
\ref{baw-geo} instead of Proposition \ref{2c14}. Just note that
this time the induction formulas are only appearing in the induction
axioms, which means that they are geometric, as the induction axiom
appearing in the derivation is supposed to be semi-geometric.
\end{proof}

Note that we have also included $i=0$ in the second item
of the above Theorem. In fact, the provabilities of
$\us_0(A)$ and $\us(A)$ are not equivalent in
$\BAw$, unlike $\BA$. That is because $\BAw$ is not faithful,
which we now set out to prove.

\begin{lemma}
\label{wll-tt}
Let $\T$ and $\T'$ be theories in a first-order language such that
\begin{itemize}
\item
$\T'$ extends $\BQC$;
\item
for any $n$ (including $0$), any seuqnce
$\all\x(A_0\ra B_0),\all\x(A_1\ra B_1),\dots,\all\x(A_n\ra B_n)$
of sentences such that
\AXC{$A_1\Ra B_1$}\AXC{$\dots$}\AXC{$A_n\Ra B_n$}\TIC{$A_0\Ra B_0$}\DP
is a rule of $\T$ (or in case $n=0$, $A_0\Ra B_0$ is an axiom of $\T$),
and any formula $A$ no free variable of which appears in $\x$,
$\T'\vdash\bigwedge_{i=1}^n\all\x(A\land A_i\ra B_i)
\Ra\all\x(A\land A_0\ra B_0)$.
\end{itemize}
Then for any seuqnce
$\all\x(A_0\ra B_0),\all\x(A_1\ra B_1),\dots,\all\x(A_n\ra B_n)$
of sentences such that $\T+\{A_i\Ra B_i\}_{i=1}^n\vdash A_0\Ra B_0$,
and any formula $A$ no free variable of which appears in $\x$,
$\T'\vdash\bigwedge_{i=1}^n\all\x(A\land A_i\ra B_i)
\Ra\all\x(A\land A_0\ra B_0)$.
\end{lemma}
\begin{proof}
The proof is essentially the same as the one showing that
a theory with a well-formed axiomatization is well-formed
(Proposition 4.13 of \cite{R98}). Note that the assumptions
on $\T'$ made explicit in the statement of this Lemma are
exactly what is needed for the proof to work.
\end{proof}

\begin{corollary}
\label{baw-vs-ba}
Let $A$ and $B$ be formulas in the language of arithmetic.
The following are equivalent:
\begin{enumerate}
\item
$\BA\vdash A\Ra B$;
\item
$\BAw\vdash\all\x(A\ra B)$ for some sequence of variables $\x$
containing all the free variables of $A$ and $B$;
\item
$\BAw\vdash\all\x(A\ra B)$ for all sequences of variables $\x$.
\end{enumerate}
\end{corollary}
\begin{proof}
To prove (2) assuming (1), apply Lemma \ref{wll-tt} with $\T=\BA$ and
$\T'=\BAw$. Note that other than the rule of induction, all the axioms
and rules of $\BA$ are already axioms and rules of $\BAw$, and the
presence of the induction axiom in $\BAw$ covers all that is needed for
applying the Lemma. (3) follows from (2) by using axiom 17, and (1)
follows from (3) by the faithfulness of $\BA$ (Corollary \ref{bafaith})
and the fact that $\BA$ extends $\BAw$.
\end{proof}

\begin{corollary}
\label{baw-ba-ptf}
$ $
\begin{enumerate}
\item
$\BAw\vdash\all\x\ex yA(\x,y)$ iff $\BA\vdash\es(A(\x,y))$.
\item
$\BAw\vdash\us_0(A)$ iff $\BA\vdash\us(A)$.
\end{enumerate}
\end{corollary}
\begin{proof}
Straightforward by Corollary \ref{baw-vs-ba}.
\end{proof}

The fact that $\BAw$ is not faithful comes from the fact that the provability
of $\all\x\ex yA(\x,y)$ in the first item of Corollary \ref{baw-ba-ptf}
cannot be strengthened to the provability of $\es(A(\x,y))$. To see this,
we need to consider a new model for $\BAw$, as in the following definition.

\begin{definition}
\label{nt}
$ $
\begin{itemize}
\item
$\Rzp$ is the classical structure of the set of non-negative real numbers.
The symbols $0$, $+$, $\cdot$ and $<$ of the language of arithmetic are
interpreted standardly. The symbol $S$ is interpreted as $Sa=a+1$
for every $a\in \Rzp$.

\item
$\kzp$ is the Kripke model with just one irreflexive
node with the structure $\Rzp$.
\end{itemize}
\end{definition}
\begin{lemma}
\label{nt-models-gaw}
$ $
\begin{enumerate}
\item
$\Rzp$ is a (classical) model of $\GAw$. $\kzp$ is a model of $\BAw$.
\item
$\Rzp\not\models x=0\lor\ex y\,x=Sy$.
$\kzp\nVdash x=0\lor\ex y\,x=Sy$.
\item
$\GAw\nvdash x=0\lor\ex y\,x=Sy$.
$\BAw\nvdash x=0\lor\ex y\,x=Sy$.
\end{enumerate}
\end{lemma}
\begin{proof}
$ $
\begin{enumerate}
\item
It is straightforward to verify that $\Rzp$ satisfies the arithmetical axioms
of $\GAw$, and hence $\Rzp\models\GAw$. As $\kzp$ is a single-node irreflexive
Kripke model, we have $\kzp\Vdash A\Lr\gp A$ for all formulas $A$,
which implies $\kzp\Vdash\BAw$.
\item
$\frac12$ is neither equal to $0$ nor to a successor of any element of $\Rzp$.
\item
By combining the previous items.
\end{enumerate}
\end{proof}

\begin{corollary}
\label{baw-not-faith}
$ $
\begin{enumerate}
\item
$\BAw$ is not faithful.
\item
$\BAw$ is neither reflexively rooted nor irreflexively rooted.
\end{enumerate}
\end{corollary}
\begin{proof}
$ $
\begin{enumerate}
\item
By Proposition \ref{pd} item 1 of Corollary \ref{baw-ba-ptf},
$\BAw\vdash\all\x\ex y\p(x,y)$. If $\BAw$ is faithful,
we must have $\BAw\vdash\es(\p)$, which contradicts
item 3 of Lemma \ref{nt-models-gaw}.
\item
By the previous item and Proposition \ref{fwr-dep}.\qedhere
\end{enumerate}
\end{proof}

Looking at the proof of item 1 of Corollary \ref{baw-not-faith},
we can see that $\PTF_0(\BAw)\subsetneq\PTRF(\BA)$. By item 2 of
Corollary \ref{baw-ba-ptf}, this is due to unprovability of some
existence sequents, and not that of uniqueness sequents. That
shows the important role of the rule of induction of $\BA$ in
the provability of totality of functions.

As a closing remark, we note that while $\BAw$ lacks rootedness,
which was the main tool for proving disjunction and existence properties
for theories like $\BA$ and $\HA$ (see Proposition \ref{fwr-dep}),
it still satisfies some versions of these properties.
The question of $\BAw$ having disjunction and existence properties
is still open to us.

\begin{proposition}
\label{baw-disj-ex}
$ $
\begin{enumerate}
\item
For all sentences $A$ and $B$, $\BAw\vdash\top\ra A\lor B$
implies $\BAw\vdash\top\ra A$ or $\BAw\vdash\top\ra B$.
\item
For all sentences $\ex xA$, $\BAw\vdash\top\ra\ex xA$ implies
$\BAw\vdash\top\ra A[x/t]$ for some cosed term $t$.
\end{enumerate}
\end{proposition}
\begin{proof}
$ $
\begin{enumerate}
\item
If $\BAw\vdash\top\ra A\lor B$, then by Corollary \ref{baw-vs-ba},
$\BA\vdash A\lor B$. By the disjunction property of $\BA$,
we have either $\BA\vdash A$ or $\BA\vdash B$, which using Corollary
\ref{baw-vs-ba} again, implies $\BAw\vdash\top\ra A$ or $\BAw\vdash\top\ra B$.
\item
If $\BAw\vdash\top\ra\ex xA$, then by Corollary \ref{baw-vs-ba},
$\BA\vdash\ex xA$. By the existence property of $\BA$,
we have $\BA\vdash A[x/t]$ for some clodes term $t$, which using
Corollary \ref{baw-vs-ba} again, implies $\BAw\vdash\top\ra A[x/t]$.
\end{enumerate}
\end{proof}
\subsection{The Provably Total Recursive Functions of Extensions of $\BA$}
In the last section, we showed that the provably total recursive functions of
$\BA$ are primitive recursive, however there are some primitive recursive
functions that are not provably total in $\BA$. In this section, we consider
three extensions of $\BA$ that their provably total recursive functions are
exactly the primitive recursive functions. The first extension is by adding
the cancellation law (the axiom $\U:x+y=x+z\Ra y=z$) to $\BA$, the second one
is obtained by adding a symbol for the cut-off subtraction to the language of
$\BA$, and the third one is the theory $\EBA$, introduced in \cite {AH}.
As is known, the cut-off subtraction is a primitive recursive function,
and it turns out that adding a symbol for it to the language of $\BA$,
and its properties as additional axioms, will result then that the provably
total recursive functions of this extension captures all the primitive
recursive functions. We show that the first two extensions of $\BA$ coincide
in some sense, see Theorem \ref{BAUC}. The theory $\EBA$ that is an extension
of $\BA$ by adding the sequent axiom $\top\ra\bot\Ra\bot$, is a stronger theory
than the previous two extensions of $\BA$, and is very close to $\HA$ in some
respects, however still weaker than that. For more details on motivations and
some properties of $\EBA$, that we will use in this paper, see \cite{AH}.

We start with adding the axiom $\U$ to our theories. The name ``$\U$"
comes from the equivalence of the cancellation law of addition and
the uniqueness sequent of the cut-off subtraction, $\us(\co)$. We take a look
at this relation in the next Lemma. We say that two sequents $\alpha$
and $\beta$ are equivalent over a theory $\T$, whenever
$\T+\alpha\vdash\beta$ and $\T+\beta\vdash\alpha$. It turns out that
$\U$ is equivalent to the only axiom of $\PAm$ that is not provable in $\BA$.

\begin{lemma}
\label{U-equiv}
The following sequents are equivalent over $\BA$.
\begin{enumerate}
\item
$\co(x,y,z)\land\co(x,y,w)\Ra z=w$
\item
$x+y=x+z\Ra y=z$
\item
$x+y=x\Ra y=0$
\item
$S(y+x)=x\Ra\bot$
\item
$x<x\Ra\bot$
\end{enumerate}
\end{lemma}
\begin{proof}
By definition of $\co$ and logical rules and axioms, 1 is equivalent to
$$(x<y\land z=0\land w=0)\lor(x<y\land z=0\land x=y+w)\lor
(x<y\land x=y+z\land w=0)\lor(x=y+z\land x=y+w)\Ra z=w\text,$$
which in particular implies provability of $x=y+z\land x=y+w\Ra z=w$ and
therefore 2. 3 follows from substituting $0$ for $z$ in 2. To get 4 from 3,
we note that $\BA$ proves $y+x=x+y$, $S(x+y)=x+Sy$ and $Sy=0\Ra\bot$.
5 is equivalent to $\ex y\,(x+Sy=x)\Ra\bot$, and it follows from
$x+Sy=x\Ra\bot$, and therefore from 4. To get 1 from 5, consider the above
equivalent for 1. It suffices to prove $\BA+x<x\Ra\bot\vdash A\Ra z=w$
for all the four $A$ that are disjuncts of the left-hand side. Clearly
$\BA\vdash z=0\land w=0\Ra z=w$. Since $\BA\vdash w=0\lor\ex u\, w=Su$,
we have $\BA\vdash y\le y+w$. Hence $\BA\vdash x<y\land x=y+w\Ra y<y$,
and therefore $\BA+x<x\Ra\bot\vdash x<y\land x=y+w\Ra z=w$, and similarly,
$\BA+x<x\Ra\bot\vdash x<y\land x=y+z\Ra z=w$. Finally, note that
$\BA\vdash z<w\lor z=w\lor z>w$. As $\BA\vdash y+z=y+w\land z<w\Ra y+z<y+z$
and $\BA\vdash y+z=y+w\land z>w\Ra y+w<y+w$, we get
$\BA+x<x\Ra\bot\vdash x=y+z\land x=y+w\Ra z=w$.
\end{proof}

\begin{definition}
\label{3d3}
For a quantifier-free formula $ A$, the \emph{geometric equivalent} and the
\emph{geometric negation} of $A$ are respectively denoted by $\gps A$ and
$\gng A$, and are defined recursively as follows:
\begin{itemize}
\item $\gps\top\equiv\top$ and $\gng\top\equiv\bot$;
\item $\gps\bot\equiv\bot$ and $\gng\bot\equiv\top$;
\item $\gps{(s=t)}\equiv s=t$ and $\gng{(s=t)}\equiv s<t \lor t<s$,
if $s$ and $t$ are terms;
\item $\gps{(s<t)}\equiv s<t$ and $\gng{(s<t)}\equiv s=t \lor t<s$,
if $s$ and $t$ are terms;
\item $\gps A\equiv\gps B\land\gps C$ and
$\gng A\equiv\gng B\lor\gng C$, if $A$ is of the form $B\land C$;
\item $\gps A\equiv\gps B\lor\gps C$ and
$\gng A\equiv\gng B\land\gng C$, if $A$ is of the form $B\lor C$;
\item $\gps A\equiv\gng B\lor\gps C$ and
$\gng A\equiv\gps B\land\gng C$, if $A$ is of the form $B\ra C$.
\end{itemize}
\end{definition}
\begin{lemma}
\label{3l4}
For every quantifier free formula $A$, $\gps A$ and $\gng A$
are geometric open formulas such that
$\IEx+\U\vdash A\Lr\gps A$ and $\IEx+\U\vdash\neg A\Lr\gng A$.
\end{lemma}
\begin{proof}
The first part of the lemma is obvious. For the second part,
we first note that $\IEx\vdash x<y\lor x=y\lor y<x$. Now we
prove the lemma by induction on complexity of $A$.
\begin{itemize}
\item If $A$ is of the form $s=t$, for some terms $s$ and $t$, then
$\gng A$ is $s<t\lor t<s$.
Hence $\IEx\vdash\neg(s=t)\Ra s<t\lor t<s$. On the other hand,
$\IEx+\U\vdash x<x\Ra\bot$, so $\IEx +\U\vdash s<t\lor t<s\Ra\neg(s=t)$.
\item If $A$ is of the form $s<t$, for some terms $s$ and $t$, then
$\gng A$ is $s=t\lor t<s$.
Hence $\IEx\vdash\neg(s<t)\Ra s=t\lor t<s$. On the other hand, since
$\IEx+\U\vdash x<x\Ra\bot$ and $\IEx+\U\vdash x<y\land y<x\Ra\bot$,
we get $\IEx+\U\vdash s=t\lor t<s\Ra\neg(s<t)$.
\item
If $A$ is of the form $B\ra C$, then $\IEx+\U\vdash B\ra C\Lr\neg B\lor C$
and $\IEx+\U\vdash\neg(B\ra C)\Lr B\land\neg C$. By
the induction hypothesis, $\IEx+\U\vdash B\ra C\Lr\gng B\lor\gps C$
and $\IEx+\U\vdash\neg(B\ra C)\Lr\gps B\land\gng C$.
\end{itemize}
It is routine to check the other cases for $A$.
\end{proof}
\begin{remark}
\label{3r5} \emph{Note that $\IEx\nvdash\neg(x=y)\Lr\gng{(x=y)}$ and
$\IEx\nvdash\neg(x=y)\Lr\gps{(\neg(x=y))}$, for the following reason.
By Lemma \ref{3d17}, $\N$ is a model of $\IEx$ such that $\N\models\8=\8$,
$\N\models\gng{(\8=\8)}$ and $\N\models\gps{(\neg(\8=\8))}$. So the axiom
$\U$ is necessary for the proof of Lemma \ref{3l4}.}
\end{remark}

\begin{lemma}
\label{3l6}
For every $\Eo$ formula $A$, there exists a
$\Ex$ formula $B$ with the same free variables as $A$,
such that $\IEx+\U\vdash A\Lr B$.
\end{lemma}
\begin{proof}
Since $A$ is $\Eo$, we have $\IEx+\U\vdash A\Lr\ex\y\,C$ for some
quantifier free formula $C$. If we define $B\equiv\ex\y\,\gps C$,
then by Lemma \ref{3l4}, $\IEx+\U\vdash A\Lr B$.
\end{proof}
We now turn to the well-known MRDP theorem. This theorem is about the
existence of Diophantine equivalents for $\Sig$ formulas in a strong enough
arithmetical theory. We will study our theories of interest in this
relation, and show which are strong enough in this sense and which
are not. Our main results on this matter are postponed until the next
Section, but we need to consider some instances here, for the purpose
of characterizing the provably total recursive functions of our theories.
A well-known classical instance of the MRDP theorem appears in the proof
of the following Theorem, which is an instance of the MRDP theorem itself.
\begin{theorem}
\label{3t7}
For every $\Sig$ formula $ A$,
there exists a $\Ex$ formula $ B$ with the same free
variables as $ A$, such that $\IEx+\U \vdash A \Lr B$.
\end{theorem}
\begin{proof}
Given a $\Sig$ formula $A$, by Corollary 4.11 in
\cite{K90}, there exists a $\Eo$ formula $A'$ with the
same free variables as $A$, such that $\IE\vdash A\Lr A'$. By
Lemma \ref{3l6}, for every $\Eo$ formula $C$ there exists
an $\Ex$ formula $C'$ such that $\IEx+\U\vdash C\Lr C'$. Also
$\IEx\vdash\all x\y\,(C'\ra C'[x/Sx])\Ra\all x\y\,(C'[x/0]\ra C')$, so
$\IEx+\U\vdash\all x\y(C\ra C[x/Sx])\Ra\all x\y(C[x/0]\ra C)$. Thus
$\IEx+\U\vdash\IE$, and hence $\IEx+\U\vdash A\Lr A'$. By Lemma \ref{3l6},
there exists an $\Ex$ formula $B$ with the same free variables as $A'$
such that $\IEx+\U\vdash A'\Lr B$. Hence $\IEx+\U\vdash A\Lr B$.
\end{proof}
\begin{corollary}
\label{3c8}
$\IEx+\U\vdv\ISig$
\end{corollary}
\begin{proof}
By Theorem \ref{3t7}, $\IEx+\U\vdash\ISig$. Also $\ISig\vdash\U$ and
$\Ex\subseteq\Sig$, so $\ISig\vdash\IEx+\U$.
\end{proof}
\begin{corollary}
\label{3c9}
The provably total recursive functions of $\IEx+\U$ are exactly the primitive
recursive functions, i.e. $\PTRF(\IEx+\U)=\PR$. Furthermore these
functions are definable in $\IEx+\U$ by $\Ex$ formulas.
\end{corollary}
\begin{proof}
Combine Corollary \ref{3c8} and Theorems \ref{isig} and \ref{3t7}.
\end{proof}
\begin{theorem}
\label{3t11}
The provably total functions of $\BA+\U$ are exactly the primitive recursive
functions, and they are definable in $\BA+\U$ by $\Ex$ formulas. Consequently,
$\PTRF(\BA+\U)=\PTF(\BA+\U)=\PR$
\end{theorem}
\begin{proof}
If a function $f:\nn^n\to\nn$ is provably total in $\BA +\U$, then by
Theorem \ref{2t9} there exists a $\Ex$ formula $A(\x,y)$ such that:
\begin{itemize}
\item $\BA+\U \vdash \es(A)$,
\item $\BA+\U \vdash \us(A)$,
\item $\nn\models A(\a,f(\a))$, for every $\a\in\nn^n$.
\end{itemize}
Then by Corollary \ref{2c6.1}:
\begin{itemize}
\item $\IEx+\U \vdash \es(A)$,
\item $\IEx+\U \vdash \us(A)$.
\end{itemize}
So $f$ is provably total recursive in $\IEx+\U$ and by Corollary \ref{3c9},
$f$ is primitive recursive. Therefore $\PTF(\BA+\U)\subseteq\PR$.
Now suppose $g:\nn^m\to\nn$ is a primitive recursive function. Then by
Corollary \ref{3c9}, there exists a $\Ex$ formula $B(\x,y)$ such that:
\begin{itemize}
\item $\IEx+\U \vdash \es(B)$,
\item $\IEx+\U \vdash \us(B)$,
\item $\nn\models B(\b,g(\b))$, for every $\b\in\nn^m$.
\end{itemize}
Thus by Corollary \ref{2c6.1}, 
\begin{itemize}
\item $\BA+\U \vdash \es(B)$,
\item $\BA+\U \vdash \us(B)$,
\end{itemize}
and hence $\PR\subseteq\PTRF(\BA+\U)$.
\end{proof}

In the next theorem, we show that the provably total recursive functions
of $\IEx$ are also exactly the primitive recursive functions.
Our method of characterizing the provably total functions of
$\IEx+\U$ did rely on the axiom $\U$, and the result was that
the provably total recursive functions of $\IEx+\U$ are definable
in $\IEx+\U$ by geometric formulas. This fails to hold in $\IEx$,
and the MRDP theorem also fails. To characterize the provably total
recursive functions of $\IEx$, we consider $\Sig$ defining formulas
that are not geometric. The idea of the proof is essentially due to
\cite{Gh}, and goes by the following lines. Given an $\Ex$
formula $\ex\z\,A$ defining a primitive recursive function in $\IEx+\U$,
we first modify $A$ by adding to it bounded versions of its uniqueness sequent
and the axiom $\U$, to get $B$. Then we take the new defining formula
for the function to be $\ex\z\,(B\lor C)$, where $C$ roughly states that
both $\U$ and $B$ are false. While $\IEx$ is not strong enough to refute
$C$, it can prove that $B$ and $C$ cannot happen together, which paves
the way for a proof of the uniqueness sequent. The proof of the existence
sequent relies on derivability of the principle of excluded middle in $\IEx$.
The fact that $\ex\z\,(B\lor C)$ defines the same function as before
comes from $B$ being true and $C$ being false in the standard model $\nn$.

\begin{theorem}
\label{3t000}
The provably total recursive functions of $\IEx$
are exactly the primitive recursive functions,
i.e. $\PTRF(\IEx)=\PR$.
\end{theorem}
\begin{proof}
Let $f:\nn^n\to\nn$ be a primitive recursive function. By \ref{3c9},
there is a geometric open formula $A(\x,y,\z)$ such that $f$ is
definable in $\IEx+\U$ by the formula $\ex\z A(\x,y,\z)$,
where $\z=(z_1,\dots,z_m)$. Without loss of generality, we can assume that
$\z$ is nonempty, because we can consider $A(\x,y,\z)\land z=z$ instead of
$A(\x,y,\z)$. Suppose $x$, $y'$, $y''$, $z'_1$, $\dots$, $z'_m$, $z''_1$,
$\dots$, $z''_m$, $u$, $v$ and $w$ are pairwise distinct fresh variables,
$\z'=(z'_1,\dots,z'_m)$, $\z''=(z''_1,\dots,z''_m)$,
$t\equiv y+z_1+\dots+z_m$, $t'\equiv y'+z'_1+\dots+z'_m$
and $t''\equiv y''+z''_1+\dots+z''_m$, and let:
\begin{gather*}
U(x)\equiv\all uvw(u+v+w\le x\land u+w=v+w\ra u=v)\text;\\
B(\x,y,\z)\equiv U(t)\land A(\x,y,\z)\land
\all y'\z'(t'\le t\land A(\x,y',\z')\ra y=y')\text;\\
C(\x,y,\z)\equiv\neg U(t)\land y=0\land
\all y'\z'(t'\le t\ra\neg B(\x,y',\z'))\text;\\
D(\x,y)\equiv\ex\z(B(\x,y,\z)\lor C(\x,y,\z))\text.
\end{gather*}
It's easy to see that $U(x)$, $B(\x,y,\z)$ and $C(\x,y,\z)$ are provably
equivalent in $\IEx$ to $\Dz$ formulas, and thus $D(\x,y)$ is provably
equivalent in $\IEx$ to a $\Sig$ formula.

Since $\IEx+\U\vdash\us(\ex\z A(\x,y,\z))$, we have 
$\IEx+\U\vdash A(\x,y,\z)\Ra\all y'\z'(A(\x,y',\z')\ra y=y')$, and
because $\IEx+\U\vdash\es(\ex\z A(\x,y,\z))$, we get 
$\IEx\vdash\all uvw(u+w=v+w\ra u=v)\Ra\ex y\ex\z B(\x,y,\z)$. Hence
$\IEx\vdash\all y\z\neg B(\x,y,\z)\Ra\ex x\neg U(x)$. Now, we note that
$$\IEx\vdash\neg U(x)\land y=0\land z_1=0\land\dots\land z_{m-1}=0
\land z_m=x\Ra\neg U(t)\land y=0\text,$$
which together with the previous result, yields
$\IEx\vdash\all y\z\neg B(\x,y,\z)\Ra\ex y\ex\z C(\x,y,\z)$. As
$\IEx\vdash\all y\z\neg B(\x,y,\z)\lor\ex y\ex\z B(\x,y,\z)$, we get
$\IEx\vdash\es(D)$.

Now, we show that $\IEx\vdash\us(D)$. For this purpose, we argue as follows.
\begin{enumerate}
\item We first show
$\IEx\vdash B(\x,y,\z)\land C(\x,y'',\z'')\Ra\bot$.
This can be proven using $\IEx\vdash t\le t''\lor t''\le t$,
$\IEx\vdash t\le t''\land C(\x,y'',\z'')\Ra\neg B(\x,y,\z)$
(by the definition of $C(\x,y'',\z'')$) and
$\IEx\vdash t''\le t\land B(\x,y,\z)\Ra U(t'')$
(by definitions of $B(\x,y,\z)$ and $U(t'')$).
\item 
Then, we show that $\IEx\vdash B(\x,y,\z)\land B(\x,y'',\z'')\Ra y=y''$.
This is true because $\IEx\vdash t\le t''\lor t''\le t$,
$\IEx\vdash t\le t''\land B(\x,y'',\z'')\land A(\x,y,\z)\Ra y=y''$
(by definition of $B(\x,y'',\z'')$) and
$\IEx\vdash t''\le t\land B(\x,y,\z)\land A(\x,y'',\z'')\Ra y=y''$
(for a similar reason). 
\item 
At last, it is easy to see that
$\IEx\vdash C(\x,y,\z)\land C(\x,y'',\z'')\Ra y=y''$, since
$\IEx\vdash C(\x,y,\z)\Ra y=0$ and $\IEx\vdash C(\x,y'',\z'')\Ra y''=0$.
\end{enumerate}
Combining these three facts, we get what was claimed.

Finally, as $\IEx+\U\vdash A(\x,y,\z)\Lr B(\x,y,\z)$ and
$\IEx+\U\vdash C(\x,y,\z)\Ra\bot$, we get
$\IEx+\U\vdash\ex\z A(\x,y,\z)\Lr D(\x,y)$. As $\nn\models\IEx+\U$,
we see that $D(\x,y)$ defines $f$. Thus $f$ is definable
in $\IEx$ by a $\Sig$ formula. Consequently $\PR\subseteq\PTRF(\IEx)$.

For the other way around, i.e. $\PTRF(\IEx)\subseteq\PR$, it is enough
to note that $\ISig \vdash\IEx$ and use Theorem \ref{isig}.
\end{proof}
\begin{corollary}
\label{3c000}
The MRDP theorem does not hold in $\IEx$, i.e. there is a $\Sig$ formula
$A$ with no geometric formula $B$ such that $\IEx\vdash A\Lr B$.
\end{corollary}
\begin{proof}
The cut-off subtraction is primitive recursive, and thus by Theorem
\ref{3t000}, it is provably total in $\IEx$. Let $A$ be a $\Sig$ formula
that defines the cut-off subtraction in $\IEx$. If there is a geometric
formula $B$ such that $\IEx\vdash A\Lr B$, then $\IEx$ proves existence
and uniqueness sequents of $B$. By Corollary \ref{2c6.1}, $\BA$ proves
these sequents as well. This contradicts Corollary \ref{cCutOff}.
\end{proof}

Now, we consider another extension of $\BA$. In this new extension $\BAc$,
we augment the language $\LL$ with a new symbol ``$\dotminus$" for the the
cut-off subtraction, to get the language $\Lc$, and extend the theory
with the axioms $x\le y\Ra x\dotminus y=0$ and
$y\le x\Ra Sx\dotminus y=S(x\dotminus y)$. It turns out that
this extension of $\BA$ coincides with $\BA+\U$ over $\LL$.

Note that by Corollary \ref{cCutOff}, the cut-off subtraction is not provably
total in $\BA$, hence $\BAc$ might be stronger than $\BA$ (over the language
$\LL$). The next Theorem states that it is indeed the case, by showing that
$\BAc$ and $\BA+\U$ prove the same $\LL$-sequents. We need the following
definitions and lemmas concerning \emph{elimination of the cut-off subtraction
function symbol} to prove the Theorem.

\begin{definition}
\label{def-fnc-elim}
$ $
\begin{itemize}
\item
For an $\Lc$-term $t$ and a variable $x$ not occuring in $t$,
$\trm t x$ is defined recursively as follows:
\begin{itemize}
\item
$\trm tx\equiv t=x$, if $t$ is a variable or $0$;
\item
$\trm tx\equiv\ex y(\trm ry\land Sy=x)$
for a fresh variable $y$, if $t$ is of the form $Sr$;
\item
$\trm tx\equiv\ex yz(\trm ry\land\trm sz\land y+z=x)$
for fresh variables $y$ and $z$, if $t$ is of the form $r+s$;
\item
$\trm tx\equiv\ex yz(\trm ry\land\trm sz\land y\cdot z=x)$
for fresh variables $y$ and $z$, if $t$ is of the form $r\cdot s$;
\item
$\trm tx\equiv\ex yz(\trm ry\land\trm sz\land \co(y,z,x))$
for fresh variables $y$ and $z$, if $t$ is of the form $r\dotminus s$.
\end{itemize}
\item
For an $\Lc$-formula $A$, $\dfe A$ is defined recursively as follows:
\begin{itemize}
\item
$\dfe A\equiv\ex xy(\trm sx\land\trm ty\land x=y)$
for fresh variables $x$ and $y$, if $A$ is of the form $s=t$;
\item
$\dfe A\equiv\ex xy(\trm sx\land\trm ty\land x<y)$
for fresh variables $x$ and $y$, if $A$ is of the form $s<t$;
\item
$\dfe A\equiv A$, if $A$ is $\top$ or $\bot$;
\item
$\dfe A\equiv\dfe B\circ\dfe C$,
if $A$ is of the form $B\circ C$ and $\circ$ is $\lor$ or $\land$;
\item
$\dfe A\equiv\ex u\,\dfe B$, if $A$ is of the form $\ex u\,B$;
\item
$\dfe A\equiv\all\x\,(\dfe B\ra\dfe C)$,
if $A$ is of the form $\all\x\,(B\ra C)$.
\end{itemize}
\end{itemize}
$\dfe{(A\Ra B)}$ is defined as $\dfe A\Ra\dfe B$.
For a set $\Gamma$ of $\Lc$-sequents, $\dfe\Gamma$
is defined as $\{\dfe\alpha|\alpha\in\Gamma\}$.
For a rule $R$, $\dfe R$ is defined as the rule
obtained by replacing any upper or lower sequent
$\alpha$ of $R$ with $\dfe\alpha$.
\end{definition}

\begin{lemma}
\label{dfe-bqc}
$ $
\begin{enumerate}
\item
For every $\Lc$-term $t$ and every variable $x$ not occurring in $t$,
$\trm tx$ is an $\LL$-formula. Moreover, in case $t$ is an $\LL$-term,
$\BQC\vdash\trm tx\Lr t=x$.
\item
For every $\Lc$-formula $A$, $\dfe A$ is an $\LL$-formula. Moreover,
in case $A$ is an $\LL$-formula, $\BQC\vdash\dfe A\Lr A$.
\end{enumerate}
\end{lemma}
\begin{proof}
$ $
\begin{enumerate}
\item
Straightforward by induction on $t$.
\item
Straightforward by induction on $A$.\qedhere
\end{enumerate}
\end{proof}

\begin{lemma}
\label{lBAcU}
$ $
\begin{enumerate}
\item
$\BAc\vdash\U$.
\item
$\BAc\vdash\co(x,y,x\dotminus y)$.
\end{enumerate}
\end{lemma}
\begin{proof}
$ $
\begin{enumerate}
\item
We show that $\BAc\vdash(x+y)\dotminus x=y$ using the rule of induction
on the formula $(x+y)\dotminus x=y$ with $y$ as the eigenvariable.
It is easy to see that $\BAc\vdash(x+0)\dotminus x=0$.
Note that $\BA\vdash x\le x+y$, hence
$\BAc\vdash(S(x+y))\dotminus x=S((x+y)\dotminus y)$,
which implies
$\BAc\vdash(x+y)\dotminus x=y\Ra(x+Sy)\dotminus x=Sy$.
By the rule of induction,
$\BAc\vdash(x+0)\dotminus x=0\Ra(x+y)\dotminus x=y$,
which yields what was claimed. As
$\BAc\vdash x+y=x+z\Ra(x+y)\dotminus x=(x+z)\dotminus x$,
the proof is complete.
\item
Since $\BAc\vdash x<y\Ra x\dotminus y=0$,
we have $\BAc\vdash x<y\Ra\co(x,y,x\dotminus y)$.
As $\BAc\vdash x<y\lor y\le x$, it is now sufficient to prove
$\BAc\vdash y\le x\Ra x=y+(x\dotminus y)$, which implies
$\BAc\vdash x<y\Ra\co(x,y,x\dotminus y)$. To show this,
note that $\BAc\vdash y\le x\Lr\ex z\,(z=y+z)$, and thus proving
$\BAc\vdash x=y+z\Ra z=x\dotminus y$ does the job.
This last statement follows from the proof of the previous item:
$\BAc\vdash(y+z)\dotminus y=z$.\qedhere
\end{enumerate}
\end{proof}

\begin{lemma}
\label{BAc-dfe}
$ $
\begin{enumerate}
\item
For every $\Lc$-term $t$ and any variable $x$ not occuring in it,
$\BAc\vdash\trm tx\Lr t=x$.
\item
For every $\Lc$-formula $A$, $\BAc\vdash\dfe A\Lr A$
\end{enumerate}
\end{lemma}
\begin{proof}
$ $
\begin{enumerate}
\item
The proof is by induction on the term $t$. The only nontrivial case is
when $t$ is of the form $r\dotminus s$. By the induction hypothesis,
we have $\BAc\vdash\trm tx\Lr\ex yz\,(r=y\land s=z\land\co(y,z,x))$.
On the one hand, by item 2 of Lemma \ref{lBAcU},
we have $\BAc\vdash\co(r,s,t)$, and therefore
$\BAc\vdash t=x\Ra\ex yz\,(r=y\land s=z\land\co(y,z,x))$.
On the other hand, by Lemma \ref{U-equiv} and item 1 of Lemma \ref{lBAcU},
$\BAc\vdash\us(\co)$, and therefore $\BAc\vdash\co(r,s,x)\Ra t=x$.
This yields $\BAc\vdash\ex yz\,(r=y\land s=z\land\co(y,z,x))\Ra t=x$,
which together with the previous results completes the proof.
\item
The proof is by induction on the formula $A$. The basis is covered by
the previous item, and the induction steps are straightforward.\qedhere
\end{enumerate}
\end{proof}

\begin{lemma}
\label{dfe-subst}
$ $
\begin{enumerate}
\item
Let $t$ be an $\Lc$-term, and $x$ be a variable not occurring in $t$.
Then $\BA\vdash\ex x(\trm tx)$.
\item
Let $t$ be an $\Lc$-term, and $x$ and $y$ be variables not occurring in $t$.
Then $\BA+\U\vdash\trm tx\land\trm ty\Ra x=y$.
\item
Let $s$ and $t$ be $\Lc$-terms, $z$ be a variable occurring
neither in $s$ nor in $s[y/t]$, and $x$ be a variable not occurring
in $t$ and substitutable for $y$ in $\trm s z$. Then
$\BA+\U\vdash\trm{(s[y/t])}z\Lr\ex x(\trm t x\land(\trm sz)[y/x])$.
\item
Let $A$ be an $\Lc$-formula, $t$ be an $\Lc$-term substitutable for $y$
in $A$, and $x$ be a variable substitutable for $y$ in $A*$. Then
$\BA+\U\vdash\dfe{(A[y/t])}\Lr\ex x(\trm tx\land\dfe A[y/x])$.
\end{enumerate}
\end{lemma}
\begin{proof}
$ $
\begin{enumerate}
\item
Straightforward by induction on $t$, using Proposition \ref{cut-off-ptf}
in the case where $t$ is of the form $r\dotminus s$.
\item
The proof is by induction on $t$. The only nontrivial case is when $t$
is of the form $r\dotminus s$. By induction hypothesis, we have
$$\BA+\U\vdash\trm rz\land\trm sw\land\co(z,w,x)\land
\trm ru\land\trm sv\land\co(u,v,y)\Ra\co(u,v,x)\land\co(u,v,y)\text,$$
for suitable variables $z$, $w$, $u$ and $v$.
Since $\BA+\U\vdash\us(\co)$, we get
$$\BA+\U\vdash\trm rz\land\trm sw\land\co(z,w,x)\land
\trm ru\land\trm sv\land\co(u,v,y)\Ra x=y\text,$$
which in turn proves $\BA+\U\vdash\trm tx\land\trm ty\Ra x=y$.
\item
The proof is by induction on $s$. The only nontrivial case is when $s$
is of the form $s_1\dotminus s_2$. $\trm sz$ is of the form
$\ex u_1u_2(\trm{s_1}{u_1}\land\trm{s_2}{u_2}\land\co(u_1,u_2,z))$ for
some fresh variables $u_1$ and $u_2$, and $\trm{(s[x/t])}z$ is of the form
$\ex v_1v_2(\trm{(s_1[x/t])}{v_1}\land
\trm{(s_2[x/t])}{v_2}\land\co(v_1,v_2,z))$ for
some fresh variables $v_1$ and $v_2$, which respectively can be assumed to be
the same as $u_1$ and $u_2$, without loss of generality. By induction hypothesis, we have $\BA+\U\vdash\trm{(s_i[y/t])}{u_i}\Lr
\ex x(\trm t x\land(\trm{s_i}{u_i})[y/x])$ for $i=1,2$. Therefore
$$\BA+\U\vdash\trm{(s[x/t])}z\Lr\ex u_1u_2
(\ex x(\trm t x\land(\trm{s_1}{u_1})[y/x])\land
\ex x(\trm t x\land(\trm{s_2}{u_2})[y/x])\land\co(u_1,u_2,z))\text,$$
which by the previous item gives
$$\BA+\U\vdash\trm{(s[x/t])}z\Lr\ex u_1u_2
(\ex x(\trm t x\land(\trm{s_1}{u_1})[y/x]\land(\trm{s_2}{u_2})[y/x])
\land\co(u_1,u_2,z))\text.$$
Using logical rules, this statement is equivalent to
$$\BA+\U\vdash\trm{(s[x/t])}z\Lr\ex x(\trm t x\land(\ex u_1u_2
((\trm{s_1}{u_1})[y/x]\land(\trm{s_2}{u_2})[y/x]
\land\co(u_1,u_2,z))))\text,$$
which completes the proof.
\item
The proof is by induction on $A$. The basis is covered by the previous item.
The induction steps are straightforward using item 2. In the case where
$A$ is of the form $\all\z(B\ra C)$, also use item 1.\qedhere
\end{enumerate}
\end{proof}

\begin{theorem}
\label{BAUC}
For every $\Lc$-sequent $\alpha$ and every set $\Gamma$ of $\Lc$-sequents,
$\BAc+\Gamma\vdash\alpha$ iff $\BA+\U+\dfe\Gamma\vdash\dfe\alpha$.
\end{theorem}
\begin{proof}
Since $\BAc$ is an extension of $\BA$ and $\BAc\vdash\U$, If
$\BA+\U+\dfe\Gamma\vdash\dfe\alpha$ then $\BAc+\dfe\Gamma\vdash\dfe\alpha$,
which by item 2 of Lemma \ref{BAc-dfe} implies $\BAc+\Gamma\vdash\alpha$.
For the other direction, note that for any $\beta$ in $\Gamma$, $\dfe\beta$
is in $\dfe\Gamma$. Therefore, if we show $\BA+\U\vdash\dfe\beta$ for any
axiom $\beta$ of $\BAc$ and $\BA+\U\vdash\dfe R$ for any rule $R$ of
$\BAc$, the intended statement follows by induction on the derivations
of $\alpha$ from $\Gamma$ in $\BAc$. Let $\beta$ be one of the
arithmetical axioms 20-25 or the equality axiom 6. Then, by item 2 of
Lemma \ref{dfe-bqc}, we have $\BA+\U\vdash\dfe\beta$. For the equlity
axiom 7, Lemma \ref{dfe-subst} implies
$\BA+\U\vdash\dfe{(A[x/y])}\Lr\ex z(\trm yz\land\dfe A[x/z])$,
and consequently
$\BA+\U\vdash\dfe{(A[x/y])}\Lr\dfe A[x/y]$.
As we also have $\BA+\U\vdash\dfe{(x=y)}\Lr x=y$, we get
$\BA+\U\vdash\dfe{(x=y\land A\Ra A[x/y])}$. For the logical axiom 16,
consider $\x=(x_1,\dots,x_n)$ and $\t=(t_1,\dots,t_n)$. Without loss of
generality, assume that $x_i$ does not occur in $t_i$ for any $i$, and let
$C\equiv\bigwedge_{i=1}^n\trm{t_i}{x_i}$. Note that
$\BA+\U\vdash\forall\x(\dfe A\ra\dfe B)\Ra
\forall\x(C\land\dfe A\ra C\land\dfe B)$,
which by item 3 of Lemma \ref{dfe-subst} shows
$\BA+\U\vdash\dfe{(\all\x(A\ra B)\Ra\all\x(A[\x/\t]\ra B[\x/\t]))}$.
A similar argument proves $\BA+\U\vdash\dfe R$, where $R$ is the logical
rule 11. If $\beta$ is any of the other logical axioms or the induction
axiom, then $\dfe\beta$ is an instance of the same axiom, and if $R$
is any of the other logical rules or the rule of induction, then
$\dfe R$ is an instance of the same rule. Therefore, it only remains to prove
$\BA+\U\vdash\dfe{(x\le y\Ra x\dotminus y=0)}$ and
$\BA+\U\vdash\dfe{(y\le x\Ra Sx\dotminus y=S(x\dotminus y))}$. Note that
$\BQC\vdash\dfe{(x\dotminus y=0)}\Lr x<y\lor x=y+0$ and
$\BQC\vdash\dfe{(Sx\dotminus y=S(x\dotminus y))}
\Lr\ex z(\co(Sx,y,Sz)\land\co(x,y,z))$,
by directly checking Definition \ref{def-fnc-elim}. Thus,
by Lemma \ref{dfe-bqc}, it is sufficient to prove
$\BA+\U\vdash x\le y\Ra x<y\lor x=y+0$ and
$\BA+\U\vdash y\le x\Ra\ex z(\co(Sx,y,Sz)\land\co(x,y,z))$.
The first one is trivial. For the second one, note that
$\BA+\U\vdash y\le x\Lr\ex z(y+z=x)$ and
$\BA+\U\vdash y+z=x\Ra y+Sz=Sx$.
\end{proof}

\begin{corollary}
The provably total functions of $\BAc$ are exactly
the primitive recursive functions, and they are definable
in $\BAc$ by $\Ex$ formulas in the language $\LL$.
Consequently $\PTRF(\BAc)=\PTF(\BAc)=\PR$.
\end{corollary}
\begin{proof}
Assume that a function $f:\nn^n\to\nn$ is definable in $\BAc$ by an
$\Lc$-formula $A$. By item 2 of Lemma \ref{BAc-dfe}, $f$ is also
definable in $\BAc$ by the $\LL$-formula $\dfe A$. By Theorem \ref{BAUC},
$f$ is definable in $\BA+\U$ by the same formula, and by Theorem \ref{2t9},
it is definable in $\BA+\U$ by the geometric formula $\gp{(\dfe A)}$. Since
$\BAc$ is an extension of $\BA+\U$, the result follows by Theorem \ref{3t11}.
\end{proof}

Now we consider a third extension of $\BA$, i.e.
$\EBA=\BA+\top\ra\bot\Ra\bot$, introduced in 
\cite{AH}. To find out the provably total recursive functions of $\EBA$,
we use a result form \cite{S03}, in which the author introduced the
notion of \emph{primitive recursive realizability} for the language of $\BA$,
and showed that its provably total functions are primitive recursive.
The following definition is from \cite{S03}.

\begin{definition}
\label{3d12}
(Primitive Recursive Realizablity) 
Let $\varphi_n$ be the unary partial recursive function with G\"odel code $n$,
$\pi_1$ and $\pi_2$ be the primitive recursive projections of a fixed
primitive recursive pairing function $\langle \cdot , \cdot \rangle$,
and $\PRec(x)$ be the formula expressing that ``in the
program\footnote{The program of a partial recursive function shows how
it is defined in terms of \emph{zero}, \emph{successor} and
\emph{projection} functions by repeatedly applying
\emph{composition}, \emph{primitive recursion} and \emph{minimization}.
Note that by choosing a suitable coding, we can assume that every natural
number is the code of a program.} $x$ there is no use of
minimization"\footnote{It is true that $\nn\models\PRec(n)$ implies
$\varphi_n\in\PR$, but not vice versa. However, for every
primitive recursive function $f$ there is a natural number $n$
such that $\varphi_n=f$ and $\nn\models\PRec(n)$.}.
For a sequence $\x=(x_1,\dots,x_m)$, $\varphi_n(\x)$ is understood as
$\varphi_n(\langle x_1,\langle x_2,\dots,
\langle x_{m-1},x_m\rangle\rangle\rangle)$.
For a formula $A$, $x\:\r A$ is
defined by induction on complexity of $ A$:
\begin{itemize}
\item $x\:\r A\equiv A$, for prime $ A$.
\item $x\:\r (B \land C) \equiv (\pi_1(x) \r B)\land (\pi_2(x) \r C)$.
\item $x\:\r (B \lor C) \equiv (\pi_1(x)=0 \land \pi_2(x)\: \r B)
\lor (\pi_1(x)\ne 0 \land \pi_2(x)\: \r C)$.
\item $x\:\r \:\ex y B(y)\equiv \pi_2(x)\: \r B(\pi_1(x))$.
\item $x\:\r \:\all \z( B(\z)\ra C(\z)) \equiv \PRec(x)\land\all y\z\,
(y\:\r B(\z)\ra\varphi_x(y,\z)\:\r C(\z))\land\all\z(B(\z)\ra C(\z))$.
\end{itemize}
For a sequent $A\Ra B$, $x\:\r( A \Ra B)\equiv \PRec(x) \land
\all y\z\,(y\:\r A\ra \varphi_x(y,\z)\:\r B)\land ( A \ra B)$, where
$\z=(z_1,\dots,z_n)$ is the sequence of all free
variables in $ A \Ra B$ in the appearing order.
\end{definition}

\begin{theorem}
\label{3t14} For all sequents $ A \Ra B$, if $\EBA \vdash
 A \Ra B$, then $\nn\models n\:\r( A \Ra B)$ for
some natural number $n$.
\end{theorem}
\begin{proof}
Suppose $\varphi_n$ is the zero function, then $\nn\models
n\:\r(\top \ra \bot \Ra \bot)$. The rest of the proof is similar
to the proof of Theorem 4.4 in \cite{S03}, which shows the soundness
of $\BA$ with respect to primitive recursive realizability.
\end{proof}
\begin{corollary}
\label{3c15} For every formula $ A(\x,y)$ with the presented
free variables, if $\EBA \vdash \ex y\: A(\x,y)$, then there is
a (unary) primitive recursive function $f$ such that $\nn
\models A(\a,f(\a))$ for all $\a\in\nn^n$,
where $f(\a)$ for a sequence $\a=(a_1,\dots,a_n)$ means
$f(\langle a_1,\langle a_2,\dots,
\langle a_{n-1},a_n\rangle\rangle\dots\rangle)$.
\end{corollary}
\begin{proof}
Similar to the proof of Corollary 4.5 in \cite{S03},
which derives the same conclusion for $\BA$ from its
soundness with respect to primitive recursive realizability.
\end{proof}

\begin{definition}
\label{godel-nt}
For a formula $A$, the \emph{G\"odel negative translation}\footnote
{Note the slight difference between the current definition and the usual
negative translation, in which $\gt A=\neg\neg A$ for atomic $A$.} of $A$
is denoted by $\gt A$ and defined recursively as follows:
\begin{itemize}
\item
$\gt A=A$, if $A$ is prime;
\item
$\gt A\equiv\gt B\land\gt C$, if $A$ is of the form $B\land C$;
\item
$\gt A\equiv\neg(\neg\gt B\land\neg\gt C)$,
if $A$ is of the form $B\lor C$;
\item
$\gt A\equiv\neg\all u\,\neg\gt B$, if $A$ is of the form $\ex u\,B$;
\item
$\gt A\equiv\all\x\,(\gt B\ra\gt C)$,
if $A$ is of the form $\all\x\,(B\ra C)$.
\end{itemize}
$\gt{(A\Ra B)}$ is defined as $\gt A\Ra\gt B$.
\end{definition}

\begin{proposition}
\label{EBA-gt}
If $\PA\vdash\alpha$ then $\EBA\vdash\gt\alpha$.
\end{proposition}
\begin{proof}
Theorem 3.14 of \cite{AH}.
\end{proof}

\begin{theorem}
\label{3t16}
The provably total recursive functions of $\EBA$ are
exactly the provably total functions of $\EBA$, and
coincide with the primitive recursive functions,
i.e. $\PTRF(\EBA)=\PTF(\EBA)=\PR$. Furthermore these functions are
definable in $\EBA$ by $\Ex$ formulas.
\end{theorem}
\begin{proof}
By Corollary \ref{3c15}, the provably total functions of $\EBA$ are
primitive recursive, i.e. $\PTF(\EBA)\subseteq\PR$. On the other hand,
Since $\PA\vdash\U$, we have $\EBA\vdash \U$ by Proposition \ref{EBA-gt}.
So by Theorem \ref{3t11}, all primitive recursive functions are defined
in $\EBA$ by $\Ex$ formulas, and thus $\PR\subseteq\PTRF(\EBA)$.
\end{proof}

\section{The MRDP theorem in $\BA$ and some of its extensions}
In this section, we consider the well-known MRDP theorem in $\BA$ and its
extensions defined in the last section. We show that in $\BA$ and in two of
its extensions, i.e., $\BA$ augmented by the cancellation law and also $\BA$
augmented by the cut-off subtraction, the MRDP theorem does \emph{not} hold.
However, $\EBA$ is strong enough to have the MRDP theorem.

In the last part of this section, we will have a closer look at $\EBA$, and
obtain some of its nice properties with relation to some classical fragments
of $\PA$.

\begin{notation}
\label{3n26} Let $\T$ be a theory. Then:
\begin{itemize}
\item $\T\vdash \MRDP$ means that for every $\Sig$ formula $ A$
there exists a $\Ex$ formula $ B$ with the same free vraiables as $ A$
such that $\T\vdash A\Lr B$;
\item $\T\vdash\MRDPw$ means that for every $\Sig$ formula $ A$
there exists a $\Ex$ formula $ B$ with the same free vraiables as $ A$
such that $\T\vdash A\lr B$.
\end{itemize}
\end{notation}
\begin{theorem}
\label{3t27}
Let $\Gamma$ be a set of geometric sequents.
If $\HA+\Gamma$ is consistent, then $\BA+\Gamma\nvdash\MRDPw$.
\end{theorem}
\begin{proof}
Let $ A$ be the $\Sig$ formula defined as
$A\equiv\top\ra\bot$. Suppose that $\BA+\Gamma\vdash\MRDPw$.
Then there exists a $\Ex$ formula $ B$ such that
$\BA+\Gamma\vdash A\lr B$. In particular,
$\BA+\Gamma\vdash(\top\ra\bot)\ra B$. Hence by Proposition \ref{2c14},
$\BA+\Gamma\vdash\top\ra B$. So $\HA+\Gamma\vdash A$. Then 
$\HA+\Gamma\vdash\bot$, which leads to a contradiction.
\end{proof}

\begin{corollary}
\label{3c28} Neither $\MRDPw$ nor $\MRDP$ can be proved in
$\BA$, $\BA+\U$ or $\BAc$.
\end{corollary}
\begin{proof}
For $\BA$ and $\BA+\U$, the proof is straightforward by Theorem \ref{3t27}.
For $\BAc$, we note that for every geometric formula $A$ in the language
$\Lc$, $\dfe A$ is a geometric formula in the language $\LL$,
and thus by Theorem \ref{BAUC}, the claim is reduced to that of $\BA+\U$.
\end{proof}
For analyzing the MRDP theorem and other properties for $\EBA$,
we first consider the theory $\EBD$, which is formalized by the axioms and
rules of $\EBA$, except that axiom and rule of induction are restricted
to $\Dz$ formulas.
\begin{definition}
\label{3d31} For a $\Dz$ formula $ A$, the
\emph{bounded negation} of $ A$ is denoted by
$\dzn A$ and is defined recursively as follows:
\begin{itemize}
\item $\dzn\top\equiv\bot$,
\item $\dzn\bot\equiv\top$,
\item $\dzn{(s=t)}\equiv t<s\lor s<t$, if $s$ and $t$ are terms,
\item $\dzn{(s<t)}\equiv t<s\lor s=t$, if $s$ and $t$ are terms,
\item $\dzn A\equiv\dzn B\lor\dzn C$ if $A$ is of the form $B\land C$,
\item $\dzn A\equiv\dzn B\land\dzn C$ if $A$ is of the form $B\lor C$,
\item $\dzn A\equiv B\land\dzn C$ if $A$ is of the form $B\ra C$,
\item $\dzn A\equiv\all x(x<t\ra\dzn{B(x)})$ if $A$ is of the form
$\ex x(x<t\land B(x))$,
\item $\dzn A\equiv\ex x(x<t\land\dzn{B(x)})$ if $A$ is of the form
$\all x(x<t\ra B(x))$.
\end{itemize}
\end{definition}
\begin{lemma}
\label{3l32}
For every $\Dz$ formula $A$, $\EBD\vdash A\lor\dzn A$
and $\EBD\vdash A\land\dzn A\Ra\bot$. Consequently
$\EBD\vdash\dzn A\Lr\neg A$.
\end{lemma}
\begin{proof}
We prove $\EBD\vdash A\lor\dzn A$ and $\EBD\vdash A\land\dzn A\Ra\bot$
simultaneously by induction on the complexity of $A$. But first, note that $\EBD\vdash\PAm$,
i.e. $\EBD$ proves all the axioms 28-44. For the those other than axiom 37,
the argument is similar to those in Lemmas 2.5 and 2.8 and Proposition 2.6
in \cite{AH}, which prove the same axioms for $\BA$, noting that only
$\Dz$-induction formulas are used. Derivability of axiom 37 follows from
applying the $\Dz$-induction rule to the formula $\neg x<x$ and
the fact that $\EBD\vdash\top\ra\bot\Ra\bot$.

If $A$ is of the form $s=t$ or $s<t$, then $A\lor\dzn A$ is
provably equivalent to $s<t\lor s=t\lor t<s$ in $\BQC$, and thus
derivable in $\EBD$. This also shows $\EBD\vdash\neg A\Ra\dzn A$, since
$\EBD\vdash\top\ra\bot\Ra\bot$. For the remaining part of
the base case, it is sufficient to prove $\EBD\vdash s=t\land s<t\Ra\bot$,
which is a consequence of $\EBD\vdash x<x\Ra\bot$.

In case $A$ is of the form $B\ra C$, we have the induction hypothesis
$\EBD\vdash B\lor\dzn B$ and $\EBD\vdash C\lor\dzn C$, together with
$\EBD\vdash B\land\dzn B\Ra\bot$ and $\EBD\vdash C\land\dzn C\Ra\bot$.
Note that we already have the derivability of $\neg B\lor C\Ra(B\ra C)$
and $(B\ra C)\land(B\land\neg C)\Ra\top\ra\bot$ in $\BQC$, which
combining with the induction hypothesis, completes this induction step.

Suppose that $A$ is of the form $\ex x(x<s\land B(x))$. Define
$C(y)\equiv\ex x(x<y\land B(x))\lor\all x(x<y\ra\dzn{B(x)})$,
where $y$ does not occur in $A$. From the induction hypothesis
$\EBD\vdash B(y)\lor\dzn{B(y)}$, it can easily be
derived that $\EBD\vdash C(y)\Ra C(Sy)$. So by the
induction rule, $\EBD\vdash C(0)\Ra C(y)$. Also from
$\EBD\vdash x<0\Ra\bot$, we have $\EBD\vdash x<0\Ra\dzn{B(x)}$,
and thus $\EBD\vdash C(0)$. Hence $\EBD\vdash C(y)$, and
$\EBD\vdash C(s)$, which means $\EBD\vdash A\lor\dzn A$.
To prove $\EBD\vdash A\land\dzn A\Ra\bot$, note that we have
$\BQC\vdash\ex x(x<s\land B(x))\land\all x(x<s\ra\neg B(x))
\Ra\top\ra(\top\ra\bot)$.

An argument similar to the one for the previous case works for
the case where $A$ is of the form $\all x(x<t\ra B(x))$.
The other cases are easy to verify.
\end{proof}
\begin{lemma}
\label{3l33}
For every $\Dz$ formulas $A$ and $B$, the following hold:
\begin{enumerate}
\item $\EBD\vdash\top\ra A\Lr A$;
\item $\EBD\vdash A\Lr\neg\neg A$;
\item $\EBD\vdash A\ra B\Lr\neg A\lor B$.
\end{enumerate}
\end{lemma}
\begin{proof}
$ $
\begin{enumerate}
\item Use Lemma \ref{3l32} and the fact that
$\BQC\vdash(\top\ra A)\land\neg A\Ra\top\ra\bot$.
\item Use Lemma \ref{3l32} and the fact that
$\BQC\vdash A\land\neg A\Ra\top\ra\bot$ and
$\BQC\vdash\neg A\land\neg\neg A\Ra\top\ra\bot$.
\item Use Lemma \ref{3l32} for $A$ and item 1 for $B$.\qedhere
\end{enumerate}
\end{proof}

The following Lemma shows that $\EBD$
(and consequently any of its extensions, for instance $\EBA$)
proves \emph{the least number principle} for $\Dz$ formulas.

\begin{lemma}
\label{3l34}
For every $\Dz$ formula $ A(x)$, $\EBD \vdash
\ex x A(x)\Ra\ex x( A(x) \land \all y(y<x \ra \neg A(x)))$.
\end{lemma}
\begin{proof}
Define $ B(x)\equiv\all y(y<x \ra \neg A(y))\lor \ex y(y<x\land
A(y) \land \all z(z<y \ra \neg A(z)))$, where $x$ does not occur
in $A(y)$. we have:
\begin{itemize}
\item $\EBD \vdash \all y(y<x \ra \neg A(y)) \land \neg A(x)\Ra
\all y(y<Sx \ra \neg A(y))$,
\item $\EBD \vdash \all y(y<x \ra \neg A(y)) \land A(x)\Ra \ex
y(y<Sx\land A(y) \land \all z(z<y \ra \neg A(z)))$,
\item $\EBD \vdash \ex y(y<x\land A(y) \land \all z(z<y \ra
\neg A(z)))\Ra \ex y(y<Sx\land A(y) \land \all z(z<y \ra
\neg A(z)))$.
\end{itemize}
By Lemma \ref{3l32}, $\EBD \vdash A(x) \lor \neg A(x)$. By the
above items we have $\EBD \vdash B(x)\Ra B(Sx)$, and by the
induction rule, $\EBD \vdash B(0)\Ra B(x)$. We also have $\EBD
\vdash B(0)$, and thus $\EBD \vdash B(x)$. Lemma \ref{3l32} implies
that $\EBD \vdash \all y(y<x\ra \neg A(y))\land \ex y(y<x \land
A(y))\Ra \bot$. Therefore $\EBD \vdash \ex y(y<x \land A(y))\Ra
\ex y(y<x\land A(y) \land \all z(z<y \ra \neg A(z)))$, and hence
$\EBD \vdash \ex y A(y)\Ra \ex y( A(y) \land \all z(z<y \ra \neg
A(z)))$.
\end{proof}

The next Lemma shows that $\EBD$ and its sequent extensions prove instances
of induction (both in the axiom form and the rule form) that are more
general than just $\Dz$-induction. This is needed to show that these
theories are functional and well-formed, and therefore satisfy the
soundness and completeness with respect to their corresponding class
of Kripke models; see subsection \ref{semantics}. Note that the proof of
the Lemma and the statements it relies on are all syntactic in nature.
In particular, completeness itself hasn't been used and no cyclic reasoning
is happening here.

\begin{lemma}
\label{EBD-func-wll}
Let $A$ be a formula, $B(x)$ be a $\Dz$ formula such that $x$ is
not free in $A$, and $\Gamma$ be a set of sequents in $\LL$. Then
\begin{enumerate}
\item
$\EBD+\Gamma+A\land B(x)\Ra B(Sx)\vdash A\land B(0)\Ra B(x)$;
\item
$\EBD+\Gamma\vdash\all\y(A\land B(x)\ra B(Sx))\Ra\all\y(A\land B(0)\ra B(x))$,
for any sequence $\y$ of variables.
\end{enumerate}
\end{lemma}
\begin{proof}
$ $
\begin{enumerate}
\item
Since $\EBD\vdash x=0\lor\ex y\,x=Sy$, we have
$\EBD\vdash B(0)\land\neg B(x)\Ra\ex y\,x=Sy$. By Lemma \ref{3l34} and
item 2 of Lemma \ref{3l33}, this implies
$\EBD\vdash B(0)\land\neg B(x)\Ra\ex y(\neg B(Sy)\land\all z(z<Sy\ra B(z)))$,
thus $\EBD\vdash B(0)\land\neg B(x)\Ra\ex y(\neg B(Sy)\land B(y))$, and hence
$\EBD\vdash A\land B(0)\land\neg B(x)\Ra\ex y(A\land B(y)\land\neg B(Sy))$.
Therefore
$\EBD+\Gamma+A\land B(x)\Ra B(Sx)\vdash A\land B(0)\land\neg B(x)\Ra\bot$,
implying
$\EBD+\Gamma+A\land B(x)\Ra B(Sx)\vdash A\land B(0)\Ra\neg\neg B(x)$,
which by item 2 of Lemma \ref{3l33} yields
$\EBD+\Gamma+A\land B(x)\Ra B(Sx)\vdash A\land B(0)\Ra B(x)$.
\item
By an argument similar to the beginning parts of what we did for
the previous item, we get $\EBD\vdash\all\y(A\land B(0)\land\neg B(x)
\ra\ex z(A\land B(z)\land\neg B(Sz)))$. As $\EBD\vdash
\all\y(A\land B(x)\ra B(Sx))\Ra\all\y\neg\ex z(A\land B(z)\land\neg B(Sz))$,
we have $\EBD+\Gamma\vdash
\all\y(A\land B(x)\ra B(Sx))\Ra\all\y\neg(A\land B(0)\land\neg B(x))$.
Consequently $\EBD+\Gamma\vdash
\all\y(A\land B(x)\ra B(Sx))\Ra\all\y(A\land B(0)\ra\neg\neg B(x))$,
and by item 2 of Lemma \ref{3l33},
$\EBD+\Gamma\vdash\all\y(A\land B(x)\ra B(Sx))\Ra\all\y(A\land B(0)\ra B(x))$.
\end{enumerate}
\end{proof}

\begin{lemma}
\label{3l36}
Let $\km$ be a Kripke model $\km$ such that $\km\Vdash\EBD$,
$k$ be a node of $\km$ and $A$ be a $\Dz$ sentence in $\LL(k)$.
Then $k\Vdash A$ implies $\M\models A$, and
$k\Vdash\neg A$ implies $\M\models\neg A$.
\end{lemma}
\begin{proof}
We prove the Lemma by induction on $A$.
\begin{itemize}
\item
Assume that $A$ is prime. By the definition of the
forcing relation, $k\Vdash A$ is equivalent to $\M\models A$.
The result follows by Lemma \ref{3l32}.
\item
Assume that $A$ is of the form $B\ra C$. If $k\Vdash A$,
then by item 3 of Lemma \ref{3l33} we have $k\Vdash\neg B$ or
$k\Vdash C$. The induction hypothesis then gives
$\M\models\neg B$ or $\M\models C$, or equivalently
$\M\models A$. If $k\Vdash\neg A$, then by Lemma \ref{3l32},
we get $k\Vdash B$ and $k\Vdash\neg C$. Applying the induction hypothesis
yields $\M\models B$ and $\M\models\neg C$, hence $\M\models\neg A$.
\item
The case where $A$ is of the form $B\circ C$ for
$\circ\in\{\land,\lor\}$ can be done similar to the previous case.
\item
Assume that $A$ is of the form $\all x(x<t\ra B(x))$.
If $k\Vdash A$, then for all $k'\succ k$ and all $a\in D(k')$,
$k'\Vdash a<t$ implies $k'\Vdash B(a)$. In particular,
this holds for all $a\in D(k)$, which shows that $k\Vdash a<t$
implies $k\Vdash\top\ra B(a)$. Using item 1 of Lemma \ref{3l33}
then shows that for all $a\in D(k)$, $k\Vdash a<t$ implies
$k\Vdash B(a)$. Since $a<t$ is atomic, $k\Vdash a<t$
is equivalent to $\M\models a<t$. Applying the induction hypothesis
will then show that for all $a\in D(k)$, $\M\models a<t$ implies
$\M\models B(a)$, hence $\M\models A$. If $k\Vdash\neg A$,
then by Lemma \ref{3l32} we have $k\Vdash\ex x(x<t\land\neg B(x))$.
Thus, there is some $a\in D(k)$ such that $k\Vdash a<t$ and
$k\Vdash\neg B(a)$. Noting that $a<t$ is atomic and using the induction
hypothesis, we can conclude that there is some $a\in D(k)$ with
$\M\models a<t$ and $\M\models\neg B(a)$. This shows
$\M\models\ex x(x<t\land\neg B(x))$, and consequently $\M\models\neg A$.
\item
The case where $A$ is of the form $\ex x(x<t\land\neg B(x))$
can be done similar to the previous case.
\end{itemize}
\end{proof}

\begin{corollary}
\label{3c38}
For every Kripke model $\km$ such that $\km\Vdash\EBD$,
every node $k$ of $\km$ and
\begin{enumerate}
\item
every $\Dz$ sentence $A$ and every $\Sig$ sentence $B$ in $\LL(k)$,
if $\M\models A\Ra B$ then $k\Vdash A\Ra B$;
\item
every $\Sig$ formulas $A(\x)$ and $B(\x)$ in $\LL(k)$,
if $k\Vdash A(\x)\Ra B(\x)$ then $\M\models A(\x)\Ra B(\x)$.
\end{enumerate}
\end{corollary}
\begin{proof}
$ $
\begin{enumerate}
\item
Suppose that $B\equiv\ex\x C(\x)$ and $\M\models A\Ra B$,
where $C(\x)$ is a $\Dz$ formula in $\LL(k)$. This means that
$\M\models A$ implies $\M\models B$. Assume $k'\Vdash A$ for
any $k'\succeq k$. In case $k'=k$, we already have $k\Vdash A$.
In case $k'\succ k$, note that $k\Vdash\neg A$ implies
$k'\nVdash A$, which cannot happen. Thus, by Lemma \ref{3l32},
we must have $k\Vdash A$. So in any case, by Lemma \ref{3l36} we have
$\M\models A$, and thus $\M\models B$. Therefore, there is some
$\a\in D(k)$ with $\M\models C(\a)$. by Lemma \ref{3l36},
$k\Vdash\neg C(\a)$ implies $\M\models\neg C(\a)$, which cannot
happen. Hence, by Lemma \ref{3l32}, we must have $k\Vdash C(\a)$.
Thus $k\Vdash B$, and consequently $k'\Vdash A\Ra B$.
\item
Suppose that $A(\x)\equiv\ex\y C(\x,\y)$, $B(\x)\equiv\ex\z D(\x,\z)$, and
$k\Vdash A(\x)\Ra B(\x)$, where $C(\x,\y)$ and $D(\y,\z)$ are $\Dz$ formulas
in $\LL(k)$. This means that for all $k'\succeq k$ and all $\a\in D(k')$,
$k'\Vdash A(\a)$ implies $k'\Vdash B(\a)$, which in particular holds when
$k'=k$. Assuming $\M\models A(\a)$ for any $\a\in D(k)$, there is some
$\b\in D(k)$ such that $\M\models C(\a,\b)$. By Lemma \ref{3l36},
$k\Vdash\neg C(\a,\b)$ implies $\M\models\neg C(\a,\b)$, which cannot happen.
Therefore, by Lemma \ref{3l32}, we must have $k\Vdash C(\a,\b)$, and thus
$k\Vdash A(\a)$. Hence we get $k\Vdash B(\a)$; i.e. there exists $\cc\in D(k)$
such that $k\Vdash D(\a,\cc)$. Applying Lemma \ref{3l36} again, we can
conclude that $\M\models D(\a,\cc)$. Hence $\M\models B(\a)$,
which completes the proof.
\end{enumerate}
\end{proof}
\begin{theorem}
\label{3t39} For every Kripke model $\km$ such that
$\km\Vdash\EBD$ and a node $k$ of $\km$, $\M\models\ID$.
\end{theorem}
\begin{proof}
We show that $\Dz$-induction holds in $\M$; the other axioms and rules
are easy to check. Let $ A(x)$ be a $\Dz$ formula. By Lemma \ref{3l34},
$\EBD\vdash\ex x\neg A(x)\Ra\ex x(\neg A(x)\land\all y(y<x\ra A(x)))$.
Then by item 2 of Corollary \ref{3c38},
$\M\models\ex x\neg A(x)\Ra\ex x(\neg A(x)\land\all y(y<x \ra A(x)))$.
If $\M\models A(0)\land\all x(A(x)\ra A(Sx))$ and
$\M\models\ex x\neg A(x)$, then there exist $a\in D(k)$ such that
$\M\models\neg A(a)\land\all y(y<a \ra A(y))$, which leads to a contradiction,
since we have $\M\models a=0\lor\ex y\,a=Sy$.
\end{proof}
\begin{theorem}
\label{3t40}
For every $\Sig$ formulas $A(\x)$ and $B(\x)$,
if $\ID\vdash A(\x)\Ra B(\x)$, then $\EBD\vdash A(\x)\Ra B(\x)$.
\end{theorem}
\begin{proof}
Suppose that $A(\x)\equiv\ex\y C(\x,\y)$, where $C(\x,\y)$ is a
$\Dz$ formula, and assume that $\EBD\nvdash A(\x)\Ra B(\x)$.
Then there exists a Kripke model of $\EBD$, a node $k$ and some
$\a\in D(k)$ such that $k\nVdash A(\a)\Ra B(\a)$. As rule 12 is
valid in $\km$, there is a node $k'\succeq k$ and some $\b\in D(k')$
such that $k'\nVdash C(\a,\b)\Ra B(\a)$. So by item 1 of Corollary
\ref{3c38}, $\Mp\not\models C(\a,\b)\Ra B(\a)$. On the other hand,
since rule 12 is also valid in $\ID$, we have $\ID\vdash C(\x,\y)\Ra B(\x)$.
By Theorem \ref{3t39} $\Mp\models\ID$, and we must have
$\Mp\models C(\a,\b)\Ra B(\a)$, which leads to a contradiction.
\end{proof}
\begin{theorem}
\label{3t41}
There exists a $\Dz$ formula $\e(x,y,z)$ such that:
\begin{itemize}
\item $\ID\vdash \e(x,0,1)$
\item $\ID\vdash y>0\Ra\e(0,y,0)$
\item $\ID\vdash \e(x,y,z)\land\e(x,y,w)\Ra z=w$
\item $\ID\vdash \e(x,y,z)\Ra\e(x,Sy,x\cdot z)$
\end{itemize}
\end{theorem}
\begin{proof}
See appendix in \cite{GD}.
\end{proof}
Let $\EXP$ be $\es(\e)$, i.e. $\ex z\e(x,y,z)$.
\begin{lemma}
\label{3l43}
$\EBA\vdash \EXP$
\end{lemma}
\begin{proof}
By Theorems \ref{3t40} and \ref{3t41}, we have $\EBA \vdash
\e(x,y,z)\Ra \e(x,Sy,x\cdot z)$. Then $\EBA \vdash \ex
z\e(x,y,z)\Ra\ex z\e(x,Sy,z)$ and so by the induction rule, we
have $\EBA \vdash \ex z\e(x,0,z)\Ra\ex z\e(x,y,z)$. Also by
Theorems \ref{3t40} and \ref{3t41}, we have $\EBA\vdash
\e(x,0,1)$. Hence $\EBA\vdash \EXP$.
\end{proof}
\begin{theorem}
\label{3t44} For every Kripke model $\km$ such that
$\km\Vdash \EBD+\EXP$ and a node $k$, $\M\models\ID+\EXP$
\end{theorem}
\begin{proof}
By Theorem \ref{3t39}, $\M\models\ID$. Also, since $k\Vdash\EXP$,
for all $a,b\in D(k)$ there is $c\in D(k)$ with $k\Vdash\e(a,b,c)$,
which by Corollary \ref{3c38} implies $\M\models\e(a,b,c)$.
Therefore $\M\models\EXP$, and hence $\M\models\ID+\EXP$.
\end{proof}
\begin{theorem}
\label{3t45} For every $\Sig$ formulas $ A(\x)$ and
$ B(\x)$, if $\ID+\EXP\vdash A(\x)\Ra B(\x)$, then
$\EBD+\EXP\vdash A(\x)\Ra B(\x)$.
\end{theorem}
\begin{proof}
The proof is similar to the proof of Theorem \ref{3t40} by using
Theorem \ref{3t44}.
\end{proof}
\begin{theorem}
\label{3t46}
$\ID+\EXP\vdash\MRDP$.
\end{theorem}
\begin{proof}
See section 4 in \cite{GD}.
\end{proof}
\begin{theorem}
\label{3t47}
$\EBD+\EXP\vdash\MRDP$.
\end{theorem}
\begin{proof}
Consider a $\Sig$ formula $A$. By Theorem \ref{3t46}, there exists a
$\Ex$ formula $B$ such that $\ID+\EXP\vdash A\Lr B$. Since both
$A$ and $B$ are $\Sig$, the result follows from Theorem \ref{3t45}.
\end{proof}
\begin{corollary}
\label{3c48}
$\EBA\vdash \MRDP$.
\end{corollary}
\begin{proof}
Straightforward by Lemma \ref{3l43} and Theorem \ref{3t47}.
\end{proof}
\begin{theorem}
\label{3t49}
$ $
\begin{enumerate}
\item For every $\Sig$ formulas $ A(\x)$ and
$ B(\x)$, if $\ISig\vdash A(\x)\Ra B(\x)$, then
$\EBA\vdash A(\x)\Ra B(\x)$.
\item For every $\Sig$ formula $A(\x)$ and every $\Dz$ formula $B(\x)$, 
if $\PA\vdash A(\x)\Ra B(\x)$, then $\EBA\vdash A(\x)\Ra B(\x)$.
\end{enumerate}
\end{theorem}
\begin{proof}
$ $
\begin{enumerate}
\item Suppose that $\ISig\vdash A(\x)\Ra B(\x)$. By Theorem \ref{3t46}:
\begin{itemize}
\item $\ID+\EXP\vdash A(\x) \Lr C(\x)$ for some $\Ex$ formula $C(\x)$,
\item $\ID+\EXP\vdash B(\x) \Lr D(\x)$ for some $\Ex$ formula $D(\x)$.
\end{itemize}
Since $\ISig\vdash \ID+\EXP$, we have $\ISig\vdash C(\x)\Ra D(\x)$.
By Corollary \ref{3c8}, $\IEx+\U\vdash C(\x)\Ra D(\x)$.
So by Corollary \ref{2c6.1}, $\BA+\U\vdash C(\x)\Ra D(\x)$.
Since $\EBA\vdash \U$, we get $\EBA\vdash C(\x)\Ra D(\x)$.
On the other hand, by Theorem \ref{3t45}:
\begin{itemize}
\item $\EBA\vdash A(\x) \Lr C(\x)$,
\item $\EBA\vdash B(\x) \Lr \ D(\x)$.
\end{itemize}
Hence $\EBA\vdash A(\x)\Ra B(\x)$.
\item Suppose that $A(\x)\equiv\ex\y C(\x,\y)$ for some $\Dz$ formula
$C(\x,\y)$, and $\PA\vdash A(\x)\Ra B(\x)$. By rule 12,
$\PA\vdash C(\x,\y)\Ra B(\x)$, hence by Proposition \ref{EBA-gt},
$\EBA\vdash\gt C(\x,\y)\Ra\gt B(\x)$. By Lemma \ref{3l32},
every $\Dz$ formula is equivalent in $\EBA$ to its G\"odel translation.
Therefore $\EBA\vdash C(\x,\y)\Ra B(\x)$, and by rule 12,
$\EBA\vdash A(\x)\Ra B(\x)$.\qedhere
\end{enumerate}
\end{proof}
\begin{corollary}
\label{3c50}
$ $
\begin{enumerate}
\item $\EBA$ proves all $\Pit$ theorems of $\ISig$.
\item $\PA$ is $\Pio$-conservative over $\EBA$.
\end{enumerate}
\end{corollary}
\begin{proof}
Straightforward by Theorem \ref{3t49}.
\end{proof}
It is worth mentioning that every Kripke model
of $\HA$ is $\ID+\Th_{\Pit}(\PA)$-normal (see
Theorem 3.1 in \cite{W96}) and also $\HA$ is complete with
respect to $\PA$-normal Kripke models (see Theorem 6
in \cite{B93}). Having these facts, we have the following
results.
\begin{theorem}
\label{3t52}
$ $
\begin{enumerate}
\item Every Kripke model of $\EBA$ is
$\ID+\Th_{\Pit}(\ISig)+\Th_{\Pio}(\PA)$-normal.
\item $\EBA$ is not complete with respect to $\Th_{\Pit}(\PA)$-normal
Kripke models. In particular, there exists a Kripke model of $\EBA$
which is not $\Th_{\Pit}(\PA)$-normal.
\end{enumerate}
\end{theorem}
\begin{proof}
$ $
\begin{enumerate}
\item Straightforward by Theorem \ref{3t39} and Corollary \ref{3c50}.
\item It is well-known that the Ackermann function is provably total
recursive in $\PA$ (see \cite{W97}). Therefore by Theorem \ref{3t46}
there exists a $\Ex$ formula $ B(x,y,z)$ such that
\begin{itemize}
\item $\PA\vdash \es(B)$,
\item $\PA \vdash \us(B)$,
\item $\nn\models B(a,b,\Psi(a,b))$, for every $a,b\in\nn$,
\end{itemize}
where $\Psi(n,m)$ is the Ackermann function.
Note that provability of $\es(B)$ is equivalent to provability of
the $\Pit$ sentece $\all xy\ex z\, B(x,y,z)$. If $\EBA$ is complete with
respect to $\Th_{\Pit}(\PA)$-normal Kripke models, then we must have
$\EBA\vdash\es(B)$. So by Corollary \ref{3c15}, there exists a unary
primitive recursive function $f$ such that
$\nn\models\all xy\,B(x,y,f(\langle x,y\rangle))$.
Because $\PA\vdash\us(B)$, $f(\langle n,m\rangle)=\Psi(n,m)$,
which leads to a contradiction,
since the Ackermann function is not primitive recursive.
Hence $\EBA$ is not complete with respect to $\Th_{\Pit}(\PA)$-normal
Kripke models.
\end{enumerate}
\end{proof}

We know that for every natural number $n$, there is a $\Sig$ formula
$\Prov_{\ISign}(x,y)$, encoding provabality relation of
$\ISign$ in the language of arthmetic (see Section 4 of Chapter I in
\cite{HP}). Consistency of $\ISign$ can be encoded by
$\Con_{\ISign} \equiv \all x\neg\Prov_{\ISign}
(x,\ulcorner\bot\urcorner)$. In fact, by the MRDP theorem, we can assume that
$\Prov_{\ISign}(x,y)$ is of the form $\ex\z\, s_n(\z,x,y)=t_n(\z,x,y)$, 
over any theory in which the MRDP theorem holds, in particular $\EBA$.
Note that $\Iop$ proves
$s=s' \lor t=t' \Lr s\cdot t+s'\cdot t' = s\cdot t'+s'\cdot t$ and
$s=s' \land t=t' \Lr s^2+s'^2+t^2+t'^2 = 2s\cdot s'+2t\cdot t'$, and hence
over any strong enough theory, every $\Ex$ formula is provably equivalent
to a formula consisting of a block of existential quantifiers followed by
an equality of two terms.
\begin{lemma}
\label{EBA-Con}
$\EBA \vdash \Con_{\ISign}$.
\end{lemma}
\begin{proof}
By Corollary 4.34 of Chapter I in \cite{HP},
$\ISignp\vdash\Con_{\ISign}$, and thus $\PA\vdash\Con_{\ISign}$.
Since $\Con_{\ISign}$ is $\Pio$, by item 2 of Corollary \ref{3c50},
$\EBA\vdash\Con_{\ISign}$.
\end{proof}

\begin{corollary}
\label{EBA-NoTrans}
If $\ISign$ is consistent, there is no translation $\trns{(\cdot)}$ from the
set of formulas of arithmetic to itself such that:
\begin{itemize}
\item for every prime formula $A$,
$\ISign\vdash\trns{A}\Lr A$,
\item for every formulas $A$ and $B$,
if $\EBA\vdash A\Ra B$ then $\ISign\vdash \trns A\Ra\trns B$.
\end{itemize}
\end{corollary}
\begin{proof}
By lemma \ref{EBA-Con}, we have $\EBA\vdash
\neg s_n(\z,x,\ulcorner\bot\urcorner)=t_n(\z,x,\ulcorner\bot\urcorner)$.
This implies that $\EBA\vdash s_n(\z,x,\ulcorner\bot\urcorner)=
t_n(\z,x,\ulcorner\bot\urcorner)\Ra\bot$ as $\EBA\vdash \top\ra\bot\Ra\bot$.
If such a translation exists, we must have $\ISign \vdash
\neg s_n(\z,x,\ulcorner\bot\urcorner)=t_n(\z,x,\ulcorner\bot\urcorner)$ and
hence $\ISign \vdash \Con_{\ISign}$, which contradicts
G\"odel's second incompleteness theorem.
\end{proof}

Corollary \ref{EBA-NoTrans} shows that a proposition similar to
Proposition \ref{2p5} cannot be proved for $\EBA$. Thus to find 
an upper bound on the class of the provably total functions of $\EBA$, we
cannot use techniques similar to those we used for $\BA$ and $\BA+\U$.
\section{Final Remarks}
We have proved that the provably total functions of $\BA$ are primitive
recursive and definable in $\BA$ by geometric formulas. Our attempt to
characterize the provably total functions of $\BA$
was not successful, however we could do it for some extensions of $\BA$.
We introduced alternative versions of the uniqueness sequent,
for some of which we also proved the geometric definability of
the provably total functions. We further proved that provably total functions
of $\BA$ in those alternative senses are primitive recursive.
It is worth mentioning that the primitive recursive realizability technique
introduced in \cite{S03} (see Definition \ref{3d12}), that we also applied
to $\EBA$ (see Theorem \ref{3t14}),
is used to analyze the provability of existence sequents. Therefore,
regardless of which definition we choose for the uniqueness sequent,
the primitive recursive upper bound can be proven to exist for the class
of the provably total functions of $\BA$.

One may expect that taking one of the alternative definition for uniqueness
sequents mentioned above, characterizing the class of the provably total
functions of $\BA$ reduces to that of $\BA+\U$ for the following reason.
The formula $x+y=x+z\ra y=z$, which is provable in $\BA$, is actually a
weakening of the axiom $\U$. We just note that such a reduction may not be
a straightforward procedure, because there are sequents consisting of
geometric formulas provable in $\BA+\U$, such that the weakened conditional
formula is not provable in $\BA$. This may mean that for the mentioned
reduction, we can neither use arguments based on derivations containing only
geometric formulas, nor arguments based on local analysis of Kripke models.
As an example, note that $\BA+\U\vdash\ex x (x+y=x+z)\Ra y=z$, but
$\BA\nvdash\ex x(x+y=x+z)\ra y=z$, since the formula is refuted in the
Kripke model consisting of two irreflexive nodes, the below one with the
structure $\nn$ and the above one with the structure $\N$.

We cannot reduce the characterization of the provably total
functions of $\BA$ to that of $\IEx$, as we did between $\BA+\U$ and
$\IEx+\U$. The reason is that the MRDP theorem does not hold in $\IEx$
(see Corollary \ref{3c000}), and thus the defining formulas of its provably
total functions may \emph{not} be equivalent to any geometric formulas,
and so we lose the benefits of Corollary \ref{2c6.1}.

Our choice of the additional axioms for $\BAc$ (axioms 46 and 47)
is not canonical. One can fix other axioms for formalizing the
properties of cut-off subtraction; but that may change results on
the class of provably total functions and the MRDP theorem for $\BAc$.
One alternative is using the axioms $x<y\Ra x\dotminus y=0$ and
$y\le x\Ra y+(x\dotminus y)=x$ instead, which may more resemble
the definition given for cut-off subtraction at the end of Subsection
\ref{AaRoBA}. Another alternative would be the three axioms
$x\dotminus0=x$, $x\dotminus y=0\Ra x\dotminus Sy=0$ and
$x\dotminus y=Sz\Ra x\dotminus Sy=z$, which may more resemble the
definition of cut-off subtraction by primitive recursion using
the predecessor function, as mentioned before Proposition \ref{pd}.
It seems that in both of these alternative cases, $x\dotminus x=0$
may not be derivable from the resulting theory. An idea for seeing
this would be considering the expansion $\N_\8$ of $\N$ in
which $n\dotminus\infty=0$ and $\infty\dotminus n=\infty$ for all
$n\in\nn$, and $\8\dotminus\8=\8$. If one could show that $\N_\8$
is a model of $\IEx$, it would result in unprovability of $x\dotminus x=0$
in the theory. Even adding $x\dotminus x=0$ to the other axioms might
not be satisfactory: consider an expansion $\N_0$ of $\N$ similar to
$\N_\8$, only this time with $\8\dotminus\8=0$. While this is a
model of the additional axiom $x\dotminus x=0$, it does not satisfy
$(\8+\8)\dotminus\8=\8$. Therefore if $\N_0$ is a model of $\IEx$,
we can show unprovability of $(x+y)\dotminus x=y$ in the corresponding
arithmetical theory. Verifying whether $\N_\8$ and $\N_0$ are
models of $\IEx$ or not is out of the scope of what we intend to do here.
The important point we want to make is that our results on $\BAc$
strongly rely on the fact that $\BAc$ is capable of proving $\U$,
which in turn paves the way for eliminating cut-off subtraction
from the language. The alternative axiomatizations might look natural
when thinking about cut-off subtraction in the context of classical
or intuitionistic arithmetical theories, but not strong enough in the
context of basic arithmetic.

And lastly, the realizability technique used in \cite{S05} to find a bound
on the provably total functions of $\BA$ leads to no satisfactory result. To
show this, we define the notion of $D$-bounded recursive realizability, which
is a generalization of the notion defined in Definition 3.4 of \cite{S05}.
It is worth mentioning that the main results of \cite{S05}
are already disproved in \cite{AKS}, by presenting an explicit sequent
which is a counterexample to the soundness of $\BAw$ with respect to the
$D$-bounded recursive realizability, when $D(n,m)=n^m+m$.
\begin{definition}
\label{3dddddddddd}
Let $\varphi_n$ be the $n$-th partial recursive function, $\pi_1$ and
$\pi_2$ be the projections of the Cantor pairing function
$\langle x,y\rangle=\frac{1}{2}(x+y)(x+y+1)+y$, and
$D:\nn^2\to\nn$ be a primitive recursive function. For a sequence
$\x=(x_1,\dots,x_m)$, $\varphi_n(\x)$ is understood as
$\varphi_n(\langle x_1,\langle x_2,\dots,
\langle x_{m-1},x_m\rangle\rangle\rangle)$. Let
$$\B(n)\equiv\varphi_{\pi_1(n)}(x)\downarrow\ \land\
\all x(\varphi_{\pi_1(n)}(x)\le D(x,\pi_2(n)))\text,$$
and define $x\:\rd A$ by induction on complexity of a formula $A$:
\begin{itemize}
\item $x\:\rd A\equiv A$, for prime $ A$.
\item $x\:\rd (B \land C) \equiv (\pi_1(x) \rd B)\land (\pi_2(x) \rd C)$.
\item $x\:\rd (B \lor C) \equiv (\pi_1(x)=0 \land \pi_2(x)\: \rd B)
\lor (\pi_1(x)\ne 0 \land \pi_2(x)\: \rd C)$.
\item $x\:\rd \:\ex y B(y)\equiv \pi_2(x)\: \rd B(\pi_1(x))$.
\item $x\:\rd \:\all \z( B(\z)\ra C(\z)) \equiv
\B(x)\land \all y\z\,(y\:\rd B(\z)\ra\varphi_{\pi_1(x)}(y,\z)\:\rd C(\z))
\land\all\z\,(B(\z)\ra C(\z))$.
\end{itemize}

For a sequent $A\Ra B$, $x\:\rd( A \Ra B)\equiv \B(x) \land \all y\z\,(y\:\rd
A\ra \varphi_{\pi_1(x)}(y,\z)\:\rd B)\land ( A \ra B)$, where
$\z=(z_1,\dots,z_n)$ is the sequence of all free
variables in $ A \Ra B$ in the order of appearance.
\end{definition}
The following theorem shows that $\BAw$ is not sound with respect to the
$D$-bounded recursive realizability for any primitive recursive function $D$.
Then we may conclude that this type of realizabiliy is not suitable for
$\BA$ by Lemma \ref{baw-vs-ba} and the fact that for every sequent $A\Ra B$,
$\nn\models\ex n(n\:\rd(A\Ra B))$ iff
$\nn\models\ex n(n\:\rd(\top\Ra\all\x(A\ra B)))$.

\begin{theorem}\label{3dcr555}
For every primitive recursive function $D:\nn^2\to\nn$, there
exists a sequent $A\Ra B$ such that $\BA\vdash A\Ra B$, but
$\nn\not\models\ex n(n\:\rd (A\Ra B))$.
\end{theorem}
\begin{proof}
Let $h(n,m)$ be the following primitive recursive function:
$$h(n,m):=\max_{i\le n,j\le m}\{D(i,j)\}+n+m+1\text.$$
It is easy to see that for every $m,n\in\nn$, $D(n,m)<h(n,m)$.
By Theorem \ref{3t11} there is a $\Ex$ formula $A(x,y)$ that
defines the function $g(n)=h(\langle0,n\rangle,n)$ and also
$\BA+\U\vdash \top\Ra \ex y A(x,y)$, hence by Corollary \ref{bafaith}
$$\BA\vdash \all xyz(x+z=y+z\ra x=y)\Ra \all x(\top \ra \ex yA(x,y))\text.$$
Let the constant zero function be the $s$-th partial recursive function,
then we have $\nn\models t \:\rd \all xyz(x+z=y+z\ra x=y)$ for
$t=\langle s,0\rangle$.
Suppose there exists a natural number $n$ such that
$\nn\models n\:\rd(\all x,y,z(x+z=y+z\ra x=y)\Ra\all x(\top\ra\ex yA(x,y)))$.
Then, by definition of $D$-bounded recursive realizability, we have
\begin{enumerate}
\item $\nn\models \B(n)\land \all a(a\:\rd \all xyz(x+z=y+z\ra x=y)
\ra \varphi_{\pi_1(n)}(a)\:\rd \all x(\top \ra \ex yA(x,y)))$,
\item then $\nn\models u\:\rd \all x(\top \ra \ex yA(x,y))$ for
$u=\varphi_{\pi_1(n)}(t)$,
\item so $\nn\models \B(u)\land
\all ab(a\:\rd \top\ra \varphi_{\pi_1(u)}(a,b)\:\rd \ex yA(b,y))$,
\item hence $\nn\models \B(u)\land
\all b(\varphi_{\pi_1(u)}(0,b)\:\rd \ex yA(b,y))$.
\end{enumerate}
Note that by definition of the realizability, we have
$\nn\models \all b (\pi_2\varphi_{\pi_1(u)}(0,b)\:\rd
A(b,\pi_1\varphi_{\pi_1(u)}(0,b))))$, so
$\nn\models \all b A(b,\pi_1\varphi_{\pi_1(u)}(0,b))$ and hence
$g(b)=\pi_1\varphi_{\pi_1(u)}(0,b)$. This implies that
$g(b)\le \varphi_{\pi_1(u)}(0,b)$. Since $\B(u)$ is true, for all $b$ we have
$$g(b)\le \varphi_{\pi_1(u)}(0,b)\le
D(\langle0,b\rangle,\pi_2(u))<h(\langle0,b\rangle,\pi_2(u))$$
Let $b=\pi_2(u)$, then
$$g(\pi_2(u))=h(\langle0,\pi_2(u)\rangle,\pi_2(u))
<h(\langle0,\pi_2(u)\rangle,\pi_2(u))\text,$$
which leads to a contradiction. Hence our assumption is false.
\end{proof}

Our final remark in this paper is a proposal for axiomatization of $\BA$. 
It is a relatively standard tradition that one axiomatizes an arithmetical
theory by the usual Peano axioms. This tradition is applied, for instance,
to the intuitionistic arithmetical theory, well-known as Heyting arithmetic,
and also to the theory based on basic predicate logic, called basic arithmetic
\cite{R98}. We suggest the axiom $\U$ to be added to the list of axioms
and rules of $\BA$. Our motive is twofold. On one hand, as we have seen
in the previous sections, $\BA + \U$ is an arithmetical theory stronger than
$\BA$ with remarkable mathematical properties, and on the other hand, it is
still a constructive theory in the sense of \cite{R98}. By Lemma
\ref{U-equiv}, we can also add other equivalent axioms instead of $\U$.
One of them, $S(x+y)=y\Ra\bot$, can in fact replace the axiom $Sx=0\Ra\bot$
of $\BA$, instead of being added. Just note that $x+0=x$ is already an
axiom of $\BA$, and therefore $Sx=0\Ra\bot$ is equivalent to the instance
$S(x+0)=0\Ra\bot$ of the new axiom. The same remarks can be
made in relation with the theories $\GA$ and $\GA+\U$. While there seems to
be no theory widely referred to as ``geometric arithmetic" in the literature,
we suggest that considering the additional axiom $\U$ is a more suitable
choice than just taking the Peano axioms over geometric logic.

\section{Appendix}
\label{2d6.0}
\subsection*{The System $\LK$}
The system $\LK$ (see 1.2.2 and 2.3.2 of \cite{B98}) is a sequent calculus
formed by the following axioms and rules: (Here, $\Delta$, $\Delta'$
and $\Delta''$ are finite lists of formulas)
\subsubsection*{Structural Axiom}
\begin{enumerate}[]
\item
\AXC{}\LeftLabel{(Ax)}\UIC{$A\Ra A$}\DP
\end{enumerate}
\subsubsection*{Structural Rules:}
\begin{enumerate}[]
\item
\AXC{$\Delta,A,B,\Delta'\Ra\Delta''$}\LeftLabel{(Ex$\Ra$)}
\UIC{$\Delta,B,A,\Delta'\Ra\Delta''$}\DP
\item
\AXC{$\Delta\Ra\Delta',A,B,\Delta''$}\LeftLabel{($\Ra$Ex)}
\UIC{$\Delta\Ra\Delta',B,A,\Delta''$}\DP
\item
\AXC{$\Delta\Ra\Delta'$}\LeftLabel{(W$\Ra$)}
\UIC{$\Delta,A\Ra\Delta'$}\DP
\item
\AXC{$\Delta\Ra\Delta'$}\LeftLabel{($\Ra$W)}
\UIC{$\Delta\Ra A,\Delta'$}\DP
\item
\AXC{$\Delta,A,A\Ra\Delta'$}\LeftLabel{(C$\Ra$)}
\UIC{$\Delta,A\Ra\Delta'$}\DP
\item
\AXC{$\Delta\Ra A,A,\Delta'$}\LeftLabel{($\Ra$C)}
\UIC{$\Delta\Ra A,\Delta'$}\DP
\item
\AXC{$\Delta\Ra A,\Delta'$}\AXC{$\Delta,A\Ra\Delta'$}
\LeftLabel{(Cut)}\BIC{$\Delta\Ra\Delta'$}\DP
\end{enumerate}
\subsubsection*{Logical Axioms:}
\begin{enumerate}[]
\item
\AXC{}\LeftLabel{($\Ra\top$)}\UIC{$\quad\Ra\top$}\DP
\item
\AXC{}\LeftLabel{($\bot\Ra$)}\UIC{$\bot\Ra\quad$}\DP
\end{enumerate}
\subsubsection*{Logical Rules:}
\begin{enumerate}[]
\item
\AXC{$\Delta\Ra A,\Delta'$}\LeftLabel{($\neg\Ra$)}
\UIC{$\Delta,\neg A\Ra\Delta'$}\DP
\item
\AXC{$\Delta,A\Ra\Delta'$}\LeftLabel{($\Ra\neg$)}
\UIC{$\Delta\Ra\neg A,\Delta'$}\DP
\item
\AXC{$\Delta,A,B\Ra\Delta'$}\LeftLabel{($\land\Ra$)}
\UIC{$\Delta,A\land B\Ra\Delta'$}\DP
\item
\AXC{$\Delta\Ra A,\Delta'$}\AXC{$\Delta\Ra B,\Delta'$}
\LeftLabel{($\Ra\land$)}
\BIC{$\Delta\Ra A\land B,\Delta'$}\DP
\item
\AXC{$\Delta,A\Ra\Delta'$}\AXC{$\Delta,B\Ra\Delta'$}
\LeftLabel{($\lor\Ra$)}
\BIC{$\Delta,A\lor B\Ra\Delta'$}\DP
\item
\AXC{$\Delta\Ra A,B,\Delta'$}\LeftLabel{($\Ra\lor$)}
\UIC{$\Delta\Ra A\lor B,\Delta'$}\DP
\item
\AXC{$\Delta\Ra A,\Delta'$}\AXC{$\Delta,B\Ra\Delta'$}
\LeftLabel{($\ra\Ra$)}\BIC{$\Delta,A\ra B\Ra\Delta'$}\DP
\item
\AXC{$\Delta,A\Ra B,\Delta'$}\LeftLabel{($\Ra\ra$)}
\UIC{$\Delta\Ra A\ra B,\Delta'$}\DP
\item
\AXC{$\Delta,A[x/y]\Ra\Delta'$}\LeftLabel{($\ex\Ra$)}
\UIC{$\Delta,\ex x A\Ra\Delta'$}\DP,
where $y$ is substitutable for $x$ in $A$,
and $y$ is not free in $\Delta,\Delta'$
\item
\AXC{$\Delta\Ra A[x/t],\Delta'$}\LeftLabel{($\Ra\ex$)}
\UIC{$\Delta\Ra\ex x A,\Delta'$}\DP,
where $t$ is substitutable for $x$ in $A$
\item
\AXC{$\Delta,A[x/t]\Ra\Delta'$}\LeftLabel{($\all\Ra$)}
\UIC{$\Delta,\all x A\Ra\Delta'$}\DP,
where $t$ is substitutable for $x$ in $A$
\item
\AXC{$\Delta\Ra A[x/y],\Delta'$}\LeftLabel{($\Ra\all$)}
\UIC{$\Delta\Ra\all x A,\Delta'$}\DP,
where $y$ is substitutable for $x$ in $A$,
and $y$ is not free in $\Delta,\Delta'$
\end{enumerate}
In each rule, the formulas appearing in $\Delta$, $\Delta'$ or $\Delta''$,
are called \emph{context}. Other formulas which appear in the
upper sequents are called \emph{active formulas}, and those which
appear in the lower sequents are called \emph{principal formulas}.
The active formula of the \emph {cut rule} (Cut) is called the
\emph{cut formula}.
\subsection*{Arithmetical Theories Based on $\LK$}
Arithmetical theories over $\LK$ are formalized in the language of
arithmetic, with the additional axioms of logic with equality
(see 2.3.3 of \cite{B98}), and axioms and rules specific to arithmetic
(see 2.4.6 of \cite{B98}):
\subsubsection*{Equality Axioms:}
\begin{enumerate}[]
\item
\AXC{}\LeftLabel{(=-ref)}\UIC{$\quad\Ra s=s$}\DP
\item
\AXC{}\LeftLabel{(=-eqv)}\UIC{$s=t,s'=t',s=s'\Ra t=t'$}\DP
\item
\AXC{}\LeftLabel{(S-fnc)}\UIC{$s=t\Ra Ss=St$}\DP
\item
\AXC{}\LeftLabel{(+-fnc)}\UIC{$s=t,s'=t'\Ra s+t=s'+t'$}\DP
\item
\AXC{}\LeftLabel{($\cdot$-fnc)}\UIC{$s=t,s'=t'\Ra s\cdot t=s'\cdot t'$}\DP
\item
\AXC{}\LeftLabel{($<$-rel)}\UIC{$s=t,s'=t',s<s'\Ra t<t'$}\DP
\end{enumerate}
\subsubsection*{Arithmetical Axioms:}
\begin{enumerate}[]
\item
\AXC{}\LeftLabel{(S-pos)}\UIC{$Ss=0\Ra\quad$}\DP
\item
\AXC{}\LeftLabel{(S-inj)}\UIC{$Ss=St\Ra s=t$}\DP
\item
\AXC{}\LeftLabel{(+-0)}\UIC{$\quad\Ra s+0=s$}\DP
\item
\AXC{}\LeftLabel{(+-S)}\UIC{$\quad\Ra s+St=S(s+t)$}\DP
\item
\AXC{}\LeftLabel{($\cdot$-0)}\UIC{$\quad\Ra s\cdot 0=0$}\DP
\item
\AXC{}\LeftLabel{($\cdot$-S)}\UIC{$\quad\Ra s\cdot St=s\cdot t+s$}\DP
\end{enumerate}
\subsubsection*{Arithmetical Rule:}
\begin{enumerate}[]
\item
\AXC{$\Delta,A\Ra A[x/Sx],\Delta'$}\LeftLabel{(ind)}
\UIC{$\Delta,A[x/0]\Ra A[x/t],\Delta'$}\DP,
where $t$ is substitutable for $x$ in $A$,
and $x$ is not free in $\Delta,\Delta'$
\end{enumerate}
The formula $A$ in the \emph{induction rule} (ind) is called the
\emph{induction formula}. If the induction formula is restricted to a class
$\C$ of formulas, the rule is called the \emph{$\C$-induction rule}.

$ $

%\textbf{Acknowledgment.} We thank S. M. Mojtahedi for his helpful comments in
%our discussions and W. Ruitenburg for his comments, in particular his
%suggestion to us looking at the cut-off subtraction over $\BA$.

\end{document}